\RequirePackage{fix-cm}
\documentclass[smallextended]{svjour3}       
\smartqed  
\usepackage{amsmath,amssymb,amsfonts}
\usepackage{accents}
\usepackage[mathscr]{eucal}
\usepackage[titletoc,title]{appendix}

\usepackage{amscd,psfrag}
\usepackage{yhmath}

\usepackage{epstopdf}
\usepackage{chngcntr}
\usepackage{indentfirst}	
\usepackage[normalem]{ulem}

\newcommand{\R}{\mathbb{R}}

\newcommand{\na}{\nabla}
\newcommand{\two}{{\rm II}}
\newcommand{\emb}{\hookrightarrow}
\newcommand{\ta}{{\rm tan}}
\newcommand{\nor}{{\rm nor}}
\newcommand{\np}{\na^\perp}
\newcommand{\itm}{\iota^*T\widetilde{M}}
\newcommand{\nae}{\nabla^E}
\newcommand{\ind}{{\rm Ind}}
\newcommand{\id}{\text{Id}}
\newcommand{\gtm}{\Gamma(TM)}
\newcommand{\gtmbar}{\Gamma(T\widetilde{M})}
\newcommand{\glnk}{\mathfrak{gl}(n+k;\R)}
\newcommand{\p}{\partial}
\newcommand{\frako}{\mathfrak{o}}

\newcommand{\onu}{O(\nu, n-\nu)}
\newcommand{\bigonk}{O(\nu+\tau, (n+k)-(\nu+\tau))}
\newcommand{\littleonk}{\frako(\nu+\tau, (n+k)-(\nu+\tau))}
\newcommand{\loc}{{\rm loc}}
\newcommand{\fg}{\mathfrak{g}}
\newcommand{\lxz}{\Lambda^{(x,Z)}}
\newcommand{\dxz}{\mathcal{D}^{(x,Z)}}
\newcommand{\mm}{\mathfrak{m}}
\newcommand{\tth}{\underaccent{\widetilde}{\Theta}}

\newcommand{\bg}{\tilde{g}_0}
\newcommand{\one}{{\rm I}}
\newcommand{\pb}{\widetilde{\p}}

\newcommand{\nn}{\mathfrak{n}}

\newcommand{\weak}{\rightharpoonup}
\newcommand{\e}{\epsilon}
\newcommand{\T}{{\mathcal{T}}}
\newcommand{\G}{\Gamma}
\newcommand{\map}{{\rightarrow}}

\newcommand{\C}{\mathbb{C}}
\newcommand{\dd}{{\rm d}}

\newcommand{\hg}{{\hat{G}}}
\newcommand{\matij}{{{\rm Mat} (I \times J; \C)}}
\newcommand{\K}{\mathcal{K}}

\newcommand{\dvg}{{\,\dd V_g}}
\newcommand{\diff}{{{\rm Diff}}}
\newcommand{\oo}{{\mathcal{O}}}
\newcommand{\gzero}{\tilde{g}_0}
\newcommand{\h}{\mathfrak{h}}
\newcommand{\W}{\mathcal{W}}

\makeatletter
\pdfpageheight\paperheight
\pdfpagewidth\paperwidth

\counterwithin{figure}{section}

\newtheorem{convention}[definition]{Convention}
\numberwithin{equation}{section}
\numberwithin{theorem}{section}
\numberwithin{lemma}{section}
\numberwithin{proposition}{section}
\numberwithin{corollary}{section}
\numberwithin{definition}{section}
\numberwithin{remark}{section}

\journalname{Archive for Rational Mechanics and Analysis\, }
\begin{document}

\title{Weak Continuity of the Cartan Structural System and Compensated Compactness on
Semi-Riemannian Manifolds with Lower Regularity\thanks{The research of Gui-Qiang G. Chen was supported in part by
the UK Engineering and Physical Sciences Research Council Award
EP/L015811/1 and the Royal Society--Wolfson Research Merit Award WM090014 (UK). The research of Siran Li was supported in part by
the UK Engineering and Physical Sciences Research Council Award EP/E035027/1.}
}

\titlerunning{Weak Continuity of Cartan Structural System $\&$ Compensated Compactness}        

\author{Gui-Qiang G. Chen \and  Siran Li}

\institute{G.-Q. G. Chen\at
             Mathematical Institute, University of Oxford, Oxford, OX2 6GG, UK\\
              \email{chengq@maths.ox.ac.uk}           
           \and
           S. Li \at
              New York University--Shanghai, Office 1146, 1555 Century Avenue, Pudong District, Shanghai 200122, China, and NYU--ECNU Institute of Mathematical Sciences,
Room 340, Geography Building, 3663 North Zhongshan Road, Shanghai 200062, China\\\email{sl4025@nyu.edu}\\
\emph{Present address:}  School of Mathematical Sciences, Shanghai Jiao Tong University, No.~6 Science Buildings,
800 Dongchuan Road, Minhang District, Shanghai 200240, China}

\date{Received: July 9, 2020 / Accepted: March 19, 2021}
\maketitle

\begin{abstract}
We are concerned with the global weak continuity of the Cartan structural
system --- or equivalently, the Gauss--Codazzi--Ricci system --- on semi-Riemannian manifolds
with lower regularity.
For this purpose, we first formulate and prove a geometric compensated compactness theorem
on vector bundles over semi-Riemannian manifolds with lower
regularity (Theorem \ref{thm: generalized quadratic theorem on manifolds-b}),
extending the classical quadratic theorem of compensated compactness. We then deduce the $L^p$
weak continuity of the Cartan structural system for $p>2$:
For a family $\{\W_\varepsilon\}$ of connection $1$-forms on a semi-Riemannian manifold $(M,g)$,
if $\{\W_\varepsilon\}$ is uniformly bounded in $L^p$
and satisfies the Cartan structural system,
then any weak $L^p$
limit of $\{\W_\varepsilon\}$ is also a solution of the Cartan structural system.
Moreover, it is proved that isometric immersions of semi-Riemannian manifolds into semi-Euclidean spaces can be constructed
from the weak solutions of the Cartan structural system or
the Gauss--Codazzi--Ricci system (Theorem \ref{theorem_main theorem, isometric immersions and GCR}),
which leads to the $L^p$ weak continuity of the Gauss--Codazzi--Ricci system on semi-Riemannian manifolds.
As further applications, the weak continuity of Einstein's constraint equations, general immersed hypersurfaces,
and the quasilinear wave equations is also established.


\end{abstract}

\section{\, Introduction}
\label{intro}

We are concerned with isometric immersions of semi-Riemannian manifolds with arbitrary signature
into semi-Euclidean spaces.
We establish the weak continuity of two fundamental systems of nonlinear partial
differential equations (PDEs): the Cartan structural system and the Gauss--Codazzi--Ricci (GCR) system,
which constitute the compatibility equations for the existence of isometric immersions.
	
The isometric immersion problem has been of fundamental importance in the development of modern differential geometry.
It has led to various new techniques and ideas in nonlinear PDEs, nonlinear analysis,
and geometric analysis ({\it cf.} \cite{BGY,gunther,HanHon06,Yau} and the references cited therein).
On the other hand, it has wide applications.
For example, in theoretical physics, the manners in which our $4$-dimensional space-time is immersed in the ambient universe correspond
to different cosmological models ({\it cf.} Mars--Senovilla \cite{mars,mars2}),
and the isometric immersion of round spheres into warped product manifolds is central
to recent versions of quasi-local mass ({\it cf.} Guan--Lu \cite{gl} and Wang--Yau \cite{wy}).
Moreover, the isometric immersions of semi-Riemannian manifolds {\em with lower regularity} are fundamental in many scientific areas.
For example, such immersions
arise in the thin-shell model
for gravitational source and the junction condition for gluing disjoint space-times; see \cite{israel,clarke3,geroch} for the details.
	
In the classical work \cite{Nas56}, Nash established the existence of isometric embeddings of Riemannian manifolds
with $C^k$ metrics, $k \geq 3$,
into the Euclidean spaces of high dimensions.
The analogous problem for semi-Riemannian manifolds ({\it i.e.}, the metrics are not necessarily positive-definite)
is posed as a natural extension.
More importantly, the isometric immersion problem of semi-Riemannian manifolds is fundamental in general relativity
and Lorentzian geometry.
Clarke \cite{clarke2} proved the existence theorem of isometric embeddings of $C^k$ semi-Riemannian manifolds into
semi-Euclidean spaces, under additional hypotheses on the signature.
Despite these general existence theorems,
the analysis for isometric immersions of semi-Riemannian manifolds appears more challenging than its Riemannian analog.
In particular, the Laplace--Beltrami operator is no longer elliptic, thus precludes the standard elliptic PDE machineries.
See Goenner \cite{Goenner}, Greene \cite{Greene}, and the references cited therein for the earlier rigorous mathematical analysis
on isometric immersions of semi-Riemannian manifolds.
	
Motivated by both mathematical and physical importance discussed above, in this paper,
we study the isometric immersions of semi-Riemannian manifolds with lower regularity.
One of the fundamental tools for investigating the isometric immersions is the GCR system
({\it cf.} \cite{BGY,csw1,csw2,Goenner,Janet,lefloch2}), which describes the geometry of the ambient space
in terms of the geometry of the tangential and normal directions of the immersed submanifold.
We are interested in the global weak continuity of the GCR system,
as well as the global weak rigidity of the corresponding isometric immersions and curvatures.
	
The analysis of the GCR system encompasses several challenges,
primarily because they do not have a fixed type --- elliptic, parabolic, or hyperbolic --- in general.
Even in the Riemannian case, when the immersed manifold has dimension higher than $3$,
it is proved by Bryant--Griffith--Yang \cite{BGY} that the GCR system has no definite type.
The novel observation by Chen--Slemrod--Wang in \cite{csw1,csw2} (also see \cite{chenli})
shows that the GCR system for Riemannian manifolds possesses an intrinsic {\em div-curl structure},
so that the compensated compactness techniques
for nonlinear analysis can be applied, which is independent of the types of the system.

In order to employ the compensated compactness techniques  in semi-Riemannian settings, however, we meet with further complications.
First, the effective  proofs of the div-curl lemma rely essentially on the ellipticity of the Laplace--Beltrami
operator; {\it cf.} Evans \cite{evans}, Robbin--Rogers--Temple \cite{rrt}, Kozono--Yanagisawa \cite{kozono},
Chen--Li \cite{chenli}, and the references cited therein.
This does not hold for semi-Riemannian manifolds.
Moreover, the non-trivial signatures of the semi-Riemannian metrics
make it difficult to identify  the div-curl structure globally.

To overcome the new complications, we further exploit the geometry of isometric immersions of semi-Riemannian submanifolds.
Rather than tackling the GCR system directly, we first establish the weak continuity of the Cartan structural system.
This is proved to be equivalent to the GCR system, even for the semi-Riemannian manifolds with lower regularity in $W^{2,p}$.
The Cartan structural system possesses a natural quadratic structure.
For this purpose, we first establish a global, intrinsic
compensated compactness result (Theorem~\ref{thm: generalized quadratic theorem on manifolds-b})
in the setting of vector bundles over semi-Riemannian manifolds, and
then apply it to give a rigorous proof of the weak continuity of the Cartan structural system.
We emphasize the {\em global} and {\em intrinsic} nature of these results,
in the sense that their formulations are independent of local coordinate systems.

The compensated compactness techniques have been developed in the study of nonlinear PDEs in the Euclidean
space $\R^d$, especially for nonlinear conservation laws such as the Euler equations in fluid mechanics;
see \cite{Chen05,daf,evans} and the references therein.
One of the major results in the theory of compensated compactness is the {\em quadratic theorem} in $\R^d$
(see Murat \cite{Murat} and Tartar \cite{tartar}).
For our purpose, we establish a generalized quadratic theorem that is of global and intrinsic nature
on vector bundles.
Our crucial observation is that the first-order differential constraints in the quadratic theorem on $\R^d$
can be replaced by more general assumptions on the principal symbol of the associated differential operators,
while the principal symbol is diffeomorphism-invariant on manifolds.
This leads to an intrinsic formulation of the quadratic theorem on vector bundles over semi-Riemannian manifolds.
	
Other generalizations of the quadratic theorem were established in the literature.
Mi\u{s}ur--Mitrovi\'{c} in \cite{MMit} studied
the weak convergence of quadratic expressions $\sum_{i,j=1}^N q_{ij} u_\varepsilon^i v_\varepsilon^j$,
where $\{u_\varepsilon\}$ and $\{v_\varepsilon\}$ are weakly convergent
in $L^p(\R^d; \R^N)$ and $L^{p'}(\R^d; \R^N)$, respectively, for $\frac{1}{p}+\frac{1}{p'}\leq 1$.
For this, coefficients $q_{ij}, i,j=1,\cdots, N$, are allowed to depend on $x\in \R^d$,
the conditions involve fractional derivatives,
and  the idea of $H$-distributions is used in the proof; also see \S 3 in Mi\u{s}ur \cite{Miu}.
In contrast, our generalized quadratic theorem is geometric and global in nature,
which serves naturally for our purpose to establish the weak continuity of both the Cartan structural system
and the GCR system.

The results and techniques established in this paper have applications to semi-Riemannian geometry,
from the perspectives of both mathematics and physics.
For example, we deduce the weak rigidity of isometric immersions of semi-Riemannian manifolds by using the weak continuity of
the Cartan structural system or the GCR system.
The realizability of isometric immersions of semi-Riemannian manifolds with lower regularity
from the weak solutions of the Cartan structural system or the GCR system (Theorem \ref{theorem_main theorem, isometric immersions and GCR})
is proved along the way.
In addition, we demonstrate the weak continuity properties of Einstein's constraint equations,
quasilinear wave equations, and degenerate hypersurfaces in space-time.

We emphasize that, in this paper, we are concerned mainly with semi-Riemannian manifolds $(M,g)$ with lower regularity,
which means that $M$ is parametrized by $W^{2,p}_\loc$ maps, or that metric $g$ is in $W^{1,p}_\loc \cap L^\infty_\loc$.
The weak continuity of the GCR system and the Cartan structural system is established in such regularity classes
with $p>2$, regardless of the dimension of $M$.
In particular, when $\dim\,M \geq 3$, these weak continuity results cannot be deduced from
the realization theorem of isometric immersions
from the GCR system, or equivalently the Cartan structural system,
since it can be proved so far only under the {\it stronger assumption}: $p>\dim\,M$.
This imposes considerable additional difficulties.
In fact, apart from the realization theorem (Theorem \ref{theorem_main theorem, isometric immersions and GCR}),
we will restrict ourselves only to $p>2$ (rather than $p> \dim\,M$) everywhere else throughout the paper.

We remark in passing that the $W^{2,p}$ continuity of the GCR and Cartan structural systems may also be established
via computing carefully in local coordinate systems, by utilizing the compensated compactness techniques
in the flat space $\R^d$ (see, \emph{e.g.}, Chen--Slemrod--Wang \cite{csw2} and  Robbin--Rogers--Temple \cite{rrt}).
However, the compensated compactness results established in the semi-Riemannian setting
in this paper not only provide a direct intrinsic proof of the $W^{2,p}$ continuity of the Cartan structural systems,
but also are of independent interest.
In particular, the global and intrinsic formulation
of Theorem~\ref{thm: generalized quadratic theorem on manifolds-b} contributes to the theory of compensated compactness
and its further applications.

The rest of this paper is organized as follows:
In $\S \ref{sec: geom prelim}$, we review the Cartan structural system and
the basics of the semi-Riemannian submanifold theory.
The bundle-theoretic perspectives are emphasized.
In $\S \ref{sec: generalised quadratic thm}$, we establish a global intrinsic compensated compactness theorem
on vector bundles over semi-Riemannian manifolds, which is also extended to locally compact Abelian groups.
Employing the results in $\S \ref{sec: generalised quadratic thm}$,
we deduce  the weak continuity of the Cartan structural system in $\S \ref{sec: weak continuity of cartan}$.
Next, in $\S \ref{sec: realisation}$, we solve the realization problem ({\it i.e.}, the construction of
isometric immersions from the GCR system, or equivalently the Cartan structural system) on simply-connected
semi-Riemannian manifolds with lower regularity.
Finally,  in $\S \ref{sec: further app}$, we discuss further applications of the theorems and techniques established in earlier sections.
In particular, we demonstrate the weak continuity of Einstein's constraint equations, quasilinear wave equations with
the null structure, and general hypersurfaces in space-time.
For completeness, the proofs of several semi-Riemannian geometric results and facts,
as well as the proof of Theorem 3.5 (the generalized quadratic theorem on locally compact Abelian groups),
are presented in Appendices A and B.
	
\section{\, The Cartan Structural System and Isometric Immersions of Semi-Riemannian Manifolds}\label{sec: geom prelim}

In this section, we discuss the Cartan structural system.
One of our motivations comes from the isometric immersion problem for semi-Riemannian manifolds:
the Cartan structural system is known to
be equivalent to the GCR system, since both systems are the classical compatibility equations for
the existence of isometric immersions.
The isometric immersion problem is an important topic in theoretical physics and differential geometry.
In particular, it is closely related to the definition of quasi-local mass
in space-time (see Brown--York \cite{by}, Wang--Yau \cite{wy},
and the references therein).

We first review the submanifold theory in semi-Riemannian geometry.
Then we discuss the derivation of the GCR system and the formulation of the Cartan structural system.
Our exposition follows essentially from O'Neill \cite{oneill};
nevertheless, several {\it ad hoc} constructions therein are clarified by using the language of vector bundles.

\subsection{\, Semi-Riemannian Submanifold Theory}\label{subsec: semi-riem geometry prelims}

Let $M$ be an $n$-dimensional manifold. It is said to be {\em semi-Riemannian} if there exists a symmetric, non-degenerate $2$-form field $g$
on the tangent bundle $TM$ with constant index. Then $g$ is known as a {\em semi-Riemannian metric}. The semi-Riemannian metric $g$
is {\em non-degenerate} on $M$ if, for each $x\in M$,
there exists no $v\in T_xM\setminus \{0\}$ such that $g(v,w)=0$ for every $w\in T_xM$.

The {\em index} of the semi-Riemannian metric $g$ on $T_xM$ is defined by
$$
\text{Ind}(g;T_xM) := \max\left\{\dim V\,: \,
  \begin{array}{ll}\text{$V \subset T_xM$ is a vector subspace}\\
                                              \text{and $g|_V$ is negative definite}
                                             \end{array} \right\}.
$$
Clearly, if $M$ is connected, then $\text{Ind}(g;T_xM)$ is constant for all $x\in M$,
which will be written as $\text{Ind}(g)$ in the sequel.
Employing the Gram--Schmidt process to a subset $U \subset M$,
we can find a local orthonormal basis $\{e_i\}_1^n\subset TU$ so that $g$ is diagonalized:
\begin{equation*}
g=\{g_{ij}\}=\delta_{ij}|g_{ij}|\epsilon^j\qquad \text{ for each } i,j\in\{1, 2,\ldots, n\},
\end{equation*}
where ${\epsilon}:=(\epsilon^1,\ldots,\epsilon^n)^\top\in \{-1,1\}^n$ is called the {\em signature} of
metric $g$.
As $g$ is non-degenerate, it has only non-zero entries on the diagonal so that $\text{Ind}(g)$ equals to the number
of ``$-1$'' in signature $\epsilon$.
For simplicity, from now on, the semi-Riemannian manifold $(M,g)$ is always taken to be connected, and  $\text{Ind}(g)$
is called the {\em index of $M$} for the fixed metric $g$.

Let $(\widetilde{M},\tilde{g})$ be a given semi-Riemannian manifold,
and let $M$ be a submanifold via the embedding $\iota: M \emb (\widetilde{M},\tilde{g})$,
{\em i.e.}, both $\iota:  M \emb (\widetilde{M},\tilde{g})$ and $\dd\iota: TM\rightarrow T\widetilde{M}$ are injective.
We say that $(M,\iota^*\tilde{g})$ is a {\em semi-Riemannian submanifold} of $(\widetilde{M},\tilde{g})$,
provided that $\iota^*\tilde{g}$ is non-degenerate on $M$, where $\iota^*\tilde{g}$ denotes the pullback of $\tilde{g}$
defined by
\begin{equation*}
(\iota^*{\tilde{g}})_x(v,w):=\tilde{g}_{\iota(x)}(\dd\iota (v), \dd\iota(w)) \qquad \text{ for each $x\in M$ and $v,w\in T_xM$}.
\end{equation*}

Before further development,
we introduce one notation: For any vector bundle $E$ over $M$,
we write $\G(E)$ for the space of sections of $E$, {\it i.e.},
$s: M \rightarrow E$ such that $\pi \circ s = {\rm id}_M$,
where $\pi: E \rightarrow M$ is the projection of bundle $E$ onto the base manifold.

Next, we consider $\iota^*T\widetilde{M}$, the vector bundle with base manifold $M$ and fiber $T_{\iota(x)}\widetilde{M}$ at each $x\in M$.
Then $\Gamma(\iota^*T\widetilde{M})$ consists of the vector fields in $T\widetilde{M}$ defined along $M$.
In particular,
\begin{equation}\label{eqn decomposition of the pullback bundle}
\itm|_{x}:=T_{\iota(x)}\widetilde{M}=\dd_x\iota(T_{x}M) \oplus  [\dd_x\iota(T_{x}M)]^\perp \cong T_{x}M \oplus [\dd_x\iota(T_{x}M)]^\perp,
\end{equation}
whenever $\iota$ is a {\em local immersion}, {\it i.e.}, $\dd\iota$ is injective in some neighborhood of $x\in M$.
Here the direct sum is taken with respect to the bilinear form $\tilde{g}$ on $T_{\iota(x)}\widetilde{M}$:
\begin{equation*}
\big[\dd_x\iota(T_xM)\big]^\perp := \big\{v \in T_{\iota(x)}\widetilde{M}\,:\,\tilde{g}_{\iota(x)}(v,\dd_x\iota (w))=0
\, \text{ for all } w\in T_{x}M\big\}.
\end{equation*}
Eq. \eqref{eqn decomposition of the pullback bundle} is a special case of Lemma 23
in \cite{oneill}, which is proved by a simple dimension-counting.
It holds only when $\iota^*\tilde{g}$ is non-degenerate, {\it i.e.}, $M$ is immersed into $(\widetilde{M}, \tilde{g})$ as a semi-Riemannian submanifold.
In this case, $TM$ and $\iota^*T\widetilde{M}$ are vector bundles over $M$ and $TM \subset \iota^*T\widetilde{M}$ respectively,
hence the quotient bundle is well-defined.

\begin{definition}
The normal bundle of the isometric immersion $\iota: M \emb \widetilde{M}$ is
\begin{equation*}
TM^\perp := \frac{\iota^*T\widetilde{M}}{TM}.
\end{equation*}
\end{definition}

In view of Eq. \eqref{eqn decomposition of the pullback bundle}, the fiber of $TM^\perp$ at $x\in M$ (written as $T_xM^\perp$) is
isomorphic to $[\dd_x\iota(T_xM)]^\perp$, so that the following isomorphism of vector spaces holds:
\begin{equation}\label{eqn decomposition of the pullback bundle II}
T_{\iota(x)}\widetilde{M} \cong T_xM \oplus T_xM^\perp.
\end{equation}
The canonical projections of $T_{\iota(x)}\widetilde{M}$ onto the first and second factors
are called the {\em tangential} and {\em normal projections}, denoted by
\begin{equation}\label{tan nor 1}
\ta: \itm|_x \rightarrow T_xM, \qquad \nor:\itm|_x\rightarrow T_xM^\perp.
\end{equation}
By naturality, they induce both the projections of vector fields:
\begin{equation}\label{tan nor 2}
\ta: \Gamma(\iota^*T\widetilde{M})\rightarrow \Gamma(TM),\qquad \nor: \Gamma(\iota^*T\widetilde{M})\rightarrow \Gamma(TM^\perp),
\end{equation}
and the projections of vector fields with Sobolev regularity:
\begin{align*}
&\ta: W^{k,p}(M;\itm) \rightarrow W^{k,p}(M;TM),\nonumber\\
& \nor: W^{k,p}(M;\itm) \rightarrow W^{k,p}(M;TM^\perp)
\end{align*}
for $p \in [1,\infty]$ and $k \in \mathbb{Z}$.

Moreover, for notational convenience, we introduce the following conventions:

\begin{convention}\label{convention on iota}
We write the tangential vector fields as $X,Y,Z,\ldots\in \Gamma(TM)$
and the normal vector fields as $\xi,\eta,\zeta,\ldots \in \Gamma(TM^\perp)$.
For a generic vector field not necessarily tangential or normal,
{\it i.e.}, an element in $\Gamma(T\widetilde{M})$ or $\Gamma(\iota^*T\widetilde{M})$,
we use letters $U,V,W$.
Finally, for a bundle $E$ different from $TM$, $TM^\perp$,
and $\itm$, we write $\alpha,\beta,\ldots \in \Gamma(E)$.
\end{convention}

\begin{convention}
Given an isometric immersion $f: M \rightarrow \widetilde{M}$,
write $\{\partial_a\}$, $S$, $g$, $\na$, $R$, $\ldots$ for the geometric quantities on $M$,
and $\{\widetilde{\p}_a\}, \widetilde{S}, \bg, \widetilde{\na}, \widetilde{R}, \ldots$ for the corresponding
quantities on $\widetilde{M}$.
\end{convention}

With the orthogonal splitting of tangent and normal directions under isometric immersions,
we are ready to study the orthogonal splitting of connections.
Let $(\widetilde{M},\tilde{g})$ be a semi-Riemannian manifold, and let $\iota: M \emb \widetilde{M}$ be an immersed semi-Riemannian submanifold.
The Levi--Civita theorem says that there exists a unique affine
connection $\widetilde{\na}:\Gamma(T\widetilde{M})\times \Gamma(T\widetilde{M})\rightarrow \Gamma(T\widetilde{M})$
which is metric-compatible and torsion-free ({\it cf}. \cite{oneill}).
More precisely, the following conditions hold for any smooth function $\varphi:M\rightarrow\R$
and vector fields $U,V,W\in \Gamma(T\widetilde{M})$:

\begin{enumerate}
\item[(i)]
Affine: $\widetilde{\na}_{\varphi V} W = \varphi\widetilde{\na}_V W$ and $\widetilde{\na}_V (\varphi W) = V(\varphi)W + \varphi\widetilde{\na}_V W$;

\smallskip
\item[(ii)]
Compatible with metric: $U\tilde{g}(V,W) = \tilde{g}(\widetilde{\na}_U V,W)+\tilde{g}(V,\widetilde{\na}_U W)$;

\smallskip
\item[(iii)]
Torsion-free: $ \widetilde{\na}_V W - \widetilde{\na}_W V = [V,W]$.
\end{enumerate}

Recall that the connections can be pulled back by using the maps between topological manifolds (see {\it e.g.} \cite{steenrod}).
In particular, $\iota: M \emb \widetilde{M}$ induces the {\em pullback connection} $\iota^*\widetilde{\na}:\Gamma(TM)\times\Gamma(\itm)\rightarrow \Gamma(\itm)$
on the pullback bundle $\itm$, given by
\begin{equation*}
(\iota^*\widetilde{\na})_X(\iota^*\alpha) = \iota^*(\widetilde{\na}_{\dd\iota(X)}\alpha) \qquad \text{ for any $\alpha\in\Gamma(T\widetilde{M})$ and $X\in\Gamma(TM)$}.
\end{equation*}
Hence, for a vector field $V\in\gtmbar$ along $M$, {\it i.e.}, $V\in\Gamma(\itm)$, we have
\begin{equation}\label{eqn pullback connection}
(\iota^*\widetilde{\na})_X V = \iota^*(\widetilde{\na}_{\dd\iota(X)}\dd\iota (V))=\widetilde{\na}_{\dd\iota(X)}\dd\iota (V),
\end{equation}
where $\dd\iota(X)$ and $\dd\iota(V)$ can be viewed as the local extensions of $X\in\gtm$
and $V\in\Gamma(\iota^*T\widetilde{M})$ to the vector fields in $\gtmbar$.

For simplicity, we adopt the slight abuse of notations of systematically dropping the pullback
operator $\iota^\ast$  (see \cite{docarmo,oneill,Ten71}) when no confusion arises.
In effect, this amounts to viewing $M$ as a subset of $\widetilde{M}$,
and $\iota$ as the identity map from $M$ to its image.

\begin{convention}
Let $\iota: (M,g)\emb (\widetilde{M},\tilde{g})$ be an isometric immersion of semi-Riemannian submanifolds.
Then $(\itm, \iota^*\widetilde{\na})$ is replaced by $(T\widetilde{M}, \widetilde{\na})$.
\end{convention}

With the above preparations, we now consider the following decomposition of connections:
\begin{equation*}
\widetilde{\na}_X V=\ta [\widetilde{\na}_X (\ta V)] + \ta [\widetilde{\na}_X (\nor V)]
        + \nor [\widetilde{\na}_X (\ta V)] + \nor [\widetilde{\na}_X (\nor V)]
\end{equation*}
for any $X\in \Gamma(TM)$ and $V\in\Gamma(T\widetilde{M})$,
where both projections $tan$ and $nor$ are as in Eq. \eqref{tan nor 2}.

\begin{definition}\label{def: nabla, II, nabla perp}
Given an isometric immersion $\iota: (M,g) \emb (\widetilde{M}, \tilde{g})$, the tangential connection $\na:\Gamma (TM)\times \Gamma(TM)\rightarrow \Gamma(TM)$,
the second fundamental form ${\rm II}: \Gamma(TM)\times\Gamma(TM)\rightarrow \Gamma(TM^\perp)$,
the shape operator {\rm (}associated to ${\rm II}${\rm )} $S:\Gamma(TM)\times \Gamma(TM^\perp)\rightarrow \Gamma(TM)$,
and the normal connection $\na^\perp:\Gamma(TM)\times\Gamma(TM^\perp)\rightarrow \Gamma(TM^\perp)$ are defined as
\begin{equation*}
\begin{cases}
\na_XY:=\tan \,\widetilde{\na}_XY, \qquad {\rm II}(X,Y):={\rm nor}\, \widetilde{\na}_XY,\\
S_\xi X:=-{\rm tan} \,\widetilde{\na}_X\xi, \qquad \np_X\xi:= {\rm nor} \,\widetilde{\na}_X\xi,
\end{cases}
\end{equation*}
for $X,Y\in\Gamma(TM)$ and $\xi\in\Gamma(TM^\perp)$.
\end{definition}
We note that $\na$ is the Levi--Civita connection on $(M, \iota^*\tilde{g})$,
whenever $\widetilde{\na}$ is the Levi--Civita connection on $\widetilde{M}$.
Moreover, $\two$ and $S$ are related by
\begin{equation*}
\tilde{g}(\two(X,Y), \xi)=\tilde{g}(S_\xi X, Y).
\end{equation*}
In addition, $\two$ is symmetric (equivalently, $S_\xi$ is self-adjoint) on $\Gamma(TM)$.
The Riemann curvature tensor will be introduced in \S 2.2 below.

Finally, with $\mathfrak{gl}(n; \R)$ denoting the space of $n\times n$ real matrices,
we define the {\em semi-orthogonal group of $\R^n_\nu$} as
\begin{equation*}
O(\nu, n-\nu):=\Big\{B\in \mathfrak{gl}(n;\R): B(v,w)=
{\e_{n,\nu}} v\cdot w \text{ for all } v,w\in T\R^n_\nu\Big\},
\end{equation*}
with the signature matrix given by
\begin{equation}\label{signature matrix}
{\e_{n,\nu}} = \text{diag} (\underbrace{-1,\cdots,-1}_{\text{$\nu$ times}}, \underbrace{1,\cdots,1}_{\text{$n-\nu$ times}}).
\end{equation}
In other words, $O(\nu, n-\nu)$ is the group of linear isometries from $\R^n_\nu$ to itself.
Here and in the sequel, $\R^n_\nu$ denotes the {\em semi-Euclidean} space, {\it i.e.}, manifold $\R^n$ equipped with
metric
${\e_{n,\nu}}$. Likewise, the Lie group $O(\tau, k-\tau)$ has the signature matrix:
\begin{equation*}
{\e_{k,\tau}} = \text{diag} (\underbrace{-1,\cdots,-1}_{\text{$\tau$ times}}, \underbrace{1,\cdots,1}_{\text{$k-\tau$ times}}).
\end{equation*}
We also denote by $\R^{n+k}_{\nu + \tau}$ the semi-Euclidean space $\R^{n+k}$ with the metric:
\begin{equation*}
\tilde{g}_0 =
{\e_{n,\nu}}  \oplus
{\e_{k,\tau}}.
\end{equation*}
The direct sum is understood as the block sum of matrices.
Furthermore, we denote the Lie algebra of $O(n, n-\nu)$ as $\mathfrak{o}(n, n-\nu)$.

\subsection{\, Gauss--Codazzi--Ricci System and Isometric Immersions}\label{subsec: GCR}

The isometric immersion problem can be stated as follows:
{\it Given a semi-Riemannian manifold $(M,g)$ and a target semi-Riemannian manifold $(\widetilde{M}, \tilde{g})$ of higher dimension,
seek an immersion $f: (M,g) \emb (\widetilde{M}, \tilde{g})$ such that  $f(M)$ is a semi-Riemannian submanifold of $\widetilde{M}$ with $f^*\tilde{g} = g$.}

\smallskip
A necessary compatibility condition for the existence of  an isometric immersion $f$
is that the Riemann curvature tensor of $\widetilde{M}$
should be splitted nicely in the tangential and normal directions, {\it i.e.}, in $TM$ and $TM^\perp$.
In what follows, we discuss the Riemann curvature on semi-Riemannian manifolds and derive the compatibility equations,
which are known as the GCR system. Again, for our purpose, we focus on the perspectives of vector bundles,
in comparison with \cite{oneill}.
One further convention is introduced for notational convenience:

\begin{convention}\label{convention of metrics}
In the rest of the paper, we write $\langle\cdot,\cdot\rangle$ for $\tilde{g}(\cdot,\cdot), g(\cdot,\cdot)$,
and any other semi-Riemannian metrics, unless further specified.
\end{convention}

Let $(M,g)$ be an $n$-dimensional semi-Riemannian manifold of index $\nu$,
and let $E$ be a vector bundle over $M$ with fibers $F\cong \R^k_\tau$,
the {\em semi-Euclidean space $\R^k$ with index $\tau$}.
Let $\nae$ be an affine connection on bundle $E$,
{\it i.e.}, a linear map
$$
\na^E: \Gamma(TM)\times\Gamma(E)\rightarrow\Gamma(E)
$$
satisfying $\nae_{\phi X}\alpha=\phi\nae_X\alpha$ and $\nae_X(\phi \alpha)=X(\phi)\alpha+\phi\nae_X\alpha$ for any $\phi:M\rightarrow \R$.
This can be compactly written as
\begin{equation*}
\nae(\phi\alpha)=\phi\nae\alpha+\dd\phi\otimes \alpha,
\end{equation*}
once we view $\nae: \Gamma(E)\rightarrow \Omega^1(E):=\Gamma(E\otimes T^*M)$, the space of differential $1$-forms on bundle $E$.
The {\em Riemann curvature on bundle $E$} is given by $R^E:\Gamma(TM)\times\Gamma(TM)\rightarrow\Gamma(\text{End} E)$ as
\begin{equation*}
R^E(X,Y):=[\nae_X,\nae_Y] - \nae_{[X,Y]},
\end{equation*}
where $\text{End} E$ is the {\it endomorphism bundle} on $E$.
That is, $\text{End} E$ is the vector bundle over $M$ with the typical fiber $\mathfrak{gl}(F)$,
the group of linear transforms from $F$ to itself.
Note that $R^E(X,Y,\alpha)\in\Gamma(E)$ for $\alpha\in \Gamma(E)$.
Also, $R^E$ is often written as the $(0,4)$--tensor:
\begin{equation*}
R^E(X,Y,\alpha,\beta):=\langle R^E(X,Y,\alpha),\beta\rangle_E  \qquad\mbox{for $X,Y\in \Gamma(TM)$ and $\alpha,\beta\in \Gamma(E)$},
\end{equation*}
where we write $\langle\cdot,\,\cdot\rangle_E$ to emphasize the bundle metric.

Now we may investigate the orthogonal splitting of the Riemann curvature along
the projections $tan$ and $nor$ (see $\S 2.1$).
Given an isometric immersion $f: (M,g)\rightarrow (\widetilde{M},\tilde{g})$, three vector bundles over $M$ are of interest:
$E=TM$, $TM^\perp$, and $f^*T\widetilde{M}$. We denote the last bundle by $T\widetilde{M}$ in light of Convention \ref{convention on iota}.
We also fix the notations:
\begin{equation*}
\begin{cases}
\na=\na^{TM},\qquad \widetilde{\na}=\na^{T\widetilde{M}}, \qquad \na^\perp=\na^{TM^\perp},\\
R=R^{TM}, \qquad \widetilde{R}=R^{T\widetilde{M}}, \qquad R^\perp=R^{TM^\perp},
\end{cases}
\end{equation*}
where $\na^{TM}$ denotes the Levi--Civita connection on $M$.

In what follows, we are concerned with the special case:
$$
\widetilde{M}=\R^{n+k}_{\nu+\tau}, \qquad \ind(\widetilde{M})=\ind(M)+\ind(\R^k_\tau).
$$
Thus, $\widetilde{R}(X,Y) \in \Gamma(\text{End} T\widetilde{M})$ constantly vanishes so that
\begin{equation}\label{splitting of R bar}
\widetilde{R}(X,Y,Z_1,Z_2)=0,\qquad \widetilde{R}(X,Y,Z,\xi)=0,\qquad \widetilde{R}(X,Y,\xi,\eta) = 0,
\end{equation}
for arbitrary $Z,Z_1,Z_2\in\Gamma(TM)$ and $\xi,\eta\in\Gamma(TM^\perp)$.
Applying projections {\it tan} and {\it nor} to Eq. \eqref{splitting of R bar} and expressing them via $R,R^\perp,\two, S$, and $\na$ as
in Definition \ref{def: nabla, II, nabla perp}, we deduce

\begin{theorem}\label{theorem GCR equations}
The following three equations are equivalent to Eq. \eqref{splitting of R bar}{\rm :}
\begin{align}
& \langle {\rm II}(X,Z_1),\two(Y,Z_2)\rangle - \langle\two(X,Z_2),\two(Y,Z_1)\rangle = R(X,Y,Z_1,Z_2),\label{gauss}\\
& {\na}^\perp_X\two(Y,Z) = {\na}^\perp_Y\two(X,Z), \label{codazzi}\\
& \langle[S_\xi, S_\eta]X,Y\rangle = -R^\perp(X,Y,\xi,\eta) \label{ricci}
\end{align}
for any $X,Y,Z_1,Z_2\in \Gamma(TM)$ and $\xi,\eta\in\Gamma(TM^\perp)$,
where the covariant derivative of $\two$
is defined via the Leibniz rule{\rm :}
\begin{equation*}
{\na}^\perp_X\two(Y,Z) = X(\two (Y,Z)) - \two (\na_X Y, Z) - \two(Y,\na_X Z).
\end{equation*}
\end{theorem}

A sketched proof of the above theorem is given in Appendix A.1,
which is analogous to the derivation in do Carmo \cite[$\S 6$]{docarmo} for
the Riemannian case. The three equations \eqref{gauss}, \eqref{codazzi}, and \eqref{ricci}
are named after Gauss, Codazzi, and Ricci, respectively, which
form the GCR system.

Three remarks on the GCR system are in order:
\begin{enumerate}
\item[(i)]
The GCR system is a first-order nonlinear PDE system on the semi-Riemannian manifold $(M,g)$,
with given $g$ (hence $\na$ and $R$) and unknowns $(\two, \na^\perp)$.
The nonlinear terms in this system are of forms $\two \otimes \two$, $\two \otimes \na^\perp$, or $\na^\perp \otimes \na^\perp$,
which are of {\em quadratic nonlinearity}.

\smallskip
\item[(ii)]
The GCR system in Theorem \ref{theorem GCR equations} takes the same form as in the Riemannian case{\it ;} see \cite{chenli,docarmo,Spi79}.
Such coincidence, nevertheless, is merely formal.
The GCR system for semi-Riemannian manifolds includes the information of non-trivial signatures, which leads to further analytical difficulties.

\smallskip
\item[(iii)]
The GCR system can be generalized to any vector bundle $E$ in place of $TM^\perp$.
Indeed, since the Riemann curvature is defined for any bundle $E$ ({\it i.e.}, $R^E$),
for any symmetric tensor $\two:\Gamma(TM)\times\Gamma(TM)\rightarrow \Gamma(E)$ and $S:\Gamma(E)\times \Gamma(TM)\rightarrow \Gamma(TM)$ given
by
$$
\langle S_\alpha X, Y\rangle = \langle \two(X,Y),\alpha\rangle \qquad\mbox{for $X,Y\in\Gamma(TM)$ and $\alpha\in \Gamma(E)$},
$$
the GCR system in Theorem \ref{theorem GCR equations} is still well-defined for $X,Y,Z_1,Z_2\in \Gamma(TM)$ and $\xi,\eta\in\Gamma(E)$,
wherein we replace $R^\perp$ by $R^E$ in Eq. \eqref{ricci}. Such equations are called {\em the GCR system on bundle $E$}.

Suppose that the trivial bundle of the ambient semi-Euclidean space $T\R^{n+k}$ admits an orthogonal splitting $TM \oplus E$
as the Whitney sum of vector bundles. Then it is clear that the GCR system on bundle $E$ is necessary for the splitting.
Conversely, we will prove in Theorem~5.1 that, for an abstract vector bundle $E$ over $M$, the GCR system on $E$ is also
a sufficient condition for the local existence of such a splitting.
Moreover, the splitting holds globally if $M$ is simply-connected, under suitable regularity assumptions.
\end{enumerate}

\subsection{\, Cartan Structural System}\label{subsec:Cartan}

Now we introduce the {\em Cartan structural system} for the semi-Riemannian submanifolds,
first appeared in the formalism of exterior differential calculus due to E. Cartan ({\it cf}. \cite{clelland}).
This can be viewed as an equivalent form of the Gauss--Codazzi--Ricci system,
which is more suitable for the weak continuity and realizability considerations in the subsequent sections.

Cartan's formalism ({\em a.k.a.} the method of moving frames) is a classical tool in differential geometry;
see \cite{chern,Spi79,sternberg}.
In particular, it plays a crucial role in the establishment of the realization theorem for Riemannian submanifolds
by Tenenblat \cite{Ten71}, as well as the existence and uniqueness of immersions of smooth manifolds into
affine homogeneous spaces by Eschenburg--Tribuzy \cite{german}.
In this paper, we develop Cartan's formalism for the semi-Riemannian submanifolds.
It serves as the foundation for the Cartan structural system.

To set up Cartan's formalism, we need to introduce the frame field on $TM$ and its co-frame field on $T^*M$,
as well as the field of connection $1$-forms. The following convention is adopted:
\begin{convention}\label{convention on index}
From now on, the superscripts and subscripts obey the following rule{\rm :}
\begin{equation*}
1\leq i,j,k,l,s,t \leq n;\qquad n+1\leq \alpha,\beta,\gamma \leq n+k; \qquad 1\leq a,b,c,e \leq n+k.
\end{equation*}
\end{convention}

Now, let $\{\partial_1,\ldots,\partial_n\}\subset \gtm$ be a frame field for $M$; that is, at each point $P$ on $M$,
$\{\partial_i|_P\}_1^n$ forms an orthonormal basis for the tangent space $T_PM$. The orthonormality means
\begin{equation*}
\langle \partial_i, \partial_j\rangle = \delta^i_j \epsilon^i \qquad \text{ for all } i,j\in\{1,\ldots,n\}
\end{equation*}
in the semi-Riemannian settings.
We write $\{\theta^1,\ldots,\theta^n\}\subset \Gamma(T^*M)$ for the co-frame field:
\begin{equation*}
\theta^i(\partial_j) = \delta_{j}^i.
\end{equation*}
Similarly, we can also take $\{\partial_{n+1},\ldots,\partial_{n+k}\}\subset \Gamma(E)$ to be a frame field
for $E$, {\it i.e.}, orthonormal with respect to the bundle metric $g^E$,
and $\{\theta^{n+1},\ldots,\theta^{n+k}\}\subset \Gamma(E^*)$ to be its co-frame field.

In light of Convention \ref{convention on index}, we define the {\em connection $1$-forms}:
\begin{definition}\label{def of W}
Let $(M,g)$ be a semi-Riemannian manifold, and let $E$ be a vector bundle over $M$ with bundle metric $g^E$.
The connection $1$-form $\W$ is a $1$-form-valued $(n+k)\times(n+k)$ matrix field{\rm :}
\begin{equation*}
\W=\{\omega^a_b\} \in \Gamma(\glnk\otimes T^*M),
\end{equation*}
defined component-wise as
\begin{equation}\label{eqn-connection one form}
\begin{cases}
\omega^i_j(\partial_l) := \theta^j(\na_{\p_l}\p_i) = \epsilon^j \langle\na_{\p_l}\p_i,\p_j\rangle,\\
\omega^i_\alpha(\p_j) := \theta^\alpha(\two(\p_i,\p_j)) = \epsilon^\alpha\langle\two(\p_i,\p_j),\p_\alpha\rangle,\\
\omega^\alpha_\beta(\p_i) := \theta^\beta(\na^E_{\p_i}\p_\alpha) = \epsilon^\beta \langle\na^E_{\p_i}\p_\alpha, \p_\beta\rangle,\\
\omega^\alpha_i:=-\epsilon^i\epsilon^\alpha \omega^i_\alpha.
\end{cases}
\end{equation}
\end{definition}

\begin{remark}
\, We identify
$$
\Gamma(\mathfrak{gl}(n+k; \R)\otimes T^*M) \cong \Gamma(\mathfrak{gl}(TM\oplus E)\otimes T^*M) =: \Omega^1(\glnk).
$$
The right-most expression means the space of $\glnk$-valued differential $1$-forms.
In general, for a Lie algebra $\mathfrak{g}$, the space of differential $k$-forms with entries
in $\mathfrak{g}$ is written as
\begin{equation}\label{Omega k}
\Omega^k(\mathfrak{g}):= \Gamma(\wedge^k T^*M \otimes \mathfrak{g}).
\end{equation}
This notation is needed for subsequent development.
\end{remark}

Now we introduce the two Cartan structural systems for semi-Riemannian manifolds, the second of which is equivalent to the GCR system introduced
in $\S \ref{subsec: GCR}$.
This seems to be known in the semi-Riemannian geometry community;
nevertheless, we have not been  able to
locate a proof in the literature,
so it is needed to present a detailed proof for completeness in Appendix A.3.

\begin{proposition}\label{prop_second structural equation}
The GCR system \eqref{gauss}--\eqref{ricci} is equivalent
to the following system for the connection $1$-form $($known as the second structural system$)${\rm :}
\begin{equation}\label{eqn_second structural eqn}
\dd\W = \W\wedge \W.
\end{equation}
\end{proposition}

Its proof relies on a key lemma (see Appendix A.2), which says that $\W$ is a ``semi-skew-symmetric'' matrix:
\begin{lemma}\label{lemma_skew symmetry of W}
$\W=\{\omega^a_b\} \in \Omega^1 (\frako(\nu+\tau; (n+k)-(\nu+\tau)))$.
\end{lemma}

For subsequent developments, we note that $\W$ can be schematically represented in the block-matrix form:
\begin{equation}\label{schematic rep for W}
\{\omega_a^b\}_{1\leq a,b \leq k+n} = \begin{bmatrix}
\omega^i_j & \omega^\alpha_i\\[1mm]
\omega_\alpha^i & \omega^\beta_\alpha
\end{bmatrix} = \begin{bmatrix}
\theta^j(\na_\bullet \p_i) & S_{\partial_\alpha}\p_i\\[1mm]
-S_{\partial_\alpha}\p_i & \theta^\beta (\na^E_\bullet \p_\alpha)
\end{bmatrix}.
\end{equation}

\begin{remark}$\,$
System \eqref{eqn_second structural eqn} is understood {\em as an equality on $\Omega^2(\fg)$}.
On the left-hand side, the exterior differential $\dd$ is viewed as acting only on the $T^*M$ factor
if $\W \in \Omega^1(\mathfrak{g})$, where $\mathfrak{g}=\mathfrak{gl}(n+k; \R)$ in Eq. \eqref{Omega k}.
Then $\dd\W \in \Omega^2(\fg)$ and is given by
\begin{equation*}
\dd\W(U,V):= U(\W(V)) - V(\W(U)) + \W([U,V]) \quad \text{ for all } U, V \in \Gamma(TM).
\end{equation*}
On the right-hand side, the wedge product on $\Omega^1(\fg)$ is taken by combining the wedge product
on the $T^*M$ factor and the matrix multiplication on the $\fg$ factor in Eq. \eqref{Omega k}. That is,
\begin{equation*}
(\W\wedge \W) (U,V):=\W(U)\cdot \W(V)-\W(V)\cdot \W(U) \quad \text{ for all } U, V \in \Gamma(TM).
\end{equation*}
\end{remark}

So far, we have established the equivalence between the GCR system and system \eqref{eqn_second structural eqn}.
It is known as the {\em second} structural system.
In fact, the {\em first} structural system consists of the following identities on $\Omega^1(\mathfrak{gl}(n;\R))$:
\begin{equation}\label{eqn_first structure eq}
\dd\theta = \theta \wedge \W.
\end{equation}
This is equivalent to the torsion-free property of connection $\na$.
As this property is independent of metrics (regardless of Riemannian or semi-Riemannian),
it does not provide additional information to the isometric immersions.
The proof is standard and is sketched in Appendix A.4.

In the rest of the paper, we always refer to the second structural system \eqref{eqn_second structural eqn} as the Cartan structural system.
In $\S \ref{sec: weak continuity of Cartan's structural eq}$, we establish its global weak continuity.

\medskip
\section{\, Weak Continuity of Quadratic Functions on Semi-Riemannian Manifolds}
\label{sec: generalised quadratic thm}

In order to establish the weak continuity of the Cartan structural system on semi-Riemannian manifolds with lower regularity,
we need to pass to the weak limit of the quadratic nonlinear term $\W \wedge \W$,
where $\W$ is the connection $1$-form in Proposition \ref{prop_second structural equation}.
We establish a geometrically intrinsic compensated compactness theorem on vector bundles over the semi-Riemannian manifold
and apply it to develop
a geometric, global approach to our problem. This is the main goal of this section.

Our generalized quadratic theorem concerns the weakly convergent $L^2$ sections
of a vector bundle $E$ over a semi-Riemannian manifold $M$.
Its prototype is the quadratic theorem {\it \`{a} la} Tartar \cite{tartar} on the Euclidean space $\R^n$.
In order to formulate it globally and intrinsically,
two difficulties immediately arise:

\begin{enumerate}
\item[(i)]
Being endowed with a semi-Riemannian metric, $M$ is a {\em real} manifold.
However, our
proof is based on Fourier analysis below involving
factor $i=\sqrt{-1}$, which has to be carried out over $\C$.

\smallskip
\item[(ii)]
The Fourier transform cannot be defined globally on a generic semi-Riemannian manifold.
For $u\in L^2(M; E)$, one way
we can do is to define
\begin{equation*}
\hat{u}(x, \xi):= \int_{M} u(y)e^{-2\pi i \langle {\exp^{-1}_x(y)},\xi\rangle} \,\dd V_g(y)
\qquad\text{ for $x \in M, \xi \in T_x^*M$},
\end{equation*}
where $\exp_x: T_xM \rightarrow M$ is the exponential map on the manifold,
$\langle\cdot,\, \cdot\rangle$ is the paring of $TM$ and $T^*M$ given by metric $g$,
and $\dvg$ is the volume form of $g$.
However, it is only well-defined at $x\in M$
up to the first conjugate point of $x$,
for which $\exp^{-1}_x$ can be specified unambiguously.
\end{enumerate}

The above considerations call for a quadratic theorem on real manifolds,
for which the differential constraints are formulated globally and intrinsically.
For this purpose, we introduce three new ingredients:
\begin{itemize}
\item
The (principal) symbol of a differential operator,

\smallskip
\item
A quadratic polynomial defined globally on vector bundles,

\smallskip
\item
Complexifications of vector bundles and quadratic polynomials.
\end{itemize}

The rest of this section is organized as follows:
We first present the definitions and basic properties of the principal symbol,
quadratic polynomials, and the Sobolev norms of sections over semi-Riemannian manifolds.
Then our generalized quadratic theorem
is first stated and proved over a semi-Riemannian manifold with a $C^\infty$
metric ({\it cf.} Theorem \ref{thm: generalized quadratic theorem on manifolds})
and is then extended
over a semi-Riemannian manifold with a non-degenerate $L^\infty$
metric ({\it cf.} Theorem \ref{thm: generalized quadratic theorem on manifolds-b}).
From now on, let $M$ be a semi-Riemannian manifold, and let $E$ and $F$ be two real vector bundles over $M$.

\medskip
\noindent
{\bf Principal Symbols.} We collect only some basic facts here, and refer to \cite{albin} for the details.

Denote $\T \in \diff^m(M; E,F)$ as an arbitrary differential operator $\T$ of order $m$ that maps $E$-sections to $F$-sections:
$$
\T: \Gamma(E) \rightarrow \Gamma(F).
$$
It is a crucial observation in micro-local analysis that $\sigma_m(\T)$, the {\em principal symbol} of $\T$,
can be defined intrinsically.
Indeed, for any $\xi \in T_x^*M$, we may choose a function $f\in C^\infty(M)$ such that $\dd_xf=\xi$,
and then set
\begin{equation}\label{def of ppl symbol}
\sigma_m(\T)(x,\xi):= \lim_{t\rightarrow \infty} \frac{[e^{-2\pi itf}\circ \T \circ e^{2\pi itf}](x)}{t^m}.
\end{equation}
It is easy to check that $\sigma_m(\T)(x,\xi) \in {\rm Hom}(E_x, F_x)$ for any given $\xi$
and that the definition is independent of the choice of $f$.
Here and hereafter, $E_x \cong \R^J$ and  $F_x \cong \R^I$ denote the fiber of $E$ and $F$ at point $x\in M$, respectively,
and ${\rm Hom}(E_x, F_x)$ denotes the space of vector space homomorphisms from $E_x$ to $F_x$.
Moreover, $\sigma_m$ is a  homogeneous polynomial of order $m$ on each fiber of $T^*M$:
\begin{equation*}
\sigma_m(\T)(x,\lambda\xi) = |\lambda|^m \sigma_m(\T)(x,\xi)\qquad
\text{for all $x\in M, \xi \in T_x^*M$, $\lambda \in \C$}.
\end{equation*}

More abstractly, denoting $\mathcal{P}_l(V,W)$ as the vector space of $l$-degree homogeneous polynomials
between the vector bundles $V$ and $W$,
the principal symbol map $\sigma_m$ defines the following vector space homomorphism:
\begin{equation*}
\sigma_m: \diff^m(M; E,F) \rightarrow \mathcal{P}_m (T^*M; {\rm Hom}(E;F^\C)),
\end{equation*}
where $F^\C:=F\otimes_{\R}\C$ is the complexified vector bundle, which is necessary
since $i=\sqrt{-1}$ appears in the definition of $\sigma_m$ in Eq. \eqref{def of ppl symbol}.
We adopt this abstract language in order to emphasize the global, intrinsic nature
of the principal symbol.

For the application in \S 5, we now discuss the following example:
The exterior differential operator $\T=\dd: \wedge^q T^*M \rightarrow \wedge^{q+1} T^*M$.
In fact, we have
\begin{equation*}
\dd \in \diff^1(M;\,\wedge^qT^*M, \wedge^{q+1}T^*M),
\end{equation*}
whose the principal symbol $\sigma_1(\dd)$ is given by
\begin{equation*}
[\sigma_1(\dd)(\xi)](\omega) = -2\pi i\xi \wedge \omega
\qquad \text{ for $\xi \in T^*M$ and $\omega \in \wedge^qT^*M$}.
\end{equation*}
Owing to the presence of $i=\sqrt{-1}$, we view the exterior algebra in the range of $\dd$ as being complexified:
For each $\xi \in T^*M$, $\sigma_1(\dd)(\xi) \in \mathcal{P}_1(\wedge^qT^*M;\, \wedge^{q+1}T^*M\otimes \C)$.
In this case, notice that $\sigma_1(\dd)(\xi) = -2\pi i \xi \wedge$, which is indeed a $1$-homogeneous polynomial
of operators from $q$-tensors to  complexified $(q+1)$-tensors.

\medskip
\noindent
{\bf Intrinsic Formulation of Quadratic Polynomials.}
Now we define
a quadratic polynomial on a vector bundle $E$:

\begin{definition}
Let $E$ be a vector bundle over a real manifold $M$.
A map $Q: \Gamma(E) \rightarrow \C$ is a quadratic polynomial on $E$ if it factors as
\begin{align*}
Q: \Gamma(E) \stackrel{j}\longrightarrow \Gamma(E \otimes E) \stackrel{\mathbf{q}}\longrightarrow \C,
\end{align*}
where $j(s)=(s,s)$ is the natural inclusion of the diagonal, and $\mathbf{q} \in \Gamma({\rm Hom}(E\otimes E; \C))$ is
conjugate $1$-homogeneous in each argument{\rm :}
\begin{align*}
\mathbf{q}(\lambda s_1, s_2) = \lambda \mathbf{q}(s_1, s_2),\qquad
\mathbf{q}(s_1, \mu s_2) = \overline{\mu} \mathbf{q}(s_1, s_2)
\end{align*}
for all $s_1, s_2 \in \Gamma(E)$ and $\lambda, \mu \in \C$.
In this case, we write $Q \in \mathcal{P}_2(E; \C)$.
\end{definition}

Such constructions remain valid for $\C$ replaced by $\R$, in which $Q$ is said to be a real quadratic polynomial on $E$.
It follows from the definition that any quadratic polynomial $Q$ is $2$-homogeneous:
\begin{equation*}
Q(\lambda s) = |\lambda|^2 Q(s)\qquad \text{ for all } s \in \Gamma(E) \text{ and } \lambda \in \C.
\end{equation*}

Moreover, suppose that $U \subset M$ is a trivialized chart for the vector bundle $E$ of degree $J$, {\it i.e.},
there exists a diffeomorphism:
$$
\Phi: E \supset \pi^{-1}(U)\stackrel{\sim}\longrightarrow U \times \C^J.
$$
Then, for $s =\Phi^{-1}(x,z)\in U \times \C^J$ with $(x,z) \in U \times \C^J$,
the value of the quadratic polynomial $Q$ at $s$ is given by
\begin{equation}\label{Q quadratic-new}
Q(s)= \sum_{j,k=1}^J Q_{jk}(x) z^j \overline{z^k} \qquad \text{ with } Q_{jk} \in C^\infty(U),
\end{equation}
so that the local representation of $Q$ is obtained.

\medskip
\noindent
{\bf Sobolev Norms over Semi-Riemannian Manifolds.}
Now let us explain the construction of Sobolev norms (of sections of vector bundles) over semi-Riemannian manifolds.

Let $(M,g)$ be a semi-Riemannian manifold.
As we are concerned only with the local Sobolev spaces over $M$ in this paper
(see Theorems~\ref{thm: generalized quadratic theorem on manifolds}, \ref{thm: weak continuity of Cartan's structural eq}, and \ref{theorem_weak rigidity}),
without loss of generality, we may assume $M$ to be compact.
Let $\mathfrak{U}:=\{U_j\}_{j=1}^J$ be an atlas of coordinate charts on $M$.
Given an arbitrary $(r,s)$-tensor field ${\bf T}$ on $(M,g)$,
by restricting to each chart in $\mathfrak{U}$, one may express it in local coordinates by ${\bf T}^{i_1,\ldots, i_r}_{j_1, \ldots, j_s}$.
More precisely, let
$\{\frac{\partial}{\partial x^i}\}_{i=1}^n$
be an local orthonormal basis for $g$,
{\it i.e.}, $g(\frac{\partial}{\partial x^i},\frac{\partial}{\partial x^j})=\epsilon^j\delta_{ij}$ (no summation)
with $\epsilon=(\epsilon^1, \ldots, \epsilon^n)$
as the signature of $g$,
and let $\{\dd x^i\}$ be the co-frame dual to $\{\frac{\partial}{\partial x^i}\}$ via $g$. Then
\begin{align*}
{\bf T}^{i_1,\ldots, i_r}_{j_1, \ldots, j_s}
:= {\bf T}\big(\frac{\partial}{\partial x^{j_1}} \otimes \ldots \otimes  \frac{\partial}{\partial x^{j_s}}
   \otimes \dd x^{i_1}\otimes \ldots \otimes \dd x^{i_r}\big).
\end{align*}
The inner product of two $(r,s)$-tensor fields ${\bf T}$ and ${\bf S}$ on $(M,g)$ is given by
\begin{align}\label{def, inner product of tensors}
\langle{\bf T}, {\bf S}\rangle_g
:= {\bf T}^{i_1,\ldots, i_r}_{j_1, \ldots, j_s} {\bf S}^{a_1,\ldots, a_r}_{b_1, \ldots, b_s}
g_{i_1, a_1}\cdots g_{i_r, a_r} g^{j_1,b_1}\cdots g^{j_s,b_s},
\end{align}
where ${\bf T}=\{{\bf T}^{i_1,\ldots, i_r}_{j_1, \ldots, j_s}\}$ and ${\bf S} = \{{\bf S}^{a_1,\ldots, a_r}_{b_1, \ldots, b_s}\}$
in the local coordinates of $\mathfrak{U}$. Then we set
\begin{equation*}
|{\bf T}|_g := \sqrt{|\langle{\bf T}, {\bf T}\rangle_g|\,}.
\end{equation*}
Note that Eq.\,\eqref{def, inner product of tensors} can be readily interpreted as an $L^1_\loc$ function
when $g$ is invertible {\it a.e.},  $g_{ij}$ lies in $L^\infty$ for each $i,j$,
and ${\bf T}^{i_1,\ldots, i_r}_{j_1, \ldots, j_s}$ and ${\bf S}^{a_1,\ldots, a_r}_{b_1, \ldots, b_s}$ lie in $L^p$, $p\geq 2$,
for all possible indices $i_1,\ldots, i_r, j_1, \ldots, j_s, a_1,\ldots, a_r$, and $b_1, \ldots, b_s$.

Now, take a scalar function $f: (M, g)\,\map\, \R$.
Similar to the Riemannian case ({\it cf.} Chapter 2 in Hebey \cite{h-book}),
we define its $W^{k,p}$ norm, $k=0,1,2,\ldots$, by
\begin{equation}\label{Wk,p norm}
\|f\|_{W^{k,p}(M,g)}
:= \begin{cases} \sum_{m=0}^k \Big\{\displaystyle \int_M \big(|\na^mf|_g\big)^p \,{\rm d}V_g\Big\}^{\frac{1}{p}}\qquad\text{for  $p\in[1,\infty)$},\\[2mm]
\sum_{m=0}^k \text{ess\,sup}_{M} \, |\na^m f|_g\qquad\text{for  $p=\infty$}.
\end{cases}
\end{equation}
In the above, $\na^m:=\overbrace{\na\circ\ldots\circ\na}^{\text{$m$ times}}$ denotes the iterated covariant derivatives,
and the semi-Riemannian volume form is
\begin{equation}\label{volume form}
{\rm d}V_g := \sqrt{|\det\,g|}\,{\rm d}\mathcal{L}^n
\end{equation}
on each local chart of $\mathfrak{U}$, with  the Lebesgue measure $\mathcal{L}^n$.
The integration of a scalar function on $M$ with respect to ${\rm d}V_g$ is defined in the standard way,
by using an arbitrary partition of unity subordinate to $\mathfrak{U}$. The Sobolev space $W^{k,p}(M,g)$ is the completion of $C^\infty(M)$
under the norm in Eq.\,\eqref{Wk,p norm}. For $k<0$ and $p\in[1,\infty]$, $W^{k,p}(M,g)$ is defined as the dual space of $W^{-k,p'}(M,g)$,
where $\frac{1}{p}+\frac{1}{p'}=1$.

A tensor field ${\bf T}$ on $(M,g)$ is said to have $W^{k,p}$--regularity if and only if ${\bf T}^{i_1,\ldots, i_r}_{j_1, \ldots, j_s}\in W^{k,p}(M,g)$
for all indices $i_1, \ldots, i_r$, and $j_1, \ldots, j_s$. Similarly, a connection on $TM$ is $W^{k,p}$ if and only if its Christoffel symbols
$\Gamma^\alpha_{\beta\gamma} \in W^{k,p}(M,g)$ for all $\alpha,\beta$, and $\gamma$.
Given a vector bundle $E$ over $(M,g)$ equipped with the bundle metric $g^E$, we write $W^{k,p}(M,g; E,g^E)$ for the space of $E$-sections
with $W^{k,p}$ regularity, defined in an analogous manner as for tensor fields by considering trivialized charts for $E$.

We remark that the above definition of the $W^{k,p}$--norms may depend on the atlas $\mathfrak{U}$ and the trivialization of bundle $E$.
Nonetheless, all these norms are equivalent modulo constants depending only on the differentiable structure of $M$.
Thus, the corresponding Sobolev spaces are identical vector spaces with equivalent topologies; in particular, they are independent
of local coordinates.

From now on, we assume that the semi-Riemannian metric $g$ lies in $L^\infty_\loc$,
with the non-degeneracy condition (see \S \ref{subsec: semi-riem geometry prelims})
understood in the {\it a.e.} sense. This is a very natural and mild condition,
which suggests that $M$ as a metric space does not contain interior infinity points.
As a consequence, $\det g$ and $g^{-1}$ ({\it e.g.}, obtained from the Cramer's rule) are also in $L^\infty_\loc$, and
hence $\dvg$ defined in Eq.\,\eqref{volume form} is an $L^\infty_\loc$ differential $n$-form.

\medskip
With the preceding preparations, we now state our geometric quadratic theorem
on vector bundles over a semi-Riemannian manifold, first with a $C^\infty$ metric $g$.

\begin{theorem}\label{thm: generalized quadratic theorem on manifolds}
Let $M$ be a semi-Riemannian manifold with a $C^\infty$ metric $g$.
Let $E$ and $F$ be two real $C^\infty$ vector bundles over $M$.
Consider a family of $E$-sections $\{u_\varepsilon\} \subset L^2_\loc (M; E)$,
a differential operator $\T \in \diff^m(M; E, F)$ for some $m \in \R_{+}$ with the principal
symbol $\sigma_m(\T): T^*M \rightarrow {\rm Hom}(E; F^\C)$, and a quadratic
polynomial $Q : \Gamma(E) \rightarrow \R$. If the following conditions hold{\rm :}
\begin{enumerate}
\item[{\em (C1)}]
$u_\varepsilon \weak u$ weakly in $L^2_\loc(M;E)$,
\item[{\em (C2)}]
$\{{\T u_\varepsilon}\}$ is pre-compact in $H^{-m}_{\rm loc}(M; F)$,
\item[{\em (C3)}]
$Q\circ s=0$ for all $s \in \Lambda_\T$, where the {\em cone} of $\T$ is defined by
\begin{equation*}
\Lambda_\T:=\big\{s \in \Gamma(E)\,:\,\sigma_m(\T)(\xi)(s) = 0
\text{ for some $\xi \in T^*M \setminus \{0\}$}\big\},
\end{equation*}
\end{enumerate}
then, for any $\psi \in C^\infty_c(M)$,
\begin{equation*}
\lim_{\varepsilon\rightarrow 0} \int_M  (Q\circ u_\varepsilon)(x)\,\psi(x)\,\dvg(x)
= \int_M  (Q\circ u)(x)\,\psi(x)\,\dvg(x).
\end{equation*}
\end{theorem}

\medskip
Before presenting the proof, we make several remarks on  Theorem \ref{thm: generalized quadratic theorem on manifolds}:

\begin{enumerate}
\item[(i)]
Theorem \ref{thm: generalized quadratic theorem on manifolds} is formulated globally and intrinsically
on the semi-Riemannian manifold $M$, since symbol $\sigma_m$,
cone $\Lambda_\T$, and  the Sobolev spaces $H^\bullet$ of sections are all defined
without referring to local coordinates.
In addition, $\sigma_m$ is defined only by using the differentiable structure of $M$,
without resort to the Riemannian or semi-Riemannian structure.
Therefore, cone $\Lambda_\T$ in (C3) depends only on the {\em algebraic} properties of $\T$.

\smallskip
\item[(ii)]
In Theorem \ref{thm: generalized quadratic theorem on manifolds},
we denote the target space of symbol  $\sigma_m(\T)$ by ${\rm Hom}(E;F^\C)$,
which is understood as the vector bundle of $\R$-bundle homomorphisms from $E$ to the {\em complexification} of $F$,
{\it i.e.}, $F^\C := F \otimes \C$.
It is also common to write it as
\begin{equation*}
\sigma_m(\T) \in \Gamma(TM \otimes {\rm Hom}(E; F^\C)).
\end{equation*}

\item[(iii)]
The following lemma concerns the {\it naturality} of the principal symbol under the action
of diffeomorphism group. It is crucial for the proof of Theorem \ref{thm: generalized quadratic theorem on manifolds}.
\end{enumerate}

\begin{lemma}\label{lemma: naturality of symbols under diffeos}
Let $\mathcal{O}$ and $\tilde{\oo}$ be open subsets of $\R^n$,
and let $F: \oo \rightarrow \tilde{\oo}$ be a diffeomorphism.
Then $F_\ast P \in \diff^m(\tilde{\mathcal{O}})$ for $P \in \diff^m(\oo)$.
Moreover, the principal symbols of $P$ and $F_\ast P$, {\it i.e.},
$\sigma_m(F_\ast P)$ and $\sigma_m(P)$, are related as
\begin{equation*}
\sigma_m(F_\ast P)(F(x), \xi)
 = \sigma_m(P) (x, [\dd_x F]^\top(\xi))
 \quad \text{for each $x\in \oo$ and $\xi \in T^*\R^n$},
\end{equation*}
where $F_\ast P$ denotes the {\em pushforward} of $P$ under $F${\rm :}
\begin{equation*}
(F_\ast P)(\varphi) = P(\varphi \circ F) \circ F^{-1} \qquad \text{ for all } \varphi \in C^\infty_c(\tilde{\oo}).
\end{equation*}
\end{lemma}

This is a special case of Theorem 20 in \cite{albin}.
In full generality, the first assertion holds for general pseudo-differential operators,
and the second assertion holds for pseudo-differential operators with classical
total symbols.

The strategy for the proof of Theorem \ref{thm: generalized quadratic theorem on manifolds} is as follows:
First of all, using a partition-of-unity, together with the commutator estimate of $\T \in \diff^m(M; E,F)$
and a multiplication operator,
we reduce the theorem to a local problem on one single chart of the manifold.
Next, thanks to Lemma \ref{lemma: naturality of symbols under diffeos},
we can flatten the local chart to $\R^n${\rm ;} this cannot be done directly, owing to the non-trivial semi-Riemannian metrics on the
manifold and the bundles.
Nevertheless, in view of the quadratic structure of $Q$, the {\em signature} of the semi-Riemannian metrics
does not affect the proof.
Therefore, locally we can regard the metrics as ``close'' to the Euclidean metrics,
and then modify the arguments  by Tartar \cite{tartar} to complete the proof.

\bigskip
\noindent
{\bf Proof of Theorem {\rm \ref{thm: generalized quadratic theorem on manifolds}}}.
The proof is divided into eight steps.
	
\medskip
\noindent
{\bf 1.} We first justify the following two reductions:

\begin{enumerate}
\item[(i)] It suffices to prove the theorem for $u=0$. Indeed, we note that
\begin{align*}
Q(u_\varepsilon-u) = Q(u_\varepsilon) + Q(u) - 2\sum_{i,j=1}^J Q_{ij} u_\varepsilon^i u^j
\end{align*}
for $Q$ is a real quadratic polynomial. Condition (C1) yields
$$
\sum_{i,j=1}^J Q_{ij} u_\varepsilon^i u^j \weak Q(u)
\qquad \mbox{weakly in $L^2_\loc$}.
$$
Thus, $Q(u_\varepsilon-u)$ and $Q(u_\varepsilon)-Q(u)$ have the same distributional limit as $\varepsilon\map 0$.

\smallskip
\item[(ii)] We can {\em localize} the statement to each chart of the differentiable manifold $M$.
To fix the notations, let $\{U^k\}_{k \in \mathcal{I}}$ be an atlas of the differentiable manifold $M$. We {\it claim} that it suffices
to prove Theorem \ref{thm: generalized quadratic theorem on manifolds} for sequence $\{u_\varepsilon\}$ supported on one single $U^k$.
\end{enumerate}	

For this purpose, take any $\psi \in C^\infty_c(M)$ and consider the following identity:
	\begin{equation*}
	\T(\psi u_\varepsilon) = \psi \T u_\varepsilon + [\T, \psi] u_\varepsilon,
	\end{equation*}
where $[\T,\psi]$ denotes the commutator of $\T$ and the operator of multiplication by $\psi$.

Clearly, $[\T,\psi]$ is a differential operator of order not exceeding $m-1$.
Since $\{u_\varepsilon\}$ is pre-compact (hence uniformly bounded) in $L_\loc^2$,
$\{[\T,\psi]u_\varepsilon\}$ is uniformly bounded in $H^{-m+1}_{\rm loc}$,
which is compactly embedded in $H^{-m}_{\rm loc}$ by the Rellich lemma.
Moreover, by condition (C2),
$\{\psi \T u_\varepsilon\}$ is also pre-compact in $H^{-m}_{\rm loc}$.
Thus, the same holds for $\{\T(\psi u_\varepsilon)\}$.
In addition, the transition function $\varphi^{k,l}$ between any two overlapping charts $U^k$ and $U^l$
is a diffeomorphism,
so that both $\T|_{U^k}$ and $\T|_{U^l}$ have the principal symbols of order $m$,
which are $m$-homogeneous polynomials in the fiber of the cotangent bundle $T^*M$.
Indeed, they differ only by a multiplicative factor controlled by the Lipschitz norm of $\varphi^{k,l}$,
which is bounded uniformly on $M$ for all $k,l\in \mathcal{I}$.
Up to now, we have justified that the assumptions of
the theorem are invariant under operation $u_\varepsilon \mapsto \psi u_\varepsilon$,
where $\psi \in C^\infty_c(M)$ is an arbitrary test function.

It remains to establish the {\em local-to-global} result:
If the assertion holds for $\{u_\varepsilon\}$ supported in each chart,
then it also holds for arbitrary $\{u_\varepsilon\}$.
To this end, let $\{\phi_k\}_{k\in \mathcal{I}}$ be a partition-of-unity subordinate to
atlas $\{U^k\}_{k \in \mathcal{I}}$, {\it i.e.}, $0 \leq \phi^k\leq 1$, $\phi^k\in C^\infty_c(U^k)$ for
each $k\in \mathcal{I}$, and $\sum_{k\in \mathcal{I}} \phi^k =1$ on $M$.
Then we can find $\psi^k \in C^\infty_c(U^k)$ with $0 \leq \psi^k\leq 1$
such that $\phi^k = (\psi^k)^2$ for each $k \in \mathcal{I}$.
To proceed, suppose that Theorem \ref{thm: generalized quadratic theorem on manifolds} is proved
for sequence $\{\psi^k u_\varepsilon\} \subset L^2(U^k; E)$ for each $k\in\mathcal{I}$,
with $\psi^k u_\varepsilon \weak \psi^k u$ in $L^2(U^k; E)$ along some
subsequence $\{\psi^k u_{\varepsilon_i}\}_{i \in \mathcal{I}_1\subset \mathcal{I}}$.
Then, for a neighboring chart $U^l$, {\it i.e.}, $U^k\cap U^l \neq \emptyset$,
we can select a further subsequence
$\{\psi^l u_{\varepsilon_j} \}_{j \in {\mathcal{I}}_2}\subset \{\psi^k u_{\varepsilon_i}\}_{i \in \mathcal{I}_1}$
such that $\mathcal{I}_2 \subset \mathcal{I}_1$, and $\{\psi^l u_{\varepsilon_j} \}$ converges
weakly in $L^2(U^l; E)$ to some $\psi^l\tilde{u}$.
However, due to the uniqueness of subsequential weak limits, we have
\begin{equation*}
u = \tilde{u} \qquad \text{ on $U^k \cap U^l$}.
\end{equation*}
Hence, we can write $\psi^l\tilde{u}$ as $\psi^l u$ without ambiguity,
according to the interpretation: the limit function $u$, previously defined only on $U^k$,
is now extended to domain $U^k \cup U^l$.

Now, since $M$ is second-countable (which is a part of the definition of differentiable manifolds),
we can take the index set $\mathcal{I}$ for the atlas to be at most countable.
Thus, performing a diagonalization process to the arguments in the preceding paragraph,
we obtain a subsequence (still denoted) $\{u_\varepsilon\}$ and a function $u \in L^2_\loc(M; E)$ defined
on manifold $M$ such that
\begin{equation*}
\psi^k u_\varepsilon \weak \psi^k u \qquad \text{ for each $k \in \mathcal{I}$}.
\end{equation*}
Therefore, for any test function $\psi \in C^\infty_c(M)$, we can pass to the limit as follows:
\begin{align}\label{localization}
\lim_{\varepsilon\rightarrow 0} \int_M (Q\circ u_\varepsilon)(x)\psi(x) \dvg(x)
&=  \sum_{k\in \mathcal{I}}
   \lim_{\varepsilon\rightarrow 0} \int_M  {(\psi^k)^2(x)} \,(Q\circ u_\varepsilon)(x)\psi(x) \dvg(x)\nonumber \\
&= \sum_{k\in \mathcal{I}}\lim_{\varepsilon\rightarrow 0} \int_M Q(\psi^k(x) u_\varepsilon(x))\psi(x) \dvg(x) \nonumber\\
&= \sum_{k\in\mathcal{I}} \int_M Q (\psi^k(x) u(x))\psi(x)\dvg(x)\nonumber\\
&= \sum_{k\in\mathcal{I}}\int_M (\psi^k)^2(x)\,(Q\circ u)(x)\psi(x)\dvg(x)\nonumber\\
&= \int_M (Q\circ u)(x) \psi(x)\dvg(x).
\end{align}
In the first and the last lines of \eqref{localization},
we have used that $\sum_{k\in \mathcal{I}}(\psi^k)^2=1$ on $M$,
while in the second and the fourth lines, we have used the quadratic structure of $Q$.
Moreover, the order of summation over $\alpha \in \mathcal{I}$ can be interchanged with the limit
and the integration, because the partition-of-unity is locally finite.
Then the localization argument is completed by using Eq. \eqref{localization}.

\medskip
\noindent
{\bf 2.} From now on, $\{u_\varepsilon\}$ is assumed to be supported on a single chart $U^k \subset M$.
In this step, we {\em flatten the chart} by transforming $U^k$ to $\R^n$ via the coordinate map.
First, without loss of generality, we assume that the vector bundles $E$ and $F$ are trivialized on $U^k$;
otherwise, a refinement of atlas $\{U^k\}_{k \in \mathcal{I}}$ can be made if necessary.
Now, by the basic manifold theory, there exists a diffeomorphism $F^k: U^k \stackrel{\sim}\longrightarrow \R^n$
so that
\begin{equation*}
F^k_{\ast} \T \in \diff^m(\R^n;\,\R^n \times \R^J, \R^n \times \R^I).
\end{equation*}
Here and hereafter, we assume that $E$ and $F$ have typical fibers
$\R^J$ and $\R^I$, respectively.

Moreover, Lemma \ref{lemma: naturality of symbols under diffeos} implies
\begin{equation}\label{flatten local chart}
\sigma_m (F^k_{\ast} \T)(F^k(x), \zeta)
= \sigma_m(\T)(x, [\dd_x F^k]^\top(\zeta)) \quad \mbox{for all $x \in U^k$ and $\zeta \in \R^J$}.
\end{equation}
Notice that $\zeta$ and $[\dd_x F^k]^\top(\zeta)$ are simultaneously
non-vanishing in Eq. \eqref{flatten local chart}, since $F^k$ is a diffeomorphism.
We conclude
\begin{equation*}
\Lambda_\T = \Lambda_{F^k_\ast \T},
\end{equation*}
{\it i.e.}, the cones of $\T$ and $F^k_\ast \T$ coincide.

Therefore, it suffices to prove the theorem with $\{\dd F^k(\psi^k u_\varepsilon)\}$ and $F^k_\ast \T$
in place of $\{u_\varepsilon\}$ and $\T$, respectively, where $\{\psi^k: k \in \mathcal{I}\}$
is a partition-of-unity subordinate
to atlas $\{U^k\,:\, k \in \mathcal{I}\}$ as in Step $1$.
In addition, by the paracompactness of topological manifolds,
we may assume $\psi^k$ to be supported in a compact subset of $U^k$ for each $k\in \mathcal{I}$.
Thus, in the sequel, we take $\dd F^k(\psi^k u_\varepsilon)$ to be compactly supported in $\R^n$
and identify it with the map on the whole of $\R^n$,
obtained via the extension-by-zero.
To simplify the notations, we still label $\{\dd F^k(\psi^k u_\varepsilon)\}$ as $\{u_\varepsilon\}$.
Thus, we reduce to the case: $M=\R^n$.

\medskip
\noindent
{\bf 3.} Thanks to the localization and flattening arguments in Steps 1--2,
from now on, we assume
$\{u_\varepsilon\} \subset L^2_c(\R^n, \R^n \times \R^J)$ and $\T \in \diff^m(\R^n, \R^n \times \R^J; \R^n \times \R^I)$.
To simplify the notations, we  still write $E= \R^n \times \R^J$ and $F=\R^n \times \R^I$, and denote the metric
on $\R^n$ by $g$ with an abuse of notations, {\it i.e.}, assuming that $M=\R^n$, and the bundles $E$ and $F$ are globally trivialized.

To begin with, recall that the $L^p$ norm of $u: \R^n \rightarrow E$ is defined as
\begin{align*}
\|u\|_{L^p(\R^n; E)} :=&\, \Big(\int_{\R^n} |u|_{g^E}^2 \dvg\Big)^{\frac{1}{p}}\\
=&\, \Big(\int_{\R^n} \Big\{\sum_{j=1}^J\sum_{k=1}^J \e^k g^E_{jk}(x) u^j(x) {u^k(x)}\Big\}^{\frac{p}{2}}\sqrt{|\det\,g(x)|}\,\dd x\Big)^{\frac{1}{p}},
\end{align*}
where $g^E$ is the bundle metric on $E$, indices $1\leq j,k \leq J$ are for the fiber of $E$,
and $\e^k \in \{\pm 1\}$ is the {\em signature} of the $k$-th component of $g^E$ such
that $(h_{jk}):=(\epsilon^k g^E_{jk})$ becomes positive definite. Here and in the sequel,
we choose a coordinate system in which $g^E$ is diagonalized:
\begin{align*}
g^E =\, {\rm diag}(\lambda^1, \ldots, \lambda^\tau; \lambda^{\tau+1}, \ldots, \lambda^{J}),
\end{align*}
where $\lambda^j < 0$ for $1 \leq j \leq \tau$, and $\lambda^j > 0$ for $\tau+1 \leq j \leq J$.
Correspondingly, $\epsilon^1 = \cdots = \epsilon^\tau = -1$ and $\epsilon^{\tau+1} = \cdots = \epsilon^{J} = 1$,
where $\tau$ is the index of $g$.

\smallskip
Now, define a new sequence of sections $\{v_\varepsilon\} \subset L^2(\R^n; E)$ by components:
\begin{equation}\label{def of sequence v-tilde}
v_\varepsilon^j:= \sqrt{\e^j\lambda^j} {|\det\, g|}^{\frac{1}{4}} u_\varepsilon^j \qquad \text{ for each }
j = 1, 2, \ldots, J.
\end{equation}
That is, we write $v_\varepsilon =(v_\varepsilon^1, \ldots, v_\varepsilon^J)^\top$.
By this definition, $v_\varepsilon$ depends on $g^E$, $g$, and $u_\varepsilon$, and the following identity holds:
\begin{equation}\label{u L2norm = v L2 norm}
\|v_\varepsilon\|^2_{L^2(\R^n,\, g_0; E)} \equiv \sum_{j=1}^J \int_{\R^n} \{v_\varepsilon^j(x)\}^2 \,\dd x
= \|u_\varepsilon\|_{L^2(\R^n,\: g; E)}^2 \quad \text{for each $\varepsilon$},
\end{equation}
where $g_0$ denotes the {\em Euclidean} metric on $\R^n$.
Thus, by condition (C1),
$\{v_\varepsilon\}$ is uniformly bounded in $L^2$ with respect to $g_0$.
Moreover, ${\rm supp}(v_\varepsilon) \subset {\rm supp}(u_\varepsilon)$ for each $\varepsilon$ so that
all the terms of $\{v_\varepsilon\}$
are supported on a common compact set. By the Riemann-Lebesgue lemma,
there are finite numbers $K, K'>0$ such that $\|\hat{v}_\e\|_{L^\infty(B_{K})} \leq K'$,
where $B_K$ is the Euclidean ball $\{\xi\in\R^n\,:\,|\xi|< K\}$. Thanks to the Parseval identity, the Cauchy--Schwarz inequality,
and (C1), we now have
\begin{align}\label{low-frequency region for sequence v}
\lim_{\varepsilon \rightarrow 0}\int_{|\xi| \leq K} |\hat{v}_\varepsilon(\xi)|^2 \,\dd\xi
&\leq K' \lim_{\varepsilon \rightarrow 0}\int_{\R^n} \hat{v}_\varepsilon \overline{{\rm sgn}(\hat{v}_\varepsilon)}\chi_{B_{K}}\,\dd \xi \nonumber\\
&\leq K' \sqrt{K^n} \lim_{\varepsilon \rightarrow 0} \|\hat{v}_\varepsilon\|_{L^2(\R^n,g_0;E)} = 0,
\end{align}
where ${\rm sgn}(z):=\frac{z}{|z|}$ for $z\neq 0$,
the choice of $K'$ is immaterial, and $K$ will be further specified later.

As a remark, the norm on $\xi$ is also taken with respect to the Euclidean metric,
since it is the metric induced by $g_0$ on the cotangent bundle $T^*M$.

\medskip
\noindent
{\bf 4.}
Next we control the high-frequency region of $\{v_\varepsilon\}$. For $j=1,2,\ldots, J$, define
\begin{equation*}
\chi^j=\chi^j(g):= |\det\, g|^{\frac{1}{4}}  \sqrt{\e^j\lambda^j},
\end{equation*}
so that $v_\varepsilon^j=\chi^j u_\varepsilon^j$ for each $j$. Notice that $\chi^j>0$ strictly,
by the non-degeneracy of metrics $g$ and $g^E$.
Writing  $u_\varepsilon=\sum_{j=1}^J u_\varepsilon^j \p_j$ and similarly for $v_\varepsilon$ in local coordinates,
by the linearity of the differential operator $\T$, we have
\begin{align*}
\T v_\varepsilon &= \T\Big\{\sum_{j=1}^J \chi^j u_\varepsilon^j\p_j\Big\} \\
&= \sum_{j=1}^J\chi^j \, \T (u_\varepsilon^j\p_j) + \sum_{j=1}^J \,  [\T, \chi^j]u_\varepsilon^j\p_j =:\one_\varepsilon + \two_\varepsilon.
\end{align*}
In $\two_\varepsilon$, $[\T,\chi^j]$
is the commutator between $\T$ and the multiplication operator by $\chi^j$.

We now argue that {\it $\{\T v_\varepsilon\}$ is pre-compact in $H^{-m}(\R^n, g_0; F)$}.
First of all, this sequence is compactly supported, by the construction of $\{v_\varepsilon\}$ and the locality of the differential operator $\T$.
Thus, we neglect subscript ``{\rm loc}'' for the corresponding Sobolev spaces.
By explicitly writing out $g_0$ in the subscript,
we emphasize that $M=\R^n$ is equipped with the Euclidean metric.
To this end, we now prove that both $\{\one_\varepsilon\}$ and $\{\two_\varepsilon\}$
are pre-compact in $H^{-m}(\R^n, g_0; F)$.

For $\one_\varepsilon$, we first compute:
\begin{align}\label{term I mu 1}
\|\one_\varepsilon\|_{H^{-m}(\R^n, g_0; F)}
&\leq \sup_{1\leq j \leq J} \|\chi^j -1\|_{L^\infty(\R^n)}\, \|Tu_\varepsilon\|_{H^{-m}(\R^n, g_0; F)}\nonumber\\
&\leq \big(1+ \|g^E\|^{\frac{1}{2}}_{L^\infty(E)}\|\det \,g\|^{\frac{1}{4}}_{L^\infty(M)}\big)\|Tu_\varepsilon\|_{H^{-m}(\R^n, g_0; F)}.
\end{align}

Next, we show that the final term $\|Tu_\varepsilon\|_{H^{-m}(\R^n, g_0; F)}$ can be related to $\|Tu_\varepsilon\|_{H^{-m}(\R^n, g; F)}$,
whose pre-compactness is assumed by condition (C2).
For this purpose, it requires to invoke the Fourier characterization of the Sobolev norms $\|\cdot\|_{H^{-m}(\R^n, g_0; F)}$
and $\|\cdot\|_{H^{-m}(\R^n, g; F)}$.
Since we have localized sequence $\{u_\varepsilon\}$ to a chart $U^k$ of $M$, on which $E$ and $F$ are trivialized in Steps 1--2,
$g|_{U^k}$ has no self-intersecting geodesics,
provided that $U^k$ is contained in a geodesic normal neighborhood.
This can be assumed by shrinking $U^k$ if necessary.
Then the pushforward metric $F^k_\ast g$ --- which is still labelled as $g$ from Step $2$ onward --- satisfies
the same property on $\R^n = M$, so that $\|\cdot\|_{H^{-m}(\R^n, g; F)}$ can be defined globally
via the Fourier transform unambiguously.

In this way, we now obtain
\begin{align}\label{term I mu 2}
\|\T u_\varepsilon\|_{H^{-m}(\R^n, g_0; F)} &= \int_{\R^n}\frac{|\widehat{\T u_\varepsilon}(\xi)|_{g^F}^2}{(1+|\xi|^2)^{m/2}} \,\dd\xi \nonumber \\
&= \int_{\R^n}\frac{|\widehat{\T u_\varepsilon}(\xi)|_{g^F}^2}{(1+|\xi|_g^2)^{m/2}} \sqrt{|\det \, g|}
\bigg\{\frac{(1+|\xi|_g^2)^{m/2}}{(1+|\xi|^2)^{m/2}} \frac{1}{\sqrt{|\det \, g|}} \bigg\}\,\dd\xi \nonumber\\
&\leq C\int_{\R^n}\frac{|\widehat{\T u_\varepsilon}(\xi)|_{g^F}^2}{(1+|\xi|_g^2)^{m/2}} \sqrt{|\det \, g|} \,\dd\xi \nonumber\\
&=: C \|\T u_\varepsilon\|_{H^{-m}(\R^n, g; F)},
\end{align}
where $C$ depends only on $m$, $\|g\|_{L^\infty}(M)$, and $\inf_M |\det \, g|$.
Together with Eq. \eqref{term I mu 1}, we have
\begin{equation*}
\|\one_\varepsilon\|_{H^{-m}(\R^n, g_0; F)}  \leq \tilde{C} \|\T u_\varepsilon\|_{H^{-m}(\R^n, g; F)},
\end{equation*}
where $\tilde{C}$ depends only on $g$, $g^E$, and $m$,
but independent of $\varepsilon$.
In view of (C2),  $\{\one_\varepsilon\}$ is pre-compact in $H^{-m}(\R^n, g_0; F)$.

We now turn to $\{\two_\varepsilon\}$: $\,$ Since $\T\in \diff^m(M; E,F)$ and $\chi^j$ is a multiplication operator,
$[\T, \chi^j] \in \diff^r(M; E, F)$ for $r \leq m-1$.
By assumption (C1), $\{u_\varepsilon\}$ is bounded in $L^2(M, g; E)$,
hence $\{\two_\varepsilon\}$ is pre-compact in $H^{-m}(\R^n, g; F)$ due to the Rellich lemma.
Again, by the estimates in Eq. \eqref{term I mu 2},
$\{\two_\varepsilon\}$ is also pre-compact in $H^{-m}(\R^n, g_0; F)$.

Therefore, $\{\T v_\varepsilon\}$ is pre-compact in $H^{-m}(\R^n, g_0; F)$ so that
\begin{equation}\label{convergence of TV mu in H -m}
\T v_\varepsilon \longrightarrow 0  \qquad  \text{ strongly in } H^{-m}(\R^n, g_0; F).
\end{equation}
This is because $u_\varepsilon \weak 0$ in $L^2$ (see Step 1 above).
Here $\R^n$ is endowed with the Euclidean metric $g_0$, and $F$ has the bundle metric $g^F$.

\medskip
\noindent
{\bf 5.}
Now we estimate the Euclidean $L^2$ norm of $\widehat{\T v_\varepsilon}$ on $\{|\xi|\geq 1\}$,
where $\widehat{\T v_\varepsilon}$ is the standard Fourier transform on Euclidean spaces:
\begin{equation*}
\widehat{\T v_\varepsilon}(\xi):= \int_{\R^n} \T v_\varepsilon(x) e^{-2\pi i \xi \cdot x} \,\dd x.
\end{equation*}
Indeed, since $\T \in \diff^m(M; E,F)$,
by the localization and flattening in Steps 1--2, we have
\begin{equation*}
\T = \sum_{|\alpha|\leq m} a_\alpha(x) \p^\alpha_x,
\end{equation*}
and the principal symbol is given by
\begin{equation*}
\sigma_m(\T)(x,\xi) = (-2\pi i)^m \sum_{|\alpha|=m}a_\alpha(x) \xi^\alpha;
\end{equation*}
see
\S 3 in \cite{albin}.
Combining with the lower order terms, we have
\begin{equation*}
\widehat{\T v_\varepsilon}(\xi) = (-2\pi i)^m \sum_{|\alpha|=m}a_\alpha(x) \xi^\alpha\hat{v}_\varepsilon(\xi)
  + \sum_{|\beta| \leq m-1} b_{\beta}(x, \xi) \xi^\beta\hat{v}_\varepsilon(\xi),
\end{equation*}
where $|a_\alpha(x)|+|b_\beta(x,\xi)| \leq C_0$ for all $x\in M$ and $\xi \in T_x^*M$, and for each $\alpha$ and $\beta$.
Then
\begin{align*}
&\|\T v_\varepsilon\|^2_{H^{-m}(\{|\xi|\geq K\})}\\
&:=\int_{|\xi|\geq K} \frac{\big|(-2\pi i)^m \sum_{|\alpha|=m}a_\alpha(x) \xi^\alpha\hat{v}_\varepsilon(\xi)
   + \sum_{|\beta| \leq m-1} b_{\beta}(x, \xi) \xi^\beta\hat{v}_\varepsilon(\xi)\big|^2}{ (1+|\xi|^2 )^m} \,\dd\xi \\
&\geq C_1^{-1} \sum_{|\alpha| = m} \int_{|\xi|\geq K}  \frac{|a_\alpha(x)|^2|\xi|^{2m}}{(1+|\xi|^2)^m}|\hat{v}_\varepsilon(\xi)|^2 \,\dd\xi \\
&\qquad    - C_2 \sum_{|\gamma| \leq 2m-1}\int_{|\xi|\geq K} \frac{|\xi|^\gamma}{(1+|\xi|^2)^m} |\hat{v}_\varepsilon(\xi)|^2 \,\dd\xi,
\end{align*}
where $C_1$ depends only on $m$, while $C_2= C_2(\sup_x|b_\beta(x)|, m)$.
This is obtained by expanding the quadratic in the second line above and separating the highest order term from the other terms.
Now, choosing $K \geq 1$ so large that the second term is majorized by the first term in the last line, we have
\begin{equation*}
\|\T v_\varepsilon\|^2_{H^{-m}(\{|\xi|\geq K\})}
\geq C_3^{-1}\sum_{|\alpha| = m}
 \int_{|\xi|\geq K}  \frac{|a_\alpha(x)|^2|\xi|^{2m}}{(1+|\xi|^2)^m}|\hat{v}_\varepsilon(\xi)|^2 \,\dd\xi,
\end{equation*}
which converges to $0$ by Eq. \eqref{convergence of TV mu in H -m}, where $C_3$ depends on $C_1, C_2$, and $K$.

\medskip
\noindent
{\bf 6.} In this step, we complexify $Q$ and $\Lambda_\T$.
First, we view $Q: \Gamma(E) \rightarrow \R$ as a complex quadratic polynomial
$Q^\C : \Gamma(E^\C) \rightarrow \C$, given by the following expression in local coordinates:
\begin{equation*}
Q^\C(z):=Q^\C_x (z) := \sum_{j,k} Q_{jk}(x) z^j \overline{z^k} \qquad \text{ for $x \in M$ and $z\in E^\C_x$},
\end{equation*}
where $E^\C_x \cong \C^J$ is the fiber of the complexified bundle $E^\C:= E\otimes \C$
at point $x$.
Thus, $Q^\C (s) = Q(s)$ for real $s \in \Gamma(E)$.
Moreover, we define the {\em complexified cone} by
\begin{equation*}
\Lambda_\T^\C := \Lambda_\T + i\Lambda_\T= \{s^\C = s+is \,:\, s\in \Lambda_\T\}.
\end{equation*}

We now compute $Q^\C(\zeta)$ for $\zeta = s+ir$, where $s,r \in \Gamma(E)$ are real: Indeed,
\begin{equation*}
Q^\C(s+ir) = Q(s) + Q(r) + i \big\{\mathbf{q}(r,s) + \mathbf{q}(s,r)\big\},
\end{equation*}
where $\mathbf{q}(r,s):=\sum_{j,k}Q_{jk}(x) r^j\overline{s^k}$
and $Q(s)=\mathbf{q}(s,s)$ as before.
In particular, we have
\begin{equation*}
Q^\C(s^\C)= 2Q(s) +2iQ(s) =2Q(s)^\C,
\end{equation*}
so that, for $s^\C=s+is \in \Lambda_\T^\C$, the following facts hold:
\begin{enumerate}
\item[(i)] $Q(s)$ $>, =$, or $<0$ if and only if $ {\rm Re} \{Q^\C(s^\C)\}$ $>, =$, or $<0$ (respectively);

\smallskip
\item[(ii)]
$s \in \Lambda_\T$ if and only if $s^\C \in \Lambda^\C_\T$;

\smallskip
\item[(iii)] For any $\psi \in C^\infty_c(M)$ and $(u_\varepsilon)^\C:=u_\varepsilon+i u_\varepsilon$, we have
\begin{equation*}
\lim_{\varepsilon \rightarrow 0} \int_M(Q \circ u_\varepsilon) \psi \dvg= 0 \quad \iff
\quad
 \lim_{\varepsilon \rightarrow 0} \int_M {\rm Re}\{Q^\C \circ (u_\varepsilon)^\C\} \psi \dvg= 0.
\end{equation*}
\end{enumerate}

\smallskip
\noindent
{\bf 7.}
We first observe the following pointwise inequality: For each $\delta>0$ and any compact set $\K \Subset T^*M \setminus \{0\}$,
there is a constant $C_{\delta,\K}\in (0,\infty)$ such that
\begin{align}\label{inequality for Re Q, complexified}
&{\rm Re}\{Q^\C(s^\C)\} \geq -\delta \big| | s^\C|_{g^{E,\C}}\big|^2
 - C_{\delta, \K} \big| | \sigma_m(T)(\eta)(s^\C)|_{g^{F,\C}}\big|^2
\end{align}
for each $\eta \in \K$ and $s\in \Gamma(E)$,
provided that ${\rm Re}(Q) \geq 0$ on $\Lambda_\T$.
Here $g^{E,\C}$ is the {\em complexified bundle metric on $E$},
obtained according to the same rule for $Q \mapsto Q^\C$,
by viewing $g^E$ as a quadratic form on each fiber ({\it i.e.}, a vector space) of $E$; and similarly for $g^{F,\C}$.

Indeed, since Eq. \eqref{inequality for Re Q, complexified} is $2$-homogeneous in $s^\C$,
the scaling: $s^\C \mapsto \lambda s^\C$ by any $\lambda \in \C$ leaves it invariant.
In particular, it is independent of the signatures of the semi-Riemannian bundle metrics $g^E$ and $g^F$.
Moreover, cone $\Lambda_\T$ in (C3) is completely determined by $\T$,
which is independent of metrics $g$, $g^E$, and $g^F$, and sequences $\{u_\varepsilon\}$ and $\{v_\varepsilon\}$.
Thus, Eq. \eqref{inequality for Re Q, complexified} follows from a simple contradictory argument
as in Tartar's proof of the classical quadratic theorem \cite{tartar}.

We now integrate Eq. \eqref{inequality for Re Q, complexified} over $\{|\xi| \geq K\}$, with $K\geq 1$ specified
at the end of Step 5 above,
$s^\C=(\hat{v}_\varepsilon)^\C$, and $\eta:=\frac{\xi^{2m}}{(1+|\xi|^2)^m}$.
Then
\begin{equation*}
2^{-m} = \frac{|\xi|^{2m}}{(2|\xi|^2)^m}\leq |\eta| \leq 1 \qquad \text{ for all $|\xi|\geq K$}.
\end{equation*}
We remark here that it is crucial for sequence $\{v_\varepsilon\}$ to be taken on $M=\R^n$ with respect
to the Euclidean metric ({\it cf.}  Step $3$ above).
In this case, the metric induced on the cotangent bundle $T^*M$ is also Euclidean,
so that $|\xi| \neq 0$ for all $\xi \in T^*M\setminus \{0\}$.

To proceed, $\eta \in \K:= \{2^{-m} \leq |\xi|\leq 1\}$ is indeed a compact subset of $T^*M \setminus \{0\}$ so that
\begin{align*}
&\int_{|\xi|\geq K} {\rm Re}\{Q^\C({\hat{v}_\varepsilon})^\C\} \,\dd\xi  \nonumber\\
&\geq  -\delta \|v_\varepsilon\|^2_{L^2(\R^N; E)} - C_{\delta, \K} \bigg\{\sum_{|\alpha| = m}
\int_{|\xi|\geq K}  \frac{|a_\alpha(x)|^2|\xi|^{2m}}{(1+|\xi|^2)^m}|\hat{v}_\varepsilon(\xi)|^2 \,\dd\xi\bigg\},
\end{align*}
where the last term on the right-hand side tends to zero as $\varepsilon\rightarrow 0$ ({\it cf.} Step $5$).
Therefore, we have
$$
\lim_{\varepsilon\to 0}\int_{|\xi|\geq K} {\rm Re}\{Q^\C({\hat{v}_\varepsilon})^\C\} \,\dd\xi
\geq  - C_0\delta
\qquad\,\, \mbox{for arbitrary $\delta>0$},
$$
where $C_0=\sup_{\varepsilon>0} \|v_\varepsilon\|^2_{L^2(\R^n; E)}=\sup_{\varepsilon>0} \|u_\varepsilon\|^2_{L^2(\mathbb{R}^n,\,g; E)}<\infty$.
This implies that the left-hand side is non-negative.
Applying the same argument for $-Q$ in place of $Q$,
thanks to condition (C3) and Step $6$ above,
we finally obtain
\begin{equation*}
\lim_{\varepsilon \to 0} \int_{|\xi|\geq K} {\rm Re}\,\{Q^\C \circ {\hat{v}_\varepsilon}^\C\}(\xi)\,\dd \xi = 0,
\end{equation*}
that is,
\begin{equation}\label{high-frequency region for sequence v}
\lim_{\varepsilon \to 0} \int_{|\xi|\geq K} {\rm Re}\,\{Q ({\hat{v}_\varepsilon})\}(\xi)\,\dd \xi = 0.
\end{equation}

\smallskip
\noindent
{\bf 8.}
Now we combine \eqref{low-frequency region for sequence v} with \eqref{high-frequency region for sequence v}
and employ the Plancherel formula to conclude
\begin{align}
&\lim_{\varepsilon \to 0} \Big|\int_{\R^n} {\rm Re}\{Q(\hat{v}_\varepsilon (\xi))\}\,\dd \xi\Big|\nonumber\\
&\le \lim_{\varepsilon \to 0}\int_{|\xi| < K} \big|{\rm Re}\{Q(\hat{v}_\varepsilon (\xi))\}\big|\,\dd \xi
   +\lim_{\varepsilon\to 0}\Big|\int_{|\xi| \geq K} {\rm Re}\{Q(\hat{v}_\varepsilon (\xi))\}\,\dd \xi\Big|\nonumber\\
&\leq C\lim_{\varepsilon \to 0}\int_{|\xi| < K} |\hat{v}_\varepsilon(\xi)|^2\,\dd \xi=0,\label{plancherel for v mu}
\end{align}
for some constant $C>0$ independent of $\varepsilon>0$.
Then we infer from the Plancherel
formula that
$$
\lim_{\varepsilon \to 0} \int_{\R^n} (Q\circ{v}_\varepsilon) (x)\,\dd x = 0.
$$
Also, recall from Equation \eqref{def of sequence v-tilde} that
$v_\varepsilon$ differs from $u_\varepsilon$ by a multiplicative factor depending only on the $L^\infty$ norms of metrics on $M$ and $E$ (independent of $\varepsilon$).
As $Q$ is quadratic, we thus deduce
\begin{align*}
\lim_{\varepsilon \to 0} \int_M (Q\circ u_\epsilon)(x)\,\dd V_g = 0.
\end{align*}

Moreover, we recall from Step $1$ that the assertion of Theorem \ref{thm: generalized quadratic theorem on manifolds}
is invariant under localizations, {\it i.e.},  multiplication by test functions $\psi \in C^\infty_c(M)$.
Therefore, we can now conclude that $\{Q\circ u_\varepsilon\}$ converges
to $Q\circ u$ in the sense of distributions. This completes the proof.

\medskip
We emphasize that the non-degeneracy condition of metric, $\det \, g \neq 0$,
is crucial to the proof. We need it in Eq. \eqref{term I mu 2} to compare
the $H^{-m}$ norms of $\T\hat{u}^{\varepsilon}$ taken with respect to $g$ and the Euclidean metric $g_0$.
Therefore, we can extend Theorem \ref{thm: generalized quadratic theorem on manifolds} to a more general
theorem, Theorem \ref{thm: generalized quadratic theorem on manifolds-b} below,
for non-smooth metrics $g$, $g^E$, and $g^F$, which is crucial to the development in \S 4.
Notice that, in the proof of Theorem \ref{thm: generalized quadratic theorem on manifolds},
only the $L^\infty_\loc$ topology
of the metrics are involved in the estimates. Thus, in view of the Morrey--Sobolev embedding,
the following result holds by an approximation argument:

\begin{theorem}\label{thm: generalized quadratic theorem on manifolds-b}
Let $M$ be a semi-Riemannian  manifold with a non-degenerate $L^\infty_{\loc}$
metric $g$ {\rm (}{\it i.e.}, $|\det \,g| \geq \eta_0 >0$ a.e.{\rm )}. Let $E$ and $F$ be two real vector bundles
over $M$ with $L^\infty_\loc$ bundle metrics $g^E$ and $g^F$, respectively.
Consider a sequence of $E$-sections $\{u_\varepsilon\} \subset L^2(M; E)$,
a differential operator $\T \in \diff^m(M; E, F)$ for
some $m \in \R_{+}$ with the principal symbol $\sigma_m(\T): T^*M \rightarrow {\rm Hom}(E; F^\C)$,
and a real quadratic polynomial $Q : \Gamma(E) \rightarrow \R$.
If the following conditions hold{\rm :}
\begin{enumerate}
\item[{\rm (C-1)}]
$u_\varepsilon \weak u$ weakly in $L^2_{\rm loc}(M;E)${\rm ,}

\smallskip
\item[{\rm (C-2)}]
$\{{\T u_\varepsilon}\}$ is pre-compact in $H^{-m}_{\rm loc}(M; F)${\rm ,}

\smallskip
\item[{\em (C-3)}]
$Q(s)=0$ for all $s \in \Lambda_\T$, where the {\em cone} of $\T$ is defined by
\begin{equation*}
\Lambda_\T:=\big\{s \in \Gamma(E)\,:\, \sigma_m(\T)(\xi)(s) = 0
 \text{ for some $\xi \in T^*M \setminus \{0\}$}\big\},
\end{equation*}
\end{enumerate}
then
\begin{equation*}
\lim_{\varepsilon\rightarrow 0} \int_M  (Q\circ u_\varepsilon)\psi\, \dvg
= \int_M  (Q\circ u) \psi\,\dvg \qquad \text{ for any } \psi \in C^\infty_c(M).
\end{equation*}
\end{theorem}

\medskip
To conclude this section, besides the geometric theorem,
Theorem~{\rm \ref{thm: generalized quadratic theorem on manifolds-b}},
we can also obtain a generalized compensated compactness theorem in the abstract harmonic analysis
settings.
Although this result is not needed for our weak continuity
theorem (Theorem~\ref{thm: weak continuity of Cartan's structural eq}) for the Cartan structural system below,
it is of independent interest from the perspectives of compensated compactness and harmonic analysis.
In addition, it may help to elucidate certain steps
in the lengthy proof of Theorem~\ref{thm: generalized quadratic theorem on manifolds}
that leads to Theorem~\ref{thm: generalized quadratic theorem on manifolds-b} above.

We first recall some basics of abstract harmonic analysis $(${\it cf}. Loomis \cite{loomis} and the notes by Tao \cite{Tao}$)$.
A {\em topological group} $G$ is a group with a topology, in which the group operation and the inverse are continuous.
If a group $G$ is Abelian whose topology is Hausdorff and locally compact,
we say  that $G$ is a {\em locally compact Abelian group}, abbreviated as {\em LCA group} in the sequel.
For any LCA group $G$, there exists an invariant Radon measure $\mu_G$, unique up to multiplicative constants,
known as the {\em Haar measure}.
The $L^p$ norm, $1\leq p < \infty$, for a function $u: G \rightarrow \C$ can then be defined as
\begin{equation*}
\|u\|_{L^p(G)}:=\Big(\int_G |u(g)|^p \,\dd\mu_G(g)\Big)^{{1}/{p}}.
\end{equation*}

Given any LCA group $G$, its {\em group of characters}, $\hg:= {\rm Hom}(G; \R / \mathbb{Z})$, is
also an LCA group endowed with the local-uniform topology of any non-trivial Haar measure (which is the weakest topology making
each element of $\hg$ continuous).
It is also known as the {\em dual} of $G$, due to the {{\em Pontryagin duality theorem}}:
$G$ is canonically isomorphic to $\hat{\hat{G}}$.
Then, for $u \in L^1(G)$, we can define its {\em Fourier transform} $\hat{u}:\hg \rightarrow \C$ by
\begin{equation}\label{eq: Fourier transform for LCA gps}
\hat{u}(\xi):= \int_G u(g) e^{-2\pi i \xi(g)} \,\dd\mu_G (g),
\end{equation}
where $\xi(g)$ is given by the duality pairing of $\hg$ and $G$.
From now on, we write $0 \in \hg$ as the group identity;
this is in agreement with the definition, $\hg:= {\rm Hom}(G; \R / \mathbb{Z})$,
which is the group of {\em additive} (not multiplicative) characters.

Next, the {\em Plancherel formula} extends to the general LCA groups:
\begin{equation*}
\|u\|_{L^2(G)} = \|\hat{u}\|_{L^2(\hat{G})} \qquad \text{ for all $u \in L^2(G)$},
\end{equation*}
with the Haar measures $\mu_G$ and $\mu_{\hg}$ suitably normalized.
In other words, the Fourier transform defined in Eq. \eqref{eq: Fourier transform for LCA gps}
is an isometry between $L^2(G)$ and $L^2(\hg)$.
Notice that all the constructions up to now can  naturally be extended to vector-valued functions
$u: G \rightarrow \C^I$ for $I \geq 1$.

Finally, we say that
$\T: L^2(G) \rightarrow L^2(G)$ is a {\em multiplier operator} if
\begin{equation*}
\widehat{\T u}(\xi) = m(\xi) \hat{u}(\xi)\qquad \text{ for some } m: \hat{G} \rightarrow \C,
\end{equation*}
where $m$ is known as the {\em Fourier multiplier} of $\T$.
More generally, for $\T: L^2(G; \C^J) \to L^2(G;\C^I)$ for $I,J \geq 1$,
the multiplier is a mapping
\begin{equation*}
m: \hg \rightarrow \matij \cong (\C^J)^\ast \otimes \C^I.
\end{equation*}
That is, for each $\xi \in \hg$, $m(\xi)$ is a linear operator from $\C^J$ to $\C^I$
(equivalently, an $I\times J$ matrix).
In the sequel, for any matrix $M \in \matij$,
we use $|M|:=\sqrt{\sum_{i=1}^I\sum_{j=1}^J |M_{ij}|^2}$ to denote its Hilbert--Schmidt norm.

In this context, we say that $Q: \C^N \rightarrow \C$ is a  {\em quadratic polynomial}
if it is a Hermitian $2$-form on $\C^N$, {\it i.e.}, $Q=\{Q_{jk}\}$
as a complex $N\times N$ matrix satisfies
\begin{equation*}
Q_{jk} = \overline{Q_{kj}} \qquad \text{ for each }j,k=1,2,\ldots, N.
\end{equation*}
That is,
\begin{equation}\label{eq: def for quadratic polynomial on flat space}
	Q(\lambda) = \sum_{j,k=1}^{N} Q_{jk} \lambda^j \overline{\lambda^k}
\quad \text{for $\lambda=(\lambda^1, \ldots,\lambda^N) \in \C^N$ and constants $Q_{jk} \in \C$}.
\end{equation}

\begin{theorem}\label{thm:3.5}
Let $G$ be an LCA group with Haar measure $\mu_G$.
Consider a sequence $\{u_\varepsilon\}$ in $L^2_c(G; \C^J)$,
a Fourier multiplier operator $\T: L^2(G; \C^J) \rightarrow H^{-s}(G; \C^I)$
with multiplier $m: \hat{G} \rightarrow \matij$ for some $s \in \R_{+}$,
and a quadratic polynomial $Q: \C^J \rightarrow \C$.
Assume that
\begin{enumerate}
\item[{\em (i)}]
$u_\varepsilon \weak u$ weakly in $L^2(G; \C^J)$.

\smallskip
\item[{\em (ii)}]
The end of $\hg$ retracts nicely onto a compact set.
More precisely, for some
compact set $\Xi \Subset \hg$ containing $0$,
there exist another compact set $\K \Subset \hg \setminus \{0\}$ and a continuous surjective map $\Phi: \hg\setminus \Xi \rightarrow \K$
such that $\{(\Phi^*m) \hat{u_\varepsilon}\}$ is pre-compact in $L^2(\hg\setminus \Xi; \C^J)$.

\smallskip
\item[{\em (iii)}]
$Q(\lambda) = 0$ for all $\lambda \in \Lambda_\T$, where $\Lambda_\T$ {\rm (}the {\em cone} of $\T${\rm )} is defined by
\begin{equation}\label{def cone of T, LCA group}
\Lambda_\T:=\big\{\lambda \in \C^J\,:\, m(\xi)(\lambda)=0 \,\text{ for some } \xi \in \hat{G}\setminus \{0\}\big\}.
\end{equation}
\end{enumerate}
Then
\begin{equation*}
\lim_{\varepsilon \rightarrow 0}\int_{G}  (Q \circ u_\varepsilon)(g) \,\dd\mu_G(g) = \int_{G}  (Q \circ u)(g) \,\dd\mu_G(g).
\end{equation*}
\end{theorem}

In Theorem \ref{thm:3.5} above, the {\em pullback} of $m$ under $\Phi$, {\it i.e.},
 $\Phi^*m: \hg \setminus \Xi \rightarrow [0,\infty]$, is given by
$\Phi^*m (\xi):= m(\Phi(\xi))$.
In the definition of $\Lambda_\T$ in \eqref{def cone of T, LCA group}, we view
$
m: \hg \rightarrow (\C^J)^\ast \otimes \C^I.
$
That is, $m(\xi)$ is an operator from $\C^J$ to $\C^I$ so that $m(\xi)(\lambda) \in \C^I$.
According to this interpretation, another characterization of the cone is
\begin{equation*}
\Lambda_\T =  \bigcup_{\xi \in \hg \setminus \{0\}} \ker [m(\xi)].
\end{equation*}

The proof of Theorem \ref{thm:3.5} can be found in Appendix B.

\smallskip
\section{\, Global Weak Continuity of the Cartan Structural System}\label{sec: weak continuity of Cartan's structural eq}
\label{sec: weak continuity of cartan}

In this section, we establish the weak continuity of the Cartan structural system
\eqref{eqn_second structural eqn} on semi-Riemannian manifolds.
The arguments are global and intrinsic,
based on the geometric compensated compactness theorem, Theorem \ref{thm: generalized quadratic theorem on manifolds-b}.
This extends our earlier results on the weak continuity of the GCR system on
Riemannian manifolds \cite{chenli,csw2}.

\begin{theorem}\label{thm: weak continuity of Cartan's structural eq}
Let $(M,g)$ be a semi-Riemannian manifold of dimension $n$,
with $\ind(M)=\nu$, $g\in L^\infty_\loc$, and the Levi--Civita connection $\na$ of $g$ in $L^p_\loc$ for $p>2$.
Assume that a family of connection $1$-forms $\{\W_\varepsilon\}$ with the same index is uniformly bounded in $L^p_\loc$
and that each $\W_\varepsilon$ satisfies the Cartan  structural system \eqref{eqn_second structural eqn}
in the sense of distributions.
Then, after passing to a subsequence if necessary,
$\W_\varepsilon$ converges weakly in $L^p_\loc$ to a connection $1$-form $\W$ that
also satisfies system \eqref{eqn_second structural eqn}.
\end{theorem}

By ``$\{\W_\varepsilon\}$ with the same index'' we mean that there are fixed positive integers $k$ and $\tau$
such that, for each $\varepsilon$,
$$
\W_\varepsilon \in L^p_\loc (M;\, T^*M \otimes \mathfrak{o}(\nu + \tau, (n+k)-(\nu+\tau))).
$$
That is, $\{\W_\varepsilon\}$ arises from isometric immersions of $M$ into
a fixed semi-Euclidean space $\R^{n+k}_{\nu+\tau}$.

\bigskip
\noindent
{\bf Proof of Theorem {\rm \ref{thm: weak continuity of Cartan's structural eq}}}.
Our goal is to pass to the limit in the system:
\begin{equation}\label{eq_dW epsilon}
{\rm d}\W_\varepsilon = \W_\varepsilon \wedge \W_\varepsilon.
\end{equation}
We divide the proof into four steps.
Throughout the proof, we write
$$
\mathfrak{h}:=\littleonk.
$$

\medskip
\noindent
{\bf 1.}
Take an arbitrary test differential form $\varphi \in C^\infty_c(M\,;\,\wedge^{n-2}T^*M)$.
Then
\begin{equation}\label{eq: W wedge W wedge phi}
\dd\W_\varepsilon \wedge \varphi = \W_\varepsilon \wedge (\W_\varepsilon \wedge \varphi)
= \star\langle \star \W_\varepsilon, \W_\varepsilon \wedge \varphi \rangle,
\end{equation}
where $\star: \wedge^j T^*M \rightarrow \wedge^{n-j}T^*M$ is the {\em Hodge star} operator
(a vector bundle isomorphism), and
$\varphi$ has no $\mathfrak{h}$-component.
In the rest of the proof, we also use $\star$ to denote its natural
extension $\star : \wedge^j T^*M \otimes \h \rightarrow \wedge^{n-j}T^*M \otimes \h$,
given by $\star(\omega \otimes A):=\star\omega \otimes A$ for $\omega \in \wedge^j T^*M$ and $A \in \h$.
In other words, we do not distinguish between $\star$ and $\star \otimes {\rm id}_\h$.

\medskip
\noindent
{\bf 2.} We now determine the differential constraints of Eq. \eqref{eq: W wedge W wedge phi}.

We start from the left-hand side.
Notice that $\dd\W_\varepsilon=\W_\varepsilon\wedge \W_\varepsilon$ with
$$
\W_\varepsilon \wedge \W_\varepsilon \in L^{\frac{p}{2}}_{\loc}(U; {\wedge}^2\, T^*M \otimes \h).
$$
Recall the following compact Sobolev embedding: If $p<2n$,
\begin{equation*}
L^{\frac{p}{2}}_\loc(U; {\wedge}^2\, T^*M \otimes \h) \emb W^{-1,q}_\loc(U; {\wedge}^2\, T^*M \otimes \h)
\qquad \text{ for any } q<\frac{pn}{2n-p}.
\end{equation*}
On the other hand, if $p\geq 2n$, we can first embed
\begin{equation*}
L^{\frac{p}{2}}_\loc (U; {\wedge}^2\, T^*M \otimes \h) \to
L^{\frac{\hat{p}}{2}}_\loc(U; {\wedge}^2\, T^*M \otimes \h)
\qquad \text{ for $2<\hat{p}<2n$},
\end{equation*}
and then compactly embed the right-hand side into $W^{-1,q}_\loc$.
Thus,  $\{\dd\W_\varepsilon\}$ is pre-compact in $W^{-1,q}_\loc(U; \wedge^2T^*M \otimes \h)$
for some $1<q<2$.
On the other hand, the Rellich lemma implies that $\{\dd\W_\varepsilon\}$ is pre-compact
in  $W^{-1,p}_\loc(U; \wedge^2T^*M \otimes \h)$
for $p>  2$.
By interpolation, we find that
$$
\{\dd \W_\varepsilon\} \qquad\text{is pre-compact in  $H^{-1}_\loc(U; \wedge^2T^*M \otimes \h)$}.
$$
Owing to the super-commutativity of $\dd$, we have
\begin{equation*}
\dd(\W_\varepsilon \wedge \varphi) =\dd\W_\varepsilon \wedge \varphi - \W_\varepsilon \wedge \dd\varphi.
\end{equation*}
Therefore, we conclude
\begin{equation}\label{diff constraint: d}
\big\{\dd(\W_\varepsilon \wedge \varphi)\big\} \qquad \text{is pre-compact in $H^{-1}_\loc(U; {\wedge}^2\, T^*M \otimes \h)$}.
\end{equation}

Next, consider the rightmost side of Eq. \eqref{eq: W wedge W wedge phi}.
Recall that the $L^2$-adjoint of $\dd$ (the co-differential),
denoted by $\delta: \wedge^j T^*M \rightarrow \wedge^{j-1} T^*M$ for $1 \leq j \leq n$,
is related to $\dd$ by
\begin{equation*}
\delta = (-1)^{j(n-j)+1} \star \dd \star.
\end{equation*}
The Hodge star extends to an isometric isomorphism
$$
\star\,:\,L^q(U; \wedge^j T^*M)\map  L^q(U; \wedge^{n-j} T^*M)
\qquad\mbox{for each $0\leq j \leq n$}.
$$
For $M$ with signature $\nu$,
\begin{equation*}
\star\star=(-1)^{j(n-j)+\nu} \, {\rm id}_{\wedge^jT^*M},
\end{equation*}
where ${\rm id}$ denotes the identity map.
Then we have obtained another differential constraint:
\begin{equation}\label{diff constraint: delta}
\{\delta \star \W_\varepsilon\} \qquad \text{ is pre-compact in } H^{-1}_\loc (U; \h).
\end{equation}

\medskip
\noindent
{\bf 3.} In view of the arguments in Step 2 above, especially Eqs. \eqref{diff constraint: d}--\eqref{diff constraint: delta},
it suffices to establish the following {\em claim}, which is of generality:

\medskip
\noindent
{\it Claim{\rm :} \, Let $\{V_\varepsilon\}$ be a family of $(n-1)$-forms so that $\{\dd V_\varepsilon\}$ is pre-compact in $H^{-1}_\loc$,
and let $\{Z_\varepsilon\}$ be a family of $(n-1)$-forms so that $\{\delta Z_\varepsilon\}$ is pre-compact in $H^{-1}_\loc$.
Assume that $V_\varepsilon \weak V$ and $Z_\varepsilon \weak Z$ weakly in $L^p_\loc$.
Then $\{\langle V_\varepsilon, Z_\varepsilon\rangle\}$ converges to $\langle {V}, {Z}\rangle$ in the sense of distributions.}
	
\smallskip
Indeed, if the claim is true, we define
	\begin{equation*}
	\begin{cases}
	V_\varepsilon := \W_\varepsilon \wedge \varphi \in L^p_\loc(U; \wedge^{n-1}T^*M\otimes\h),\\[2mm]
	Z_\varepsilon:=\star \W_\varepsilon \in L^p_\loc(U; \wedge^{n-1}T^*M\otimes \h).
	\end{cases}
	\end{equation*}	
The above {\em claim} implies that
$\langle \W_\varepsilon \wedge \varphi, \star \W_\varepsilon\rangle \to \langle {\W}\wedge \varphi, \star {\W}\rangle$
in the sense of distributions.
Using the identities of the Hodge star and the super-commutativity of the wedge product, we deduce
\begin{align*}
\langle \W_\varepsilon \wedge \varphi, \star \W_\varepsilon\rangle &= (-1)^\nu \W_\varepsilon \wedge \varphi \wedge \star \star \W_\varepsilon \nonumber\\
&= (-1)^{2\nu + n-1} \W_\varepsilon \wedge \varphi \wedge \W_\varepsilon \nonumber\\
&= (-1)^{2\nu + 2(n-1)} \W_\varepsilon \wedge \W_\varepsilon \wedge \varphi \nonumber\\
&= \W_\varepsilon \wedge \W_\varepsilon \wedge \varphi.
\end{align*}
Therefore, the previous convergence result is equivalent to the following:
\begin{equation*}
\W_\varepsilon \wedge \W_\varepsilon \wedge \varphi \longrightarrow {\W} \wedge {\W} \wedge \varphi
\qquad \text{ in the sense of distributions}.
\end{equation*}
Since the test form $\varphi$ is arbitrary, the proof is now complete.

\medskip
\noindent
{\bf 4.} We now prove the {\em claim} in Step $3$ by making crucial use
of Theorem \ref{thm: generalized quadratic theorem on manifolds-b}.
The key is to specify operator $\T$ and the vector bundles $E$ and $F$ therein.

Indeed, we define
\begin{eqnarray*}
&&E:=  \big({\wedge}^{n-1}T^*M \otimes \h\big)\oplus \big({\wedge}^{n-1}T^*M \otimes \h\big),\\
&&F:=  \big({\wedge}^{n}T^*M \otimes \h\big)\oplus \big({\wedge}^{n-2}T^*M \otimes \h\big),\\
&& \T:= \dd \oplus \delta,
\end{eqnarray*}
where $\T$ is a bundle operator $\T\,:\,E\rightarrow F$.
In this setting, the operator cone is given by
\begin{align*}
\Lambda_{\T} = \left\{(\mu, \lambda)^\top\in \Gamma(E)\, :
\begin{array}{ll}
[\sigma_1(\dd)(\xi)](\mu)=0 \text{ and } [\sigma_1(\delta)(\xi)](\lambda)=0 \\
\text{ simultaneously for some $\xi \in T^*M \setminus \{0\}$}
\end{array}
\right\},
\end{align*}
where we have utilized
$$
\sigma_1(\dd\oplus \delta) = \sigma_1(\dd) \oplus \sigma_1(\delta).
$$
It is an identity on $\mathcal{P}_1 (T^*M; {\rm Hom}(E; F^\C))$,
{\it i.e.}, the space of first-order homogeneous polynomials
that map the cotangent bundle to the homomorphism bundle from $E$ to $F^\C$.

We can further specify $\Lambda_\T$.
Indeed, recall that the principal symbols of $\dd$ and $\delta$ have global intrinsic
representations ({\it cf.} \S 3.1, \cite{albin}):
\begin{equation*}
\sigma_1(\dd)(\xi) = -(2\pi i) \xi \wedge, \qquad \sigma_1(\delta)(\xi) = (2\pi i ) \iota_{\xi^\sharp},
\end{equation*}
where $\xi^\sharp$ is the element of the tangent bundle $TM$ canonically isomorphic to $\xi$
(which can be obtained by raising the indices in the local coordinates),
and $\iota_X$ is the {\em interior multiplication} of a differential form
by the vector field $X\in \Gamma(TM)$. Then
\begin{align}\label{revise, 2}
\Lambda_{\T} = \left\{(\mu, \lambda)^\top\in \Gamma(E)\,:
\begin{array}{ll}
\xi \wedge \mu  = 0 \text{ and } \iota_{\xi^\sharp}(\lambda) =0 \text{ simultaneously}\\
\text{for some } \xi \in T^*M \setminus \{0\}
\end{array}
\right\}.
\end{align}
Notice that $\xi \wedge \mu =0$ if and only if
$\mu = (\xi \wedge \tilde{\mu}) \otimes A$ for some $A \in \h$ and $\tilde{\mu}\in \wedge^{n-2}T^*M$.
Also, $\iota_{\xi^\sharp}(\lambda)=0$ if and only if $\{\tilde{\lambda}, \xi\}$ span
an orthogonal subspace in $T^*M$ so that $\lambda=\tilde{\lambda}\otimes B$ for $B \in \h$.

Now, define the quadratic polynomial $Q: \Gamma(E)\rightarrow \R$ by
\begin{equation*}
Q((\mu, \lambda)^\top) := \langle \mu, \lambda \rangle.
\end{equation*}
The bracket, $\langle\cdot,\,\cdot\rangle$, on the right-hand side is the combination
of the inner product on $\wedge^{n-1}T^*M$ and the matrix product on $\h$.
Thus, for $(\mu,\lambda)^\top \in \Lambda_{\T}$, we have
\begin{align*}
Q((\mu,\lambda)^\top ) = \langle (\xi \wedge \tilde{\mu})\otimes A, \tilde{\lambda} \otimes B \rangle
=\langle \xi \wedge \tilde{\mu}, \tilde{\lambda}\rangle \otimes (A \cdot B),
\end{align*}
where $\cdot$ denotes the matrix multiplication.

Then $\langle \xi \wedge \tilde{\mu}, \tilde{\lambda}\rangle = 0$.
Indeed, recall that the dot product
$\langle\cdot, \, \cdot\rangle$
on $\wedge^{n-1}T^*M$
is induced from the inner product on $T^*M$ by the following rule:
For two $(n-1)$-tuples of basic elements in the cotangent bundle $T^*M$:
$\{\theta^{i_1}, \ldots, \theta^{i_{n-1}}\}$ and $\{\theta^{j_1}, \ldots, \theta^{j_{n-1}}\}$,
define
\begin{equation}\label{revise, 1}
\big\langle\theta^{i_1} \wedge \ldots\wedge \theta^{i_{n-1}}, \theta^{j_1} \wedge \ldots\wedge \theta^{j_{n-1}}\big\rangle
:= \det \big(\langle \theta^{i_k}, \theta^{j_l}\rangle_{1\leq k, l \leq n-1}\big).
\end{equation}
In particular, if some $\theta^{i_k}$ is orthogonal to $\theta^{j_l}$ in $T^*M$,
then the right-hand side of Eq.  \eqref{revise, 1} vanishes.
By Eq. \eqref{revise, 2} and the ensuing remark,
$\xi$ and $\tilde{\lambda}$ are orthogonal, so that $\langle \xi \wedge \tilde{\mu}, \tilde{\lambda}\rangle=0$.
In effect, we have checked the hypotheses on the operator cone in
Theorem \ref{thm: generalized quadratic theorem on manifolds-b};
that is, the quadratic polynomial $Q$  vanishes on cone $\Lambda_{\T}$.

In view of the above arguments, conditions (C-1)--(C-3)
in Theorem \ref{thm: generalized quadratic theorem on manifolds-b} are verified.
Applying this theorem,
we obtain
\begin{equation*}
Q((V_\varepsilon, Z_\varepsilon)^\top)
:= \langle \W_\varepsilon \wedge \varphi, \star \W_\varepsilon\rangle \longrightarrow \langle \W\wedge \varphi, \star \W\rangle
=: Q((V,Z)^\top)
\end{equation*}
in the sense of distributions.
Then the {\em claim} follows, so that the theorem is proved.

\medskip
The equivalence between the Cartan structural system and
the GCR system (Proposition \ref{prop_second structural equation})
implies the weak continuity of the GCR system:

\begin{theorem}\label{cor: weak continuity of GCR}
Let $(M,g)$ be a semi-Riemannian manifold of dimension $n$ with $\ind(M)=\nu$, $g\in L^\infty_{\loc} (M,  O(\nu, n-\nu))$,
and the Levi--Civita connection $\na$ of $g$ in $L^p_\loc$ for $p>2$.
Assume that a family of second fundamental forms and normal affine connections $\{(\two_\varepsilon, \na^{\perp}_{\varepsilon})\}$
is uniformly bounded in $L^p_\loc$, and  each $(\two_\varepsilon, \na^{\perp}_{\varepsilon})$ satisfies the GCR
system \eqref{gauss}--\eqref{ricci} in the sense of distributions.
Then, after passing to a subsequence if necessary,
$\{(\two_\varepsilon, \na^{\perp}_{\varepsilon})\}$ converges weakly in $L^p_\loc$ to $(\two, \na^\perp)$
that also satisfies Eqs.\, \eqref{gauss}--\eqref{ricci}.
\end{theorem}

As remarked in the introduction, \S 1, the weak continuity of the Cartan structural
and GCR systems (Theorems~\ref{thm: weak continuity of Cartan's structural eq}--\ref{cor: weak continuity of GCR})
may alternatively be proved by using the compensated compactness theorems in the Euclidean spaces.
For example,
the following ``generalized div-curl lemma''
for wedge products was established as Theorem~1.1 in Robbin--Rogers--Temple \cite{rrt}:

\begin{quote}
\emph{
Let $\alpha_\varepsilon \weak {\alpha}$ in $L^p_\loc(\R^n)$ and let $\beta_\varepsilon \weak {\beta}$ in $L^{p'}_\loc(\R^n)$,
where $\{\alpha_\varepsilon\}, \{\beta_\varepsilon\}, {\alpha}$, and ${\beta}$ are differential forms over $\R^n$
and $\frac{1}{p}+\frac{1}{p'}=1$. Assume that $\{\dd\alpha_\varepsilon\} \subset W^{-1,p}_\loc(\R^n; T^*\R^n)$
and $\{\dd\beta_\varepsilon\} \subset W^{-1,p'}_\loc(\R^n; T^*\R^n)$ are pre-compact.
Then $\alpha_\varepsilon \wedge \beta_\varepsilon\, \map\, {\alpha} \wedge {\beta}$ in the sense of distributions.}
\end{quote}

\noindent
One may apply the above result to deduce Theorem~\ref{thm: weak continuity of Cartan's structural eq} by
computing in local coordinates and adapting the arguments in Chen--Slemrod--Wang \cite{csw2}.
On the other hand,
independent of the goal of proving the $W^{2,p}$ continuity of the GCR and Cartan structural systems,
we comment that an extension for the above theorem
in $\R^n$ to semi-Euclidean
spaces (or more generally, to semi-Riemannian manifolds)
appears elusive. It does \emph{not} follow from direct adaptations of the arguments in \cite{rrt}.
Indeed, the proof of \cite[Theorem 1.1]{rrt}
relies crucially on the ellipticity of the Laplace--Beltrami operator, for which
the following arguments beneath \cite[Eq. (4.26), page 616]{rrt} are central:

\begin{quote}
\emph{
From the continuity of $\Delta^{-1}$ from $W^{-1,p}(\Omega)$ to $W^{1,p}(\Omega)$\footnote{There is a typo in \cite{rrt}:
the second $W^{-1,p}(\Omega)$ therein should be $W^{1,p}(\Omega)$.}, we conclude
\begin{equation*}
\Delta^{-1} {\bf d}\alpha_\varepsilon \in \text{ a compact set in $W^{1,p}(\Omega)$}.
\end{equation*}
}
\end{quote}

\noindent
However, the Laplace--Beltrami operator $\Delta$ on a semi-Riemannian manifold is never elliptic,
unless the manifold is Riemannian, so that the arguments in \cite{rrt} cannot pass through in the semi-Riemannian setting.

To conclude this section, we note that the weak continuity of the GCR and Cartan structural
systems (Theorems \ref{thm: weak continuity of Cartan's structural eq}--\ref{cor: weak continuity of GCR})
does not require any assumption on the topology of $(M,g)$.

\section{\, Realization Theorem: From the Cartan Structural Systems to Isometric Immersions
of Semi-Riemannian Manifolds}\label{sec: realisation}

In this section, we address the following problem:

\begin{quotation}
{\em Given an $n$-dimensional semi-Riemannian manifold $(M,g)$ of lower regularity satisfying the GCR system
$(${\it cf.} Theorem {\rm \ref{theorem GCR equations}}$)$ in the sense of distributions,
seek an isometric immersion $f: (M,g) \emb (\R^{n+k}, g_0)$ with the semi-Euclidean metric $g_0$.}
\end{quotation}
We refer to it as the {\em realization problem} ---  Given a weak solution $(\two, \na^\perp)$
to the compatibility equations, we would like to realize it as the geometric data of an isometric immersion.

For a Riemannian manifold $M$, the realization problem is settled in the affirmative
if $M$ is simply-connected.
The $C^\infty$ case was proved by Tenenblat \cite{Ten71},
and the $W^{2,p}_{\rm loc}$ case for $p>\dim(M)$ by Mardare \cite{Mar05,Mar07}
and Szopos \cite{szopos}.
In \cite{chenli}, we also provided a geometric and intrinsic proof.
Although the realization problem for semi-Riemannian manifolds is viewed as a ``folklore theorem'' ({\it cf.}
Chen \cite{chenby}), we still find it necessary and non-trivial to give a detailed proof.
Indeed, new ideas are required in the following two main points:

\begin{enumerate}
\item[\rm (i)] the interplay of Cartan's formalism and semi-Riemannian geometry,

\smallskip
\item[\rm (ii)] the treatment of manifolds of lower regularity.
\end{enumerate}

\subsection{\, Statement of the Realization Theorem}\label{subsec: statement of realisation}

First of all, we note that the following two conditions are necessary for the realization problem:
\begin{enumerate}
\item[{\rm (R1)}]
The resulting map $f$ must be an immersion of $M$ as a semi-Riemannian submanifold{\rm ;}

\smallskip
\item[{\rm (R2)}]
The indices of manifold $M$ and its normal bundle $TM^\perp = f^*T{\R^{n+k}}/TM$ (see Convention \ref{convention on iota})
add up to the index of the target space:
\begin{equation*}
\ind (M) + \ind (T_xM^\perp) = \ind (\R^{n+k}) \qquad \text{ for each } x\in M.
\end{equation*}
\end{enumerate}

Indeed, condition (R1) holds since $f$ is an isometry ($f^*g_0 = g$),
and a semi-Riemannian metric is non-degenerate by definition.
For example, it rules out the possibility that a semi-Riemannian manifold
is isometrically embedded into the lightcone of the Minkowski spaces.
Condition (R2) is a consequence of (R1) together with the direct
sum decomposition in Eq.\,\eqref{eqn decomposition of the pullback bundle II}.

From now on, we fix the target semi-Euclidean metric to
be $\tilde{g}_0$ (defined as in $\S  \ref{subsec: semi-riem geometry prelims}$):
\begin{align}\label{g_0}
\gzero =\text{diag}(\underbrace{-1,\cdots,-1}_{\text{$\nu$ times}},
\underbrace{1,\cdots, 1}_{\text{$n-\nu$ times}};
\underbrace{-1,\cdots,-1}_{\text{$\tau$ times}}, \underbrace{1,\cdots, 1}_{\text{$k-\tau$ times}}),
\end{align}
and fix $\ind (M) = \nu$.
As before, we write the corresponding semi-Euclidean space as $\R^{n+k}_{\nu + \tau}$.

The main result of this section is Theorem \ref{theorem_main theorem, isometric immersions and GCR} below.
It gives an affirmative answer to the realization problem of semi-Riemannian manifolds
with lower regularity, provided that conditions (R1)--(R2) are satisfied
and that the manifold is simply-connected.

\begin{theorem}\label{theorem_main theorem, isometric immersions and GCR}
Consider  an $n$-dimensional simply-connected semi-Riemannian manifold $(M,g)$
with metric $g\in W^{1,p}_{\loc}(M; O(\nu, n-\nu))$ for $p>n$ and $\nu=\ind (M)\in\{0,1,\cdots,n\}$.
Suppose that $E$ is a bundle over $M$ with fiber $F=\R^k_\tau$,
bundle metric $g^E\in W^{1,p}_{\rm loc}(M; O(\tau, k-\tau))$,
and bundle connection $\na^E\in L^p_{\text{loc}}(M;T^*M\otimes {\rm End} E)$ compatible with $g^E$.
Let $\two\in L^p_{\loc}(M;{\rm Sym}^2T^*M \otimes E)$ be a symmetric two-tensor,
and let $S$ be defined by $g(S_\alpha X, Y)=g^E(\two(X,Y),\alpha)$ for
any $X,Y\in\Gamma(TM)$ and $\alpha\in \Gamma(E)$.
Moreover, assume that the GCR system on $E$ holds in the sense of distributions.
Then there exists a $W^{2,p}_{\loc}$ isometric immersion $f:(M,g)\emb (\widetilde{M}=\R^{n+k}_{\nu+\tau},\,\gzero)$
so that the normal bundle $TM^\perp:={f^*T\widetilde{M}}/{TM}$,
the second fundamental form, and the shape operator induced by $f$ are identified with $E$, $\two$, and $S$,
respectively, and $f$ is unique modulo the rigid motions in $(\widetilde{M}, \gzero)$.

In addition, if $g$, $\na^E$, $g^E$, $\two \in C^\infty$, then there exists a smooth isometric
immersion $f\in C^\infty(M;\widetilde{M})$.
\end{theorem}

\smallskip
\begin{remark} $\,$ Concerning the statement of
Theorem {\rm \ref{theorem_main theorem, isometric immersions and GCR}}, we have
\begin{enumerate}
\item[\rm (i)]
$\na^E$ is said to be compatible with $g^E$ if, for any $X\in \Gamma(TM)$ and $\alpha,\beta\in \Gamma(E)$,
\begin{equation}\label{revise, 3}
X g^E(\alpha,\beta) = g^E(\na^E_X\alpha,\beta) + g^E(\alpha, \na^E_X\beta).
\end{equation}
For example, the Levi--Civita connection on $M$ is compatible with $g$. As in Convention {\rm \ref{convention of metrics}},
we may express Eq. \eqref{revise, 3} as
$$
X \langle \alpha, \beta\rangle = \langle \na^E_X \alpha, \beta\rangle + \langle\alpha, \na^E_X\beta\rangle.
$$

\smallskip
\item[\rm (ii)]
For a bundle $E$ over $M$, $\text{Sym}^2 E^*$ denotes the space of symmetric $2$-tensors defined on $E$,
{\it i.e.}, each $M\in \Gamma(\text{Sym}^2 E^*)$ satisfies $M(\alpha,\beta)=M(\beta,\alpha)$ for any $\alpha,\beta\in\Gamma(E)$.
Note that, in general, a semi-Riemannian metric $g$ on $M$ does not lie in $\Gamma(\text{Sym}^2 T^*M)$.
Instead, $g\in \Gamma(O(\nu,n-\nu))$
as $g_{ij}\epsilon^j = g_{ji} \epsilon^i$ $(${\it cf}. $\S \ref{subsec: semi-riem geometry prelims}$ for the notations$)$.
\end{enumerate}
\end{remark}

\begin{remark}\label{rem: global topology issue}
$\,$ Theorem {\rm \ref{theorem_main theorem, isometric immersions and GCR}} has a global topological consequence
as follows{\rm :} If the GCR equations on the  abstract vector bundle $E$ are satisfied under the indicated regularity assumptions,
then the trivial rank-$(n+k)$ bundle $T\R^{n+k}_{\nu+\tau}$ has the following Whitney sum decomposition{\rm :}
\begin{equation*}
T\R^{n+k}_{\nu+\tau} = TM \oplus E.
\end{equation*}
\end{remark}

\begin{remark}\label{rem: equivalence of the three}
$\,$ Theorem {\rm \ref{theorem_main theorem, isometric immersions and GCR}}, together with
Proposition {\rm \ref{prop_second structural equation}}, yields the {\em equivalence} of the following statements,
provided that $(M,g)$ is simply-connected and $p>\dim\,M${\rm :}
\begin{enumerate}
\item[\rm (i)]
The existence of isometric immersions of semi-Riemannian manifolds{\rm ;}

\smallskip
\item[\rm (ii)]
The solvability of the GCR system in the sense of distributions{\rm ;}

\smallskip
\item[\rm (iii)]
The solvability of the Cartan structural system in the sense of distributions.
\end{enumerate}
\end{remark}

\subsection{\, Proof of the Realization Theorem, Theorem~\ref{theorem_main theorem, isometric immersions and GCR}}
If everything is $C^\infty$, then the Frobenius theorem on the equivalence of involutive and completely integrable distributions
can be directly applied, and  hence we may adapt the proof  by Tenenblat  \cite{Ten71} for the smooth Riemannian case.
In the case of lower regularity,
we only need to replace the Frobenius theorem with an analogous existence and regularity theorem
for certain first-order PDE systems with Sobolev coefficients.

\medskip
\noindent
{\bf Proof of Theorem {\rm \ref{theorem_main theorem, isometric immersions and GCR}}}.
Without loss of generality, we can first assume the result holds for the $C^\infty$ case.
As remarked above, to this end, we can adapt Tenenblat's arguments in \cite{Ten71},
taking into account various modifications required by non-trivial signatures in the semi-Riemannian setting.
See Appendix A.5 for the details of the proof.

Now we show for the lower regularity case: $g\in W^{1,p}_{\loc}({M, O(\nu, n-\nu)})$.
As in Appendix A.5, assume that the {\em Pfaff} and {\em Poincar\'{e}} systems
with
$$
g\in W^{1,p}_{\loc}({M, O(\nu, n-\nu)}),\quad \W\in L^p_{\loc}(U\subset M; \littleonk)
$$
are solved; that is,
there exist a bundle connection $A$ and an immersion $f$ in the following spaces:
\begin{equation*}
\begin{cases}
A\in W^{1,p}_{\loc}(U\subset M; \bigonk),\\[2mm]
f \in W^{2,p}_{\loc}(M;\widetilde{M}),
\end{cases}
\end{equation*}
such that $\text{ rank} (\dd f) = n$. Then $f$ is indeed an $W^{2,p}_{\loc}$ isometric immersion by construction.
The Pfaff and Poincar\'{e} systems are, respectively, as follows:
\begin{equation}\label{eqn pfaff system'}
\W = \dd A \cdot A^{-1},\qquad A(0)=A(x_0),
\end{equation}
and
\begin{equation}\label{eqn_poincare'}
\dd f = {\underaccent{\widetilde}{\Theta}}
\cdot A,\qquad f(0)=f(x_0),
\end{equation}
where $x_0$ is a given point in a local chart $U \subset M$.

The solvability of the Poincar\'{e} system \eqref{eqn_poincare'} with Sobolev coefficients is easy to be established.
For any given
$$
A\in W^{1,p}_{\loc}(U\subset M; \bigonk),
$$
we want to solve for $f$ in $W^{2,p}_{\loc}(M;\widetilde{M})$.
Since all the results are stated in local Sobolev spaces,
it suffices to assume that $U$ is a smooth bounded open subset of $\R^n$.
In this setting, choose $J_\varepsilon\in C^\infty(\R^n)$ to be the standard mollifier
and set $\Theta_\varepsilon:=J_\varepsilon \ast (\tth \cdot A)$. It follows that
\begin{equation*}
\Theta_\varepsilon \longrightarrow \tth\cdot A \qquad \text{ in } W^{1,p}(U;\widetilde{M})
\,\,\, \text{ as $\varepsilon \rightarrow 0^{+}$}.
\end{equation*}
In particular, $\{\Theta_\varepsilon\}$ is uniformly bounded in $W^{1,p}$.

Now, $\Theta_\varepsilon$ is a smooth closed $1$-form ({\it cf.} Appendix A.5) for each $\varepsilon>0$,
so we can invoke the solvability of the Poincar\'{e} system in the $C^\infty$ case to find
some $f_\varepsilon \in C^\infty(U;\widetilde{M})$ with $\dd f_\varepsilon=\Theta_\varepsilon$.
By adding a constant, we may assume that $\int_U f_\varepsilon \,\dd x = 0$.
Then the  Poincar\'{e} inequality gives us
\begin{align*}
\|f_\varepsilon\|_{W^{2,p}(U;\widetilde{M})} \leq C\big(\|f_\varepsilon\|_{W^{1,p}(U;\widetilde{M})}+\|\Theta_\varepsilon\|_{W^{1,p}(U;\widetilde{M})}\big).
\end{align*}
Hence, thanks to the Rellich lemma and the uniform boundedness
of $\{\Theta_\varepsilon\}\subset W^{1,p}(U;\widetilde{M})$,
we obtain that
$\|f_\varepsilon\|_{W^{2,p}(U;\widetilde{M})}\leq C_0<\infty$.
Therefore, there exists a limiting function $\tilde{f}$ so that $f_\varepsilon \rightarrow \tilde{f}$
in $W^{2,p}(U;\widetilde{M})$ (modulo subsequences) with $\dd \tilde{f} = \tth \cdot A$.

The Pfaff system \eqref{eqn pfaff system'} with Sobolev coefficients is more difficult to tackle:
The Frobenius theorem cannot be directly applied, since we need at least $C^1$--regularity;
in addition, we cannot apply a simple mollification argument,
since the compatibility condition ({\it i.e.}, the second structural system $\dd \W=\W\wedge \W$)
contains quadratic nonlinear terms.

However, the following result
serves for our purpose:

\begin{lemma}[Mardare \cite{Mar07}]\label{lemma mardare}
Let $\Omega\subset \R^n$ be a simply-connected open set,
$x_0 \in \Omega$, and $M_0\in \mathfrak{gl}(l; \R)$. Then the following system{\rm :}
\begin{equation*}
\frac{\partial M}{\partial x^i}={\mathfrak{S}}_i\cdot M, \,\,\, i=1,2,\ldots,n, \,\,\qquad M(x_0)=M_0,
\end{equation*}
with the matrix fields $\mathfrak{S}_i\in L^p_{\loc}(\Omega; \mathfrak{gl}(l;\R))$ for $i=1,2,\ldots, n$, and $p>n$,
has a unique solution $M\in W^{1,p}_\loc (\Omega; \mathfrak{gl}(l; \R))$ if and only if the following compatibility
condition holds{\rm :}
\begin{equation}\label{compatibility pfaff}
\frac{\partial \mathfrak{S}_i}{\partial x^j} - \frac{\partial \mathfrak{S}_j}{\partial x^i}
=[\mathfrak{S}_i, \mathfrak{S}_j] \qquad\text{ for each } i,j=1,2, \dots, n,
\end{equation}
in the sense of distributions.
\end{lemma}

As Lemma \ref{lemma mardare} is formulated for $\Omega \subset \R^n$,
we correspondingly take $U\subset M$ as a trivialized local chart
so that bundle $E$ can be regarded as $U \times \R^k$ over $U$.
Hence, on $U$, without loss of generality, we may assume that $[\p_i,\p_j]=0$.
We take
\begin{equation*}
\mathfrak{S}=\W \in L^p_\loc (U; T^*M\otimes\littleonk), \qquad \mathfrak{S}_i = \W(\p_i).
\end{equation*}
Then
\begin{align*}
\p_i \mathfrak{S}_j - \p_j \mathfrak{S}_i
= \p_i(\W(\p_j)) - \p_j(\W(\p_i)) + \W([\p_i,\p_j]) =  \dd\W(\p_i,\p_j).
\end{align*}
On the other hand, we have
\begin{align*}
[\mathfrak{S}_i, \mathfrak{S}_j] = \W(\p_i)\cdot \W(\p_j) - \W(\p_j)\cdot \W(\p_i) = (\W\wedge \W) (\p_i,\p_j).
\end{align*}
Thus, the compatibility condition in Lemma \ref{lemma mardare} is verified by the second structural
system \eqref{eqn_second structural eqn}.
The Pfaff system \eqref{eqn pfaff system'} with Sobolev coefficients is hence uniquely solvable on local charts.

Therefore, we now arrive at the existence of a local isometric immersion in the lower regularity case,
provided that the second structural system (or equivalently, the GCR system)
holds in the sense of distributions.

Finally, we deduce the global existence of an isometric immersion, which follows from a standard monodromy argument.
Given any two points $x,y\in M$ with $x \neq y$, we connect them by a continuous
curve (again since $W^{1,p}_\loc \hookrightarrow C^0_{\rm loc}$ for $p>n$),
denoted by $\gamma: [0,1]\rightarrow M$ with $\gamma(0)=x$ and $\gamma(1)=y$.
More precisely, $\gamma$ is chosen as a continuous representative in the Sobolev space.
Let $f$ be the $W^{2,p}_\loc$ isometric immersion in a neighborhood of $x$,
whose existence is guaranteed by the earlier steps.
We cover $\gamma([0,1])$ by finitely many charts $\{V^1,\ldots,V^N\}$.
By the uniqueness statement in Lemma \ref{lemma mardare},
we can extend the isometric immersion $f$ to $\bigcup_{i=1}^N V^i$,
especially including a neighborhood of $y$.

Thus, it suffices to show that the extension of $f$ is independent of the choice of $\gamma$.
Indeed, if $\eta:[0,1]\rightarrow M$ is another continuous curve connecting $x$ and $y$,
by concatenating $\gamma$ with $\eta$, we form a loop $L\subset M$.
As $M$ is simply-connected, the restriction $f|_L$ is homotopic to a constant map
so that $(f\circ\gamma)(1)=(f\circ\eta)(1)$.
In this way, we have verified that $f$ can be extended to a global isometric
immersion of $M$ into $\widetilde{M}$, provided that $M$ is simply-connected.
This completes the proof.

\medskip
As a remark, in the realization theorem, Theorem \ref{theorem_main theorem, isometric immersions and GCR},
it requires that  $g \in W^{1,p}_\loc$ with $p>n=\dim\,M$.
This is because of both the regularity assumptions in Lemma \ref{lemma mardare}
and the continuity requirements for the topological arguments.
All the other results in this paper hold for $p>2$, regardless of the dimension of $M$.
Also note that $(M,g)$ is assumed to be simply-connected in Theorem~\ref{theorem_main theorem, isometric immersions and GCR},
which prevents the occurrence of branched immersions.

\subsection{\, Weak Rigidity of Isometric Immersions of Semi-Riemannian Manifolds}

Recall that, in Theorem~\ref{thm: weak continuity of Cartan's structural eq},
we have established the weak continuity of the Cartan structural system on a semi-Riemannian $(M,g)$
and, in Proposition \ref{prop_second structural equation},
we have shown the equivalence of the structural system with the GCR system,
both for $p>2$ {\rm regardless of $\dim M$}.
If we translate this PDE-theoretic weak continuity theorem into geometric settings,
then it is unsurprising that the $W^{2,p}_\loc$ isometric immersions of $M$ are weakly rigid.
More precisely, we have

\begin{theorem}\label{theorem_weak rigidity}
Let $(M,g)$ be a semi-Riemannian manifold of dimension $n$
with  $\ind(M)=\nu$, $g\in L^\infty_\loc(M; O(\nu,n-\nu))$,
and the Levi--Civita connection $\na$ of $g$ in $L^p_\loc$ for $p>2$.
Let $\{f_\varepsilon\} \subset W^{2,p}_\loc(M; \R^{n+k})$ be
a family of isometric immersions
of semi-Riemannian submanifolds, with the second fundamental forms $\{\two_\varepsilon\}$
and normal connections $\{\na^{\perp}_{\varepsilon}\}$
satisfying GCR system \eqref{gauss}--\eqref{ricci}.
Assume that $\{f_\varepsilon\}$ is uniformly bounded in $W^{2,p}_\loc$
and  $\R^{n+k}$ is endowed with the semi-Euclidean metric $\gzero$ as in Eq. \eqref{g_0}.
Then, after passing to a subsequence if necessary, $\{f_\varepsilon\}$ weakly converges in $W^{2,p}_\loc$
to an isometric immersion $f \in W^{2,p}_\loc(M;\R^{n+k})${\rm ;}
in addition, the second fundamental form and the normal connection of $f$ are the weak $L^p_\loc$ limits
of $\{\two_\varepsilon\}$ and $\{\na^{\perp}_{\varepsilon}\}$, respectively,
and still satisfy the GCR system.

The same result holds if $\{(\two_\varepsilon, \na^{\perp}_{\varepsilon})\}$ are replaced by
the connection $1$-forms $\{\W_\varepsilon\}$, and the GCR system is replaced by
the Cartan structural system \eqref{eqn_second structural eqn}.
\end{theorem}

\begin{proof}
$\,$ Let $\{f_\varepsilon\}$ be a bounded family in $W^{2,p}_\loc$ where $p>2$.
Then, modulo subsequences, $\{\dd f_\varepsilon\}$ is weakly convergent in $W^{1,p}_\loc$, hence strongly convergent in $L^p_\loc$ due to the Rellich lemma.
Thus, after passing to a subsequence and thanks to the H\"{o}lder inequality,
$\tilde{g}_0 (\dd f_\varepsilon, \dd f_\varepsilon)$ converges strongly in $L^{\frac{p}{2}}_\loc$ to $\tilde{g}_0 (\dd \tilde{f},\dd \tilde{f})$,
which equals to metric $g$ by assumption,
where $\tilde{f}$ is a weak $W^{2,p}_\loc \cap W^{1,\infty}_\loc$ limit of $\{f_\varepsilon\}$.
In addition, by passing to a further subsequence, we may deduce that $\dd f_\varepsilon \rightarrow \dd \tilde{f}\,\,$ {\it a.e.}
from the strong $L^p_\loc$ convergence
and that $\tilde{g}_0 (\dd \tilde{f},\dd \tilde{f})=g$ {\it a.e.} from the strong $L^{\frac{p}{2}}_\loc$ convergence,
by virtue of $p>2$. This shows that $\tilde{f}$ is an isometric immersion, again in the {\it a.e.} sense.

On the other hand, by the weak continuity of the GCR system in Theorem~\ref{thm: weak continuity of Cartan's structural eq},
we find that the second fundamental form and the normal connection of the limiting isometric immersion $\tilde{f}$ --- which are
weak $L^p_\loc$ limits of the related quantities for $f_\varepsilon$ (possibly modulo a further subsequence) --- satisfy the GCR equations
in the sense of distributions.
This observation together with Proposition~\ref{prop_second structural equation} completes the proof.
\end{proof}

In the case that $p>n$, the above result follows directly from the realization
theorem (Theorem \ref{theorem_main theorem, isometric immersions and GCR}),
together with Theorem~\ref{thm: weak continuity of Cartan's structural eq}
and Proposition~\ref{prop_second structural equation}.
In fact,
it can be proved easily for $p>n$ without applying any of the machineries above,
but just using the Sobolev-Morrey embedding $W^{2,p}_\loc \emb W^{1,\infty}_\loc$
and the identity $\two_{jk} = \p_j\p_k f -\G^i_{jk} \p_i f$ (see, {\it e.g.},
Bryant--Griffith--Yang \cite[page 959]{BGY}
for the Riemannian case).
The main point of our arguments here is to extend to the case $p>2$, irrespective of $\dim\,M$.

In particular, we comment that, {\em under the stronger hypotheses} that both $M$ is simply-connected and $p>n =\dim \,M$,
Theorems \ref{thm: weak continuity of Cartan's structural eq}--\ref{cor: weak continuity of GCR} can be deduced
easily from the realization theorem (Theorem~\ref{theorem_main theorem, isometric immersions and GCR}), in view of Remark \ref{rem: equivalence of the three}.

\bigskip
\noindent
{\bf Alternative Proof for
Theorem {\rm \ref{thm: weak continuity of Cartan's structural eq}}--{\rm \ref{cor: weak continuity of GCR}} with $\pi_1(M)=\{0\}$ and $p>n$}.

Without loss of generality, we may assume that $M$ is compact and that $f_\varepsilon$ converges weakly in $W^{2,p}$
to a map $f:M \map\,\R^{n+k}_{\nu+\tau}$.
Since the embedding $W^{1,p} \emb C^0$ is now compact for $p>n$, by choosing continuous representatives in suitable
Sobolev classes, $g_\varepsilon:=f_\varepsilon^\ast \gzero$ converges uniformly to $g:=f^\ast \gzero\in W^{1,p}$.

Note that $f_\varepsilon: (M,g_\varepsilon) \emb (\R^{n+k}_{\nu+\tau}, \gzero)$ and $f:(M,g) \emb (\R^{n+k}_{\nu+\tau}, \gzero)$
are isometric immersions by construction. By the realization theorem, Theorem \ref{theorem_main theorem, isometric immersions and GCR},
the connection $1$-forms $\W_\varepsilon$ and $\W$ (corresponding to $f_\varepsilon$ and $f$, respectively)
satisfy the Cartan structural systems:
$$
\dd\W_\varepsilon=\W_\varepsilon\wedge\W_\varepsilon, \qquad \dd\W = \W \wedge \W.
$$
These two systems are well-defined, with the left-hand sides in $W^{1,p}$ and the right-hand sides in $L^{\frac{p}{2}}$ for $p>n\geq 2$.
Also, Definition \ref{def of W} for the connection $1$-forms implies that $\W_\varepsilon \weak \W$ in $L^p$.
Then Theorem \ref{thm: weak continuity of Cartan's structural eq} follows when $\pi_1(M)=\{0\}$ and $p>n$.
We can conclude Corollary \ref{cor: weak continuity of GCR} from Proposition \ref{prop_second structural equation}.

\medskip
Nonetheless, we emphasize once more that the above short proof is available {\it only for} $p>n$;
the argument does not extend to  the less stringent case $p>2$, even with Theorem~\ref{theorem_weak rigidity} at hand.
This is because the current proof of the realization
theorem (Theorem~\ref{theorem_main theorem, isometric immersions and GCR}; {\it cf.} also Szopos  \cite{szopos}) essentially
needs $p>n$, as it is crucial for Lemma~\ref{lemma mardare}.

\medskip
\section{\, Further Applications}\label{sec: further app}

In this final section, we present some further applications of the results and techniques
developed in \S 2--\S 5 above.
\begin{enumerate}
\item[\rm (i)]
Using the weak continuity of isometric immersions (Theorem \ref{theorem_weak rigidity}),
we show the weak continuity of the constraint equations in general relativity;
\item[\rm (ii)]
Directly utilizing the geometric compensated compactness theorem,  Theorem \ref{thm: generalized quadratic theorem on manifolds-b},
we establish the weak continuity of  quasilinear wave equations satisfying the {\em null condition}
(introduced first by Klainerman \cite{klainerman-1}; see also \cite{christodoulou,klainerman}).
\item[\rm (iii)]
Employing a generalized version of the GCR system, we prove the weak continuity of general immersed hypersurfaces,
{\it i.e.}, the $1$-co-dimensional submanifolds with possibly degenerate induced metrics.
\end{enumerate}

\subsection{\, Weak Rigidity of  Einstein's Constraint Equations}

Let $(V,g)$ be a Lorentzian manifold of dimension $N+1$.
The vacuum Einstein field equation is
$$
{\rm Ric}_g=0,
$$
that is, the Ricci curvature of $g$ vanishes.
This system consists of $\frac{(N+1)(N+2)}{2}$ scalar equations,
in which $N+1$ equations are determined by the initial data on some space-like hypersurface
via the Gauss--Codazzi equations.
These $N+1$ equations are known as Einstein's constraint equations;
see Bartnik--Isenberg \cite{bartnik-isenberg}, Choquet--Bruhat \cite{choquet-bruhat},
Corvino--Schoen \cite{corvino-schoen}, and the references cited therein.

In the Minkowski case $(\R^{N+1}, \mathfrak{m})$, we can show the following theorem:

\begin{theorem}
Let $M$ be a space-like hypersurface of the Minkowski space-time $(\R^{N+1}, \mathfrak{m})$
with a family of immersions $\{f_\varepsilon\}$.
Denote by $\gamma_\varepsilon := f_\varepsilon^*\mathfrak{m}$ the pull-back metrics on $M$.
Suppose that, for each fixed $\varepsilon>0$, $(M, \gamma_\varepsilon)$ satisfies the Einstein constraint equations in the vacuum{\rm :}
\begin{equation}\label{constraint equations}
\begin{cases}
{\rm scal}_\varepsilon + ({\rm tr}_{\gamma_\varepsilon} h_\varepsilon)^2 - |h_{\varepsilon}|^2 = 0,\\
\sum_{j=1}^N \widetilde{\na}^j \big((h_\varepsilon)_{ij} - {\rm tr}_{\gamma_\varepsilon}(h_\varepsilon) (\gamma_\varepsilon)_{ij}\big)=0
\qquad \text{ for } i=1,2,\ldots,n.
\end{cases}
\end{equation}
In the above, $\widetilde{\na}$ is the Levi--Civita connection on $(\R^{N+1}, \mathfrak{m})$,
${\rm scal}_\varepsilon$ is the scalar curvature of $(M, \gamma_\varepsilon)$,
and $h_\varepsilon$ is the second fundamental form{\rm :}
\begin{equation*}
\widetilde{\na}_XY = (\na_\varepsilon)_XY + h_\varepsilon(X,Y)\mathbf{n}_\varepsilon\qquad \text{ for all } X,Y\in\G(TM),
\end{equation*}
where $\na_\varepsilon$ is the Levi--Civita connection on $(M,\gamma_\varepsilon)$ and $\mathbf{n}_\varepsilon$ is the time-like unit normal.
If $\{f_\varepsilon\}$ is uniformly bounded in $W^{2,p}_{\rm loc}(M, \R^{N+1})$ for $p>2$,
then it converges weakly in $W^{2,p}_{\rm loc}$  to an immersion $\tilde{f}: M \map (\R^{N+1}, \mathfrak{m})$
such that $(M, \tilde{f}^*\mathfrak{m})$ satisfies Einstein's constraint equations in the sense of distributions.
\end{theorem}

\begin{proof}
By construction, $f_\varepsilon: (M, \gamma_\varepsilon) \map (\R^{N+1}, \mathfrak{m})$ is an isometric immersion for each $\varepsilon>0$.
Then $f_\varepsilon \weak \tilde{f}$ in $W^{2,p}_{\rm loc}$, where $\tilde{f}$ is an isometric immersion whose
second fundamental form satisfies the Gauss--Codazzi equations in the sense of distributions, by Theorem \ref{theorem_weak rigidity}.
However, the constraint equations  \eqref{constraint equations} are implied
by the Gauss--Codazzi equations (see Bartnik--Isenberg \cite{bartnik-isenberg}).
In view of Remark \ref{rem: equivalence of the three}, the assertion now follows.
\end{proof}

\subsection{\, Weak Continuity of Quasilinear Wave Equations}
Now we give an application of our quadratic theorem of compensated compactness, {\it i.e.},
Theorem \ref{thm: generalized quadratic theorem on manifolds-b},
to the weak continuity of a special class of nonlinear wave equations:
\begin{equation}\label{wave eq}
\Box_\mm \phi^I = F^I(\phi, \p \phi) \qquad \text{ for all } I \in \{1,2,\ldots, N\}.
\end{equation}
This system is posed on $(\R^{3+1}, \mm)$, where $\mm={\rm diag}(-1, 1, 1, 1)$ is the Minkowski metric,
and $F=\{F^I\}_{1\leq I \leq N}$ is the source function.
We are concerned with $\phi=\{\phi^I\}_{1\leq I \leq N}: \R^{3+1} \rightarrow \R^N$.
The source function $F$ consists of quadratic terms with respect to $(\phi, \p\phi)$,
where $\p$ denotes the total space-time derivative.

A classical result due to Christodoulou \cite{christodoulou} and Klainerman \cite{klainerman}
is the following: When the smooth initial data is sufficiently small, the Cauchy problem for Eq. \eqref{wave eq} has a unique
solution $\phi \in C^\infty_c([0,\infty) \times \R^3; \R^N)$,
provided that $F$ satisfies the {\em null condition}{\rm :}
\begin{enumerate}
\item[(i)] $F^I(0)=0$ and $\p F^I(0)=0$,

\smallskip
\item[(ii)] $Q_{F^I}(\p \phi) = \sum_{J,K=1}^N \sum_{\mu, \nu =0}^3 A^{\mu \nu}_{IJK} (\p_\mu \phi^J) (\p_\nu \phi^K)$
  for each $I\in\{1,2,\ldots,N\}$ with
\begin{align*}
\sum_{\mu, \nu =0}^3  A^{\mu \nu}_{IJK} \xi_\mu \xi_\nu =0
\end{align*}
for any null co-vector $\xi \in T^*\R^{3+1}$ and $I,J,K\in\{1,2,\ldots,N\}$,
 where $Q_{F^I}$ denotes the quadratic part in $\p\phi$ in the Taylor expansion of $F^I$ at $(\phi, \p\phi)=0$:
\begin{equation*}
	Q_{F^I} (z):= \sum_{|\alpha|=2} \frac{\p_\alpha F^I(0)}{\alpha!} z^\alpha
   \qquad\, \text{ for all } z \in \R^N
\end{equation*}
in the multi-index notations,
and  $\xi \in T^*\R^{3+1}$ is a null co-vector if and only if $\mm^{\mu\nu}\xi_\mu\xi_\nu=0$.
\end{enumerate}

For our purpose, we take the following bundle of type--$(1,1)$ tensors:
\begin{equation*}
E=T^*\R^{3+1}\otimes T\R^N.
\end{equation*}
Then, for each $I\in\{1,2,\ldots,N\}$, define the bundle operator $\T_{I} \in {\rm Hom} (E;\R)$:
\begin{equation}\label{operator T_I}
\T_{I} \, s:= \sum_{J,K=1}^N \sum_{\mu,\nu=0}^3 A^{\mu \nu}_{IJK} (\p_\nu s^J_\mu) \theta^K,
\end{equation}
where $\{\theta^K\}\subset T^*\R^N$ is the co-vector basis dual to $\{\p_K\}$.
The associated operator cone is
\begin{align}\label{operator cone for T_I}
\Lambda_{\T_{I}}:=\left\{\lambda\in T^*\R^{3+1} \otimes T\R^N\,:
\begin{array}{ll}
 \sum_{\mu,\nu=0}^3  A^{\mu\nu}_{IJK} \lambda^J_\mu s^K_\nu = 0 \text{ for some non-zero}\\[1mm]
\text{ section } s \in \Gamma( T^*\R^{3+1}\oplus T\R^N)\setminus\{0\}
\end{array}
\right\}.
\end{align}

The following observation is crucial: For each null co-vector $\xi \in T^*\R^{3+1}$,
if it is identified with $\xi \otimes {\rm id} \in T^*\R^{3+1} \otimes T\R^N$ (where ${\rm id}$ is the tautological tensor on $T\R^N$),
then it lies in $\Lambda_{\T_{I}}$.
In other words, the {\em null cone} of the space-time $(\R^{3+1}, \mm)$ can be viewed as a subset of the operator
cone $\Lambda_{\T_{I}}$ for every $I\in \{1,2,\ldots,N\}$.

Also, for each $I \in \{1,2,\ldots,N\}$, consider the quadratic form:
\begin{align}\label{quadratic form QFI}
Q_{F^I} (s):= \sum_{J,K=1}^N \sum_{\mu,\nu=0}^3 A^{\mu\nu}_{IJK} s_\nu^J s_\mu^K
\,\,\,\,\text{ for $s=\{s_\mu^J\}_{1 \leq J \leq N, 0 \leq \mu \leq 3} \in \Gamma (E)$}.
\end{align}
It can be defined intrinsically on $\Gamma(E)$.
It is easy to check
that $Q_{F^I}$ agrees with the quadratic terms in $\p\phi$ of the source term $F^I$.

Now, applying Theorem \ref{thm: generalized quadratic theorem on manifolds-b} to the sequence
of sections
$$
\{\p \phi_\varepsilon\}\subset L^2_{\loc}(\R^{3+1}; T^*\R^{3+1}\otimes T\R^N),
$$
we obtain the following compensated compactness framework, which enables us
to verify the $H^1_{\rm loc}$ weak continuity of Eq. \eqref{wave eq}.
Indeed, it requires to pass the limits in the source term $F^I(\phi_\varepsilon, \p \phi_\varepsilon)$,
as the left-hand side of the equation is linear in $\phi_\varepsilon$.

\begin{proposition}\label{propn: wave equation}
Let the source term $F^I(\phi, \p\phi)$
satisfy the null condition so that
\begin{equation}\label{N-(iii)}
Q_{F^{I}}(s)=0 \qquad \mbox{for any $s \in \Lambda_{\T^{I}}$},
\end{equation}
where the operator cone $\Lambda_{\T^{I}}$ is defined according
to Eqs. \eqref{operator cone for T_I}--\eqref{quadratic form QFI}.
Assume that $\{\phi_\varepsilon\}$ is a family of functions in $H^1_{\loc}(\R^{3+1}, \R^N)$ such that
\begin{enumerate}
\item[\rm (i)]
$\phi_\varepsilon \weak \phi$ weakly in ${H}^1_\loc${\rm ;}
\smallskip
\item[\rm (ii)]
$\big\{\sum_{J=1}^N\sum_{\mu,\nu=0}^3 A^{\mu\nu}_{IJK} \p_\mu \p_\nu \phi_\varepsilon^J\big\}$
is pre-compact in $H^{-1}_{\loc}(\R^{3+1})$ for all $I,K\in\{1,2,\ldots,N\}$.
\end{enumerate}
Then
$$
Q_{F^{I}}(\p\phi_\varepsilon)  \weak Q_{F^{I}}(\p\phi) \qquad
\mbox{in the sense of distributions}.
$$
As a consequence, if Eq. \eqref{wave eq} admits
a family of weak solutions $\{\phi_\varepsilon\} \subset H^1_{\loc}(\R^{3+1}, \R^N)$
satisfying {\rm (i)}--{\rm (ii)}, then
the weak limit $\phi$ in $H^1$ is also a weak solution of \eqref{wave eq}.
\end{proposition}

In particular, a necessary condition for \eqref{N-(iii)}
above is that
$Q_{F^{I}} (\xi \otimes {\rm id}) = 0$ for any null co-vector $\xi \in T^*\R^{3+1}$.

The above proposition shows that the quasilinear wave equation with null condition in $3+1$ dimensions
is weakly continuous in $H^1_{\rm loc}$.
However, it is well-known ({\it cf.} Rodnianski \cite{rodnianski}) that the Einstein equations fail to satisfy the null conditions,
even in the vacuum or scalar field cases.
It would be interesting to analyze further the weak continuity of the Einstein equations and other physical/geometric PDEs.

\subsection{\, Weak Rigidity of General Immersed Hypersurfaces}

We now discuss the weak rigidity of immersed hypersurfaces that are {\em not} semi-Riemannian submanifolds of the ambient spaces.
It is remarked in $\S \ref{subsec: statement of realisation}$ ({\it cf.} Condition (R1)) that,
if metric $\gzero$ is degenerate on a hypersurface $\Sigma$,
then $\Sigma$ cannot be obtained via an isometric immersion of any semi-Riemannian manifold.
Nevertheless, such degenerate scenarios occur naturally in physics.

One primary example is the {\em lightcone}:
$$
\Lambda=\{(t,x_1,x_2,x_3) \in \R^4\,:\, t^2 = x_1^2+x_2^2+x_3^2\}
$$
of the Minkowski space-time $(\R^{3+1}, \mathfrak{m})$ with $\mathfrak{m}= \text{diag}(-1,1,1,1)$.
Although, for any $x,v\in \Lambda$, $g_x(v,w)\neq 0$ for all time-like vectors $w$ in the tangent space at $x$,
we see that $g_x(v,\cdot) \equiv 0$ on $T_x\Lambda$, where $\Lambda$ is known as a {\em null hypersurface}.
In addition, the stationary limit surface of Kerr's vacuum solution is everywhere time-like, except at the points on the axis
where it is null and tangent to the horizon ({\it cf}. \cite{mars}).
A more recent example in \cite{mars2} is the gluing of two Anti-de-Sitter (AdS) $5$-dimensional space-times
with different cosmological constants along a general hypersurface $\Sigma = \Sigma^E \sqcup \Sigma^{\rm null} \sqcup \Sigma^L$,
where $\Sigma^{\rm null}$ is $3$-dimensional, such that the restriction of the metric is time-like on $\Sigma^L$,
space-like on $\Sigma^E$, and null on $S$.
If the coordinate system is suitably chosen, $\Sigma^{\rm null}$ may lie in the hypersurface of form $\{t=t_0\}$.
This example gives a possible model for the transition between two distinct AdS universes across
{\it brane} $\Sigma$, whence $\Sigma^{\rm null}$ models the big-bang singularity.

Motivated by the physical applications above,
a treatment for the realization problem and the weak rigidity of general hypersurfaces is desired.
However, the constructions in $\S \ref{subsec: GCR}$, especially the derivation of the GCR system or the Cartan structural system,
fail in this case --- the orthogonal decomposition of tangent spaces
as in Eq. \eqref{eqn decomposition of the pullback bundle} is no longer valid.

To overcome this difficulty, we employ the construction of {\em rigging vector fields}; {\it cf.} \cite{lefloch1,lefloch2,mars,schouten}.
The idea is as follows: Consider the hypersurface via the local embedding $\iota: \Sigma \emb (\widetilde{M}, \tilde{g})$.
If $\iota^*\tilde{g}$ is null, we find a non-vanishing vector field $\ell \in \Gamma (\iota^* T\widetilde{M})$ along $\Sigma$ so that
\begin{equation}\label{the orthogonal splitting for rigging}
T_{\iota(x)} \widetilde{M} \cong T_x\Sigma \oplus  {\rm span }\{\ell_x\}.
\end{equation}
Thus, we can derive the Gauss--Codazzi equations (for hypersurfaces, the Ricci equation is always trivial) from the orthogonal
decomposition \eqref{the orthogonal splitting for rigging}.
However, technicalities are unavoidable because the rigging field $\ell$ {\em never} coincides with the normal vector field,
whenever $\Sigma$ is null --- this leads to three Codazzi equations instead of one.

From now on, $\alpha$ always denotes a co-vector field, {\it i.e.}, an element of $\Gamma(T^*\Sigma)$.
This is in agreement with \cite{mars,schouten}. The first main result in this subsection is

\begin{theorem}\label{theorem for hypersurface}
Let $\iota: \Sigma \emb (\R^{n+1}, \gzero)$ be a $W^{2,p}_{\loc}$ immersion of a simply-connected general hypersurface for $p>n$,
for which the pullback tensor $\iota^*\gzero$ is allowed to degenerate on $\Sigma$.
Let the normal $1$-form of $\Sigma$ to be  $\nn \in \Gamma(\iota^*T^*\R^{n+1})$.
Moreover, assume that $\ell \in \Gamma(T\R^{n+1})$ is a rigging vector field, {i.e.}, $\nn(\ell) =1$ everywhere on $\Sigma$.
Take $\{e^i\}_{i=1}^n\subset \Gamma(T\Sigma)$ as an orthonormal frame on $\Sigma$,
and $\{\theta^i\}_{i=1}^n\subset \Gamma(T^*\Sigma)$ as its co-frame.
Furthermore, define the tensor fields $K \in W^{1,p}_\loc (\Sigma; \wedge^2 T^*\Sigma)$
and $\Psi \in W^{1,p}_\loc (\Sigma; T\Sigma \otimes T^*\Sigma) = W^{1,p}_\loc (\Sigma; {\rm End }\, T\Sigma)$ by
\begin{equation*}
K:=\widetilde{\na}\nn,\qquad \Psi := \widetilde{\na}\ell,
\end{equation*}
that is,
\begin{align*}
K(X,Y) = \widetilde{\na}\nn (X,Y),\qquad \Psi(\alpha, X):= \alpha(\widetilde{\na}_X \ell)
\end{align*}
for each $X,Y \in \Gamma(T\Sigma)$ and $\alpha \in \Gamma(T^*\Sigma)$.
Define also $\psi \in W^{1,p}_\loc (\Sigma; T^*\Sigma)$ by
\begin{equation*}
\psi (X):= \nn(\widetilde{\na}_X \ell).
\end{equation*}
Then the following equations hold in the sense of distributions{\rm :}
\begin{eqnarray}
&&\alpha(R(X,Y)Z) - K(Y,Z)\Psi(\alpha,X) + K(X,Z)\Psi(\alpha, Y)=0,\label{hypersurface gauss}\\[2mm]
&&K(X, {\na}_YZ) - K(Y, {\na}_XZ) + K([X,Y],Z) + XK(Y,Z)- YK(X,Z)  \nonumber\\
&&\qquad - K(Y,Z)\psi(X) + K(X,Z)\psi(Y)=0,\label{hypersurface codazzi 1}\\[2mm]
&&X\Psi(\alpha, Y) - Y\Psi(\alpha, X) + \Psi(\alpha, [X,Y]) + \psi(Y)\Psi(\alpha, X)- \psi(X)\Psi(\alpha,Y) \nonumber\\
&&\qquad+\sum_{i=1}^n \big\{\Psi(\theta^i, Y)\alpha({\na}_X e^i)
  - \Psi(\theta^i, X)\alpha({\na}_Ye^i)\big\}=0,\label{hypersurface codazzi 2}\\
&&X\psi(Y) - Y\psi(X) + \psi([X,Y])\nonumber\\
 &&\qquad + \sum_{i=1}^n \big\{K(e^i, Y)\Psi(\theta^i, X)- K(e^i, X)\Psi(\theta^i, Y)\big\}=0 \label{hypersurface codazzi 3}
\end{eqnarray}
for $X,Y,Z \in \Gamma(T\Sigma)$ and $\alpha\in \Gamma(T^*\Sigma)$ such that $\alpha(l)=0$,
and $R$ is the Riemann curvature of $\Sigma$.

Conversely, if Eqs. \eqref{hypersurface gauss}--\eqref{hypersurface codazzi 3}
hold in the sense of distributions
for $K \in W^{1,p}_\loc(\Sigma; \wedge^2 T^*\Sigma)$, $\Psi \in W^{1,p}_\loc (\Sigma; {\rm End }\, T\Sigma)$,
and $\psi \in W^{1,p}_{loc}(\Sigma; T^*\Sigma)$, then there exist an immersion $\iota \in W^{2,p}_{\loc}(\Sigma; \R^{n+1})$
and a rigging vector field $\ell \in \Gamma(T\R^{n+1})$
such that $K=\widetilde{\na}\nn$, $\Psi = \widetilde{\na}\ell$, and $\psi(X)=\nn(\widetilde{\na}_X \ell)$.
\end{theorem}

Eq.  \eqref{hypersurface gauss} and Eqs. \eqref{hypersurface codazzi 1}--\eqref{hypersurface codazzi 3}
are known as the {\em Gauss} equation and the {\em Codazzi} equations
of the general hypersurface $\Sigma$, respectively.
As in the physics literature ({\it cf.} \cite{clarke3,mars,schouten}),
the geometric quantities $\{K, \Psi, \psi\}$ are interpreted as
the {\em intrinsic, extrinsic, and normal second fundamental forms}
of $\Sigma$, respectively.
If metric $\gzero$ is Lorentzian with signature $\{-1,+1,\ldots, +1\}$,
the rigging field $\ell$ can be chosen as time-like, whose trajectory thus corresponds to the worldline of an observer.
On the other hand, if $\gzero|_{\Sigma}$ is non-degenerate, then $\ell$ can be chosen as the unit normal vector field,
and Eqs. \eqref{hypersurface gauss}--\eqref{hypersurface codazzi 3} reduce to the usual Gauss-Codazzi equations in $\S \ref{subsec: GCR}$.

\bigskip
\noindent
{\bf Proof of Theorem {\rm \ref{theorem for hypersurface}}}.
The proof consists of three steps. We emphasize the difference between the case of general hypersurfaces
and the case of semi-Riemannian submanifolds (Theorem \ref{theorem_main theorem, isometric immersions and GCR}),
while the parallel arguments are only briefly sketched.

\medskip
\noindent
{\bf 1.} We first deduce the Gauss--Codazzi equations \eqref{hypersurface gauss}--\eqref{hypersurface codazzi 3}
from the immersion $\iota$.
As in \S 2, these equations are obtained by expressing the flatness of $\R^{n+1}$ (that is, the Riemann curvature $\widetilde{R}=0$)
with respect to the orthogonal splitting $T_x\R^{n+1} \cong T_x\Sigma \oplus {\rm span}\,(\ell_x)$.
Indeed, from the definition of the Riemann curvature, we have
\begin{align}\label{expression for bar-R}
\widetilde{R}(X,Y)Z =&\, R(X,Y)Z - K(X,\widetilde{\na}_YZ) - K(X,\widetilde{\na}_YZ)\ell + K(Y,\widetilde{\na}_XZ)\ell \nonumber \\
&\,-\widetilde{\na}_X\big(K(Y,Z)\ell\big) + \widetilde{\na}_Y\big(K(X,Z)\ell\big) + K([X,Y],Z)\ell.
\end{align}
The detailed computation can be found in \cite[\S 3]{mars},
with a slightly different sign convention for $\widetilde{R}$.
The Gauss equation is obtained by contracting with $\alpha$.
Since $\alpha(\ell)=0$, we have
\begin{align*}
0 = \alpha (\widetilde{R}(X,Y),Z)= \alpha(R(X,Y)Z) - \alpha(\widetilde{\na}_X(K(Y,Z)\ell)) + \alpha (\widetilde{\na}_Y(K(X,Z)\ell)),
\end{align*}
which yields \eqref{hypersurface gauss} by the definition of $\Psi$.

To obtain the Codazzi equation \eqref{hypersurface codazzi 1},
we consider
$$
\nn(R(X,Y)Z)=0,
$$
where $\nn$ is the normal $1$-form.
Invoking Eq. \eqref{expression for bar-R} for $\widetilde{R}$ again yields
\begin{align*}
\nn(\widetilde{R}(X,Y)Z) =&\, -K(X,{\na}_YZ) + K(Y,{\na}_XZ) - K([X,Y], Z) \nonumber\\
&\, - \nn (\widetilde{\na}_X(K(Y,Z)\ell)) + \nn(\widetilde{\na}_Y (K(X,Z)\ell)).
\end{align*}
On the other hand, the Leibniz rule of the connection gives us
\begin{align*}
\nn(\widetilde{\na}_Y(K(X,Z)\ell)) = YK(X,Z) \nn(\ell) - K(X,Z) \nn(\na_Y \ell),
\end{align*}
which, together with $\nn(\ell)=1$ and the definition of $\psi$,
implies Eq. \eqref{hypersurface codazzi 1}.

Next, we consider
$\widetilde{R}(X,Y)\ell := \widetilde{\na}_X\widetilde{\na}_Y \ell - \widetilde{\na}_Y\widetilde{\na}_X \ell + \widetilde{\na}_{[X,Y]}\ell$.
Notice that
\begin{align*}
\widetilde{\na}_X\widetilde{\na}_Y \ell
= \sum_{i=1}^n X\Psi(\theta^i, Y)e^i + X\psi(Y)\ell + \sum_{i=1}^n \Psi(\theta^i, Y) \widetilde{\na}_X e^i + \psi(Y)\widetilde{\na}_X \ell,
\end{align*}
where the following key identities are utilized:
\begin{equation*}
\widetilde{\na}_XY = \na_XY - K(X,Y)\ell,\quad
\widetilde{\na}_X\ell = \sum_{i=1}^n \Psi(\theta^i, X)e^i + \psi(X)\ell.
\end{equation*}
Then
\begin{align*}
\widetilde{\na}_X\widetilde{\na}_Y \ell
=&\, \sum_{i=1}^n X\Psi(\theta^i, Y)e^i + X\psi(Y)\ell + \sum_{i=1}^n \Psi(\theta^i, Y)\na_Xe^i \nonumber\\
&\,- \sum_{i=1}^n \Psi(\theta^i, Y) K(e^i, X)\ell + \sum_{i=1}^n \psi(Y)\psi(\theta^i, X)e^i + \psi(X)\psi(Y)\ell.
\end{align*}
A similar expression holds for $\widetilde{\na}_Y\widetilde{\na}_X\ell$ by interchanging $X$ and $Y$:
\begin{align*}
\widetilde{\na}_Y\widetilde{\na}_X \ell
= &\sum_{i=1}^n Y\Psi(\theta^i, X)e^i + Y\psi(X)\ell + \sum_{i=1}^n \Psi(\theta^i, X)\na_Ye^i \nonumber\\
&- \sum_{i=1}^n \Psi(\theta^i, X) K(e^i, Y)\ell + \sum_{i=1}^n \psi(X)\psi(\theta^i, Y)e^i + \psi(Y)\psi(X)\ell.
\end{align*}
Thus, contracting with $\alpha\in \Omega^1(\Sigma)$ and noting that $\alpha(\ell)=0$,
we conclude the Codazzi equation \eqref{hypersurface codazzi 2}.

Finally, the Codazzi equation \eqref{hypersurface codazzi 3}
is obtained by contracting $\widetilde{R}(X,Y)\ell$ with the normal $1$-form $\nn$.
Similarly to the above computations, we have
\begin{align*}
\nn(\widetilde{\na}_X\widetilde{\na}_Y \ell)
&=\nn (\widetilde{\na}_X (\sum_{i=1}^n \theta^i(\widetilde{\na}_Y \ell) e^i+\nn(\na_Y \ell)\ell))\\
&= X\psi(Y) + \sum_{i=1}^n\theta^i(\widetilde{\na}_Y \ell) \nn(\widetilde{\na}_X e^i) + \psi(Y)\psi(X)\\
&=  X\psi(Y) + \psi(Y)\psi(X) - \sum_{i=1} K(e^i, X)\Psi(\theta^i, Y),
\end{align*}
thanks to another important identity:
\begin{equation}\label{identity for K}
\nn(\widetilde{\na}_XY)=-\nn(K(Y,X)\ell) = -K(Y,X).
\end{equation}
Therefore, computing for $\nn(\widetilde{\na}_Y\widetilde{\na}_X\ell)$ in the similar manner:
\begin{align*}
\nn(\widetilde{\na}_Y\widetilde{\na}_X \ell)=Y\psi(X) + \psi(X)\psi(Y)-\sum_{i=1}K(e^i,Y)\Psi(\theta^i,X),
\end{align*}
we can deduce Eq. \eqref{hypersurface codazzi 3}.
Furthermore, observe that the above computations still hold  in the sense of distributions
for immersions with lower regularity, {\it i.e.},  $\iota\in W^{2,p}_{\loc}(\Sigma; \R^{n+1})$.
This proves the first part of the theorem.

\medskip
\noindent
{\bf 2.}  Now we tackle the {\it realization problem}, {\it i.e.}, finding
an immersion $\iota$ from
Eqs. \eqref{hypersurface gauss}--\eqref{hypersurface codazzi 3}.
As in the semi-Riemannian submanifolds case, the key is to verify
the second structural system \eqref{eqn_second structural eqn} for a suitable connection $1$-form.

For this purpose, we invoke the following identity for differential forms:
\begin{equation*}
\dd\beta (X,Y) = X\beta(Y)-Y\beta(X) + \beta([X,Y]),
\end{equation*}
where $\beta\in \Gamma(T^*\Sigma)$ is arbitrary.
Thus, we can  rewrite the three Codazzi equations as
\begin{equation}\label{rewritten gauss-codazzi}
\begin{cases}
\dd K(X,Y,Z) = K(Y, {\na}_XZ) - K(X,{\na}_YZ) + \psi(X) K(Y,Z) - K(X,Z)\psi(Y),\\[1mm]
\dd\Psi(\alpha, X,Y) = \psi(X)\Psi(\alpha, Y) - \psi(Y)\Psi(\alpha,X)
 \\[1mm]
\qquad\qquad\qquad\,\,\, + \sum_{i=1}^n \big(\Psi(\theta^i, X)\alpha(\na_Y e^i) - \Psi(\theta^i, Y)\alpha(\na_X e^i)\big),\\[1mm]
\dd\psi(X,Y) = \sum_{i=1}^n \big(\Psi(\theta^i, Y) K(X, e^i) - \Psi(\theta^i, X) K(Y, e^i)  \big).
\end{cases}
\end{equation}

Now, define the connection $1$-form $\W_\Sigma \in W^{1,p}_\loc(\Sigma;\,T^*\Sigma \otimes \mathfrak{gl}(n+1; \R))$ by
\begin{equation}\label{matrix}
\W_\Sigma := \begin{bmatrix}
\Gamma & \Psi \\
K& \psi
\end{bmatrix}.
\end{equation}
More precisely, in the local coordinates, we write
\begin{equation*}
\W_\Sigma := \begin{bmatrix}
\Gamma^k_{ij} \theta^k & \Psi(\theta^i, \cdot)\\[2mm]
K(e^i, \cdot) & \psi(\cdot)
\end{bmatrix},
\end{equation*}
where, as usual, the Christoffel symbols are defined via
$\na_{e^i}e^j = \sum_{k=1}^n \Gamma^k_{ij} e^k$ and  computed
from $\Gamma^k_{ij}=\frac{1}{2}g^{kl}(\partial_i g_{jl}+\partial_j g_{li} - \partial_l g_{ij})$.
The block-matrix representation of $\W$ in Eq. \eqref{matrix} is interpreted via the following identifications:
\begin{equation*}
\begin{cases}
\Gamma =\Gamma^k_{ij} \theta^k \in W^{1,p}_\loc(\Sigma\,;\,T^*\Sigma \otimes \mathfrak{gl}(n;\R)),\\[1mm]
\Psi, K \in W^{1,p}_{\loc}(\Sigma\,;\,T^*\Sigma \otimes \R^n),\\[1mm]
\psi \in W^{1,p}_\loc (\Sigma\,;\,T^*\Sigma).
\end{cases}
\end{equation*}
Thus, we can recast the Gauss equation \eqref{hypersurface gauss} and the Codazzi equations
in the form of \eqref{rewritten gauss-codazzi} into the following schematic equalities:
\begin{equation*}
\begin{cases}
\dd K = K\Gamma - \Gamma K - K\psi + \psi K = K\wedge \Gamma - K \wedge \psi,\\
\dd \Psi = \Gamma\Psi - \Psi \Gamma + \Psi \psi - \psi \Psi =  \Gamma \wedge \Psi - \Psi \wedge \psi,\\
\dd \psi = K\Psi - \Psi K = K\wedge \Psi,
\end{cases}
\end{equation*}
where the juxtaposition of matrices ({\it e.g.},\, $K\Gamma$) denotes the matrix multiplication,
and $\wedge$ is an intertwining of the wedge product on differential forms and the matrix multiplication.

On the other hand, simple manipulations on block matrices lead to
\begin{equation}\label{W Sigma wedge W Sigma}
\W_\Sigma\wedge \W_\Sigma = \begin{bmatrix}
\Gamma \wedge \Gamma  + \Psi \wedge K && \Gamma\wedge \Psi + \psi \wedge \Psi \\
K\wedge \Gamma + K\wedge \psi && K\wedge \Psi
\end{bmatrix}.
\end{equation}
In this notation, the Riemann curvature is given by
\begin{equation*}
R=\dd\Gamma - \Gamma \wedge \Gamma \in L^p_\loc(\Sigma\,;\,{\wedge}^2\, T^*\Sigma \otimes \mathfrak{gl}(n;\R)).
\end{equation*}
Then the preceding two equations yield
\begin{equation}\label{eq_sigma second structure}
\dd\W_\Sigma - \W_\Sigma\wedge \W_\Sigma =0,
\end{equation}
{\it i.e.},\, the second structural system
as in Eq. \eqref{eqn_second structural eqn}.

Invoking again Lemma \ref{lemma mardare},
we obtain the local solution $A\in W^{1,p}_\loc(U\subset \Sigma; \mathfrak{gl}(n+1,\R))$
to the following Pfaff system:
\begin{equation*}
\dd A = \W_\Sigma\cdot A,
\end{equation*}
where $U \subset \Sigma$ is an open trivialized neighborhood.

\medskip
\noindent
{\bf 3.} Now we solve the local isometric immersion $\iota: \Sigma \emb \R^{n+1}$ via the Poincar\'{e} system:
\begin{equation*}
\dd\iota = \tilde{\theta}\cdot A,
\end{equation*}
where $\tilde{\theta} = (\theta^1, \ldots, \theta^n, 0)^\top : U \subset \Sigma \rightarrow \R^{n+1}\otimes T^*\Sigma$
is the $\R^{n+1}$-valued differential $1$-form.
As before, it is solvable if and only if the following first structural system is satisfied:
$$
\dd\tilde{\theta} = \tilde{\theta}\wedge \W_\Sigma.
$$
Recall that the first structural system holds whenever the affine connection $\na =\iota^*\widetilde{\na}$ is torsion-free (see Appendix A.5).
Here, as $K(X,Y)=K(Y,X)$ ({\it cf.} \cite{lefloch1,mars}),
the torsion-free condition is verified, which leads to the existence of a solution $\iota \in W^{2,p}_\loc(U; \R^{n+1})$.

The assertion now follows from the proof of Theorem \ref{theorem_main theorem, isometric immersions and GCR}. This completes the proof.

\begin{remark}
$\,$ Theorem {\rm \ref{theorem for hypersurface}} was proved locally in \cite{lefloch2} by computations in the local coordinates.
Our proof above, being global and intrinsic in nature,
both helps clarify the geometric meanings of $\{K, \Psi, \psi\}$ and serves as a crucial step towards the establishment of
the weak rigidity theorem, Theorem {\rm \ref{theorem_weak rigidity of hypersurfaces}}, for general hypersurfaces below.
\end{remark}

In the proof above, it is crucial to establish the equivalence of the Gauss--Codazzi equations
\eqref{hypersurface gauss}--\eqref{hypersurface codazzi 3}
with Eq. \eqref{eq_sigma second structure}, namely the second structural system for $\W_\Sigma$,
which is defined in Eq. \eqref{matrix} in terms of the Christoffel symbol $\Gamma$ and the intrinsic,
extrinsic, and normal second fundamental forms $\{K, \Psi, \psi\}$.
Therefore, by invoking the quadratic theorem (Theorem \ref{thm: generalized quadratic theorem on manifolds})
and establishing the weak continuity of $\dd\W_\Sigma = \W_\Sigma \wedge \W_\Sigma$ again,
we arrive at the weak rigidity theorem for the general hypersurfaces:

\begin{theorem}\label{theorem_weak rigidity of hypersurfaces}
Let $(\Sigma,g)$ be a simply-connected $n$-dimensional hypersurface of semi-Euclidean
space $\R^{n+1}$ with $\ind(\Sigma)=\nu$ and $g\in W^{1,p}_\loc (\Sigma,  O(\nu, n-\nu))$ for $p>n$.
Let $\{f_\varepsilon\}$ be a family of immersions of semi-Riemannian submanifolds uniformly bounded
in $W^{2,p}_\loc(\Sigma; \R^{n+k})$, and let $\{l_\varepsilon\}$ be an associated family of rigging vector fields
uniformly bounded in $W^{1,p}_{\loc}(\Sigma; T\Sigma)$.
Denote by $\{K_\varepsilon, \Psi_\varepsilon, \psi_\varepsilon\}$ the corresponding intrinsic, extrinsic,
and normal second fundamental forms.
Then, after passing to a subsequence if necessary, $\{f_\varepsilon\}$ converges to
an immersion $f \in W^{2,p}_\loc(\Sigma;\R^{n+1})$ in the sense of distributions{\rm ;} in addition,
its intrinsic, extrinsic, and normal second fundamental forms are weak
limits in $L^p_\loc$ of $\{K_\varepsilon\}$, $\{\Psi_\varepsilon\}$, and $\{\psi_\varepsilon\}$, respectively.
\end{theorem}

\begin{proof}$\,$
First, thanks to Eq. \eqref{W Sigma wedge W Sigma},
all the entries of the $2$--form-valued matrix $\W^\Sigma \wedge \W^\Sigma$ are linear combinations
of the quadratic forms in $\Gamma, \Psi, K$, and $\psi$,
each of which lies in $W^{1,p}_\loc$.
Thus, $\W^\Sigma \wedge \W^\Sigma \in W^{1, \frac{p}{2}}_\loc$ by the Cauchy--Schwarz inequality,
which is compactly embedded in $W^{-1, q}_\loc$ for some $1<q<2$,
as computed in Step $3$ of the proof of Theorem \ref{theorem_weak rigidity}.
	
On the other hand, $\W^\Sigma\in W^{1,p}_\loc$ implies that $\dd\W^\Sigma \in L^p_\loc$,
which is compactly embedded into $W^{-1,p}_\loc$  by the Rellich lemma.
Using Eq. \eqref{W Sigma wedge W Sigma} and the interpolations of Sobolev spaces,
we deduce that $\{\dd\W^\Sigma_\varepsilon\}$ is pre-compact in $H^{-1}_{\loc}$.
	
Therefore, with the above pre-compactness result, the proof proceeds
as that for Theorem \ref{theorem_weak rigidity}.
In particular, we establish the weak continuity of the Cartan structural system $\dd\W^\Sigma = \W^\Sigma \wedge \W^\Sigma$.
Then, in view of the realization theorem (Theorem \ref{theorem for hypersurface}) for general hypersurfaces,
it implies the existence of the limiting immersion $f$,
together with a rigging vector field $\ell$, for which the intrinsic, extrinsic, and normal second fundamental
forms $\{K, \Psi, \psi\}$ are well-defined. After passing to a subsequence if necessary,
$\{K_\varepsilon, \Psi_\varepsilon, \psi_\varepsilon\}$ converges in the weak $L^p_\loc$ topology to $\{K, \Psi, \psi\}$ due to the uniqueness
of weak limits. Then the proof is completed.
\end{proof}

\medskip
\begin{appendices}
\section{Proofs of Several Semi-Riemannian Geometric Theorems}
\numberwithin{equation}{section}

In this appendix, we provide the proofs of several semi-Riemannian geometric theorems,
whose Riemannian analogues are well-known.
These results are viewed as {\it folklores} in the geometric community,
but the proofs appear elusive in the literature.
For the convenience of the reader, now we carefully write down the complete proofs in detail below.

\subsection{\bf Proof of Theorem~\ref{theorem GCR equations}}
The derivation of Eqs.  \eqref{gauss}--\eqref{codazzi} can be found
on page \,$100$ and page \,$115$ in \cite{oneill}, respectively.
It remains to derive the Ricci equation \eqref{ricci}.
Indeed, we have
\begin{align*}
0&= \widetilde{R}(X,Y,\xi,\eta)\\
&= \langle\widetilde{\na}_X\widetilde{\na}_Y\xi,\eta\rangle - \langle\widetilde{\na}_Y\widetilde{\na}_X\xi,\eta\rangle + \langle\widetilde{\na}_{[X,Y]}\xi,\eta\rangle \\
&= \langle\np_X\np_Y\xi,\eta\rangle - \langle\np_X(S_\xi Y),\eta\rangle - \langle\np_Y\np_X\xi,\eta\rangle+\langle \np_Y(S_\xi X),\eta\rangle
+ \langle \np_{[X,Y]}\xi,\eta\rangle\\
&=R^\perp(X,Y,\xi,\eta) - \langle\np_X(S_\xi Y),\eta\rangle +\langle \np_Y(S_\xi X),\eta\rangle,
\end{align*}
in view of the definition for $R^\perp$.
Moreover, owing to the self-adjointness of $S_{\eta}$,
we have
\begin{equation*}
\langle\np_X(S_\xi Y),\eta\rangle =X\langle S_\xi Y,\eta\rangle - \langle S_\xi Y, \ta (\widetilde{\na}_X\eta)\rangle
=  \langle S_\xi Y, S_\eta X\rangle = \langle S_\eta\circ S_\xi(Y), X\rangle,
\end{equation*}
and similarly $\langle\np_Y(S_\xi X),\eta\rangle =\langle S_\xi\circ S_\eta(Y), X\rangle$.
Then Eq. \eqref{ricci} follows.

\subsection{Proof of Lemma~\ref{lemma_skew symmetry of W}}
The connection $1$-form $\W$ (Definition \ref{def of W}) is semi-skew-symmetric.
It is crucial to observe that a matrix $B\in O(\nu, n-\nu)$ if and only if its transpose $B^\top$ takes the form:
\begin{equation}\label{eqn_semi orthogonal group}
B^\top =
{\e_{n,\nu}} B^{-1}
{\e_{n,\nu}}.
\end{equation}
The signature matrix
${\e_{n,\nu}}$ is defined in Eq. \eqref{signature matrix}.

We first observe that, for each $Z \in O(\nu, n-\nu)$,
\begin{equation}\label{eqn tangent space of special orthogonal group}
	T_Z O(\nu, n-\nu) = \big\{A\in \mathfrak{gl}(n;\R)\,:\, Z
{\e_{n,\nu}} A^\top + A
{\e_{n,\nu}} Z^\top=0\big\}.
\end{equation}
Indeed, let $\sigma:(-\varepsilon,\varepsilon)\rightarrow \onu$ be a $C^1$-curve
with $\sigma(0)=Z$.
In view of Eq. \eqref{eqn_semi orthogonal group}, we have
$$
\sigma(t)^\top =
{\e_{n,\nu}} \sigma(t)^{-1}
{\e_{n,\nu}}.
$$
Taking the derivative in $t$ yields
\begin{equation*}
\dot\sigma(t)^\top = -
{\e_{n,\nu}} \sigma(t)^{-1} \dot{\sigma}(t) \sigma(t)^{-1}
{\e_{n,\nu}}.
\end{equation*}
Thus, evaluating the above at $t=0$ and using the identity:
$Z^\top=
{\e_{n,\nu}} Z^{-1}
{\e_{n,\nu}}$,
we have
\begin{equation*}
\dot{\sigma}(0)^\top = -Z^\top
{\e_{n,\nu}} \dot{\sigma}(0)
{\e_{n,\nu}} Z^\top.
\end{equation*}
Since the elements of $T_Z\onu$ are in one-to-one correspondence with $\dot{\sigma}(0)$ for such $\sigma$,
Eq. \eqref{eqn tangent space of special orthogonal group} follows. As a consequence,
\begin{equation*}
\frako(\nu,n-\nu) := T_{\id}\onu = \big\{A\in \mathfrak{gl}(n;\R)\,:\,
{\e_{n,\nu}} A^\top + A
{\e_{n,\nu}} =0\big\}.
\end{equation*}

We now verify that $\W$ lies in the Lie algebra of the semi-orthogonal group.
Clearly, it suffices to prove that, for each $a,b\in \{1,\ldots,n+k\}$,
\begin{equation*}
\epsilon^a \omega_a^b = -\epsilon^b \omega_b^a.
\end{equation*}
Indeed, for $a=i$ and $b=\alpha$, this follows by the definition of $\omega^\alpha_i$
in the fourth equation of \eqref{eqn-connection one form}.
For $a=i$ and $b=j$, as $\na$ is compatible with metric $g$,
we deduce from the first equation  of \eqref{eqn-connection one form} that
\begin{align*}
0=\p_l\langle \p_i,\p_j\rangle = \langle \na_{\p_l}\p_i,\p_j\rangle + \langle \p_i, \na_{\p_l}\p_j\rangle
= \epsilon^j \omega^i_j(\p_l) + \epsilon^i \omega^j_i(\p_l).
\end{align*}
Finally, for $a=\alpha$ and $b=\beta$, it follows from a similar computation by using the third equation
of \eqref{eqn-connection one form}, thanks to the compatibility of $\na^E$ with $g^E$.
This completes the proof.

\subsection{\, \bf Proof of Proposition~\ref{prop_second structural equation}}
We divide the arguments into four steps.

\medskip
\noindent
{\bf 1.} We begin by observing that the definition of the connection $1$-form $\W$, {\it i.e.}, Eq. \eqref{eqn-connection one form},
implies that
\begin{equation*}
\na_{\p_i}\p_j = \sum_{l}\omega^l_j(\p_i)\p_l,
\quad \two(\p_i,\p_j) = \sum_\alpha \omega^i_\alpha (\p_j) \p_\alpha,
\quad \na^E_{\p_i}{\p_\alpha} = \sum_\beta \omega^\alpha_\beta (\p_i)\p_\beta.
\end{equation*}
One may deduce the following identities of the shape operator $S$:
\begin{equation*}
S_{\p_i}\p_\alpha = \sum_j \epsilon^j \langle S_{\p_i}\p_\alpha,\p_j\rangle\p_j
= \sum_j \epsilon^j \langle \two (\p_i,\p_j),\p_\alpha\rangle \p_j =  \sum_j \epsilon^j \epsilon^\alpha\omega^i_\alpha(\p_j)\p_j.
\end{equation*}

\smallskip
\noindent
{\bf 2.} Next, the Gauss equation \eqref{gauss} is equivalent to
\begin{equation}\label{z}
R(\p_i, \p_j, \p_k) = S_{\p_i} \two(\p_j,\p_k) - S_{\p_j}\two (\p_i,\p_k).
\end{equation}
Applying the symmetry of $\two$ twice (in the first and third equalities below), we obtain
\begin{align*}
S_{\p_i} \two(\p_j,\p_k) - S_{\p_j}\two (\p_i,\p_k) &= S_{\p_i} \two(\p_k,\p_j) - S_{\p_j}\two (\p_k,\p_i)\\
&= S_{\p_i}\big(\sum_\alpha \omega^k_\alpha (\p_j) \p_\alpha\big) - S_{\p_j} \big(\sum_\alpha \omega^k_\alpha (\p_i)\p_\alpha\big)\\
&=\sum_\alpha\sum_l \epsilon^\alpha\epsilon^l \big(\omega_\alpha^k(\p_j)\omega^l_\alpha(\p_i) - \omega^k_\alpha(\p_i)\omega^l_\alpha(\p_j) \big)\p_l\\
&= \sum_\alpha\sum_l (\omega^k_\alpha\wedge \omega^\alpha_l) (\p_i,\p_j) \p_l,
\end{align*}
where the last equality follows from Lemma \ref{lemma_skew symmetry of W}.
On the other hand, the Riemann curvature of the Levi--Civita connection on $TM$ is computed directly
from the definition:
\begin{align*}
R(\p_i,\p_j,\p_k)
&:= \na_i\na_j\p_k - \na_j\na_i\p_k + \na_{[\p_i,\p_j]}\p_k \\
&= \sum_{l}\big\{\p_i(\omega^l_k(\p_j))\p_l - \p_j(\omega^l_k(\p_i))\p_l + \omega^l_k([\p_i,\p_j])\p_l\\
&\qquad\quad\, + \omega^l_k(\p_j)\sum_s\omega^s_l(\p_i)\p_s - \omega^l_k(\p_i)\sum_s \omega^s_l(\p_j)\p_s\big\}\\
&= \sum_s \big\{\dd\omega^s_k - \sum_l\omega^l_k\wedge\omega^s_l\big\}(\p_i,\p_j) \p_s.
\end{align*}
Equating the preceding  computations via Eq. \eqref{z}, we conclude that
$$
\dd\omega^s_k =\sum_b \omega^k_b\wedge \omega^b_s.
$$

\medskip
\noindent
{\bf 3.}
Applying the same argument to $R^E(\p_i,\p_j,\p_\gamma)$ and utilizing the Ricci equation \eqref{ricci},
we deduce that $\dd\omega^\alpha_\beta = \sum_b \omega^\alpha_b\wedge \omega^b_\beta$.

Furthermore, starting with the Codazzi equations \eqref{codazzi}, we have
\begin{align}\label{eqn zz}
0=&\,\widetilde{\na}_{\p_i} \two(\p_j,\p_k) - \widetilde{\na}_{\p_j} \two(\p_i,\p_k)\nonumber\\
=&\, \sum_\gamma \p_i(\omega^j_\gamma(\p_k))\p_\gamma + \sum_\beta\omega^j_\beta(\p_k) \sum_\gamma \omega^\beta_\gamma(\p_i) \p_\gamma
  - \sum_\gamma \p_j(\omega^i_\gamma(\p_k))\p_\gamma \nonumber\\
  &\, -\sum_\beta\omega^i_\beta(\p_k) \sum_\gamma \omega^\beta_\gamma(\p_j) \p_\gamma \nonumber\\
=&\, \sum_\gamma\Big\{\p_i(\omega^k_\gamma(\p_j))-\p_j(\omega^k_\gamma(\p_i))
-\sum_\beta\big(\omega^k_\beta(\p_i)\omega^\beta_\gamma(\p_j) - \omega^k_\beta(\p_j) \omega^\beta_\gamma(\p_i)\big)\Big\}\p_\gamma\nonumber\\
=&\, \sum_\gamma\Big\{\dd\omega^k_\gamma(\p_i,\p_j) - \omega^k_\gamma [\p_i,\p_j] - \sum_\beta (\omega^k_\beta\wedge\omega^\beta_\gamma) (\p_i,\p_j)  \Big\}\p_\gamma.
\end{align}
In the penultimate equality, we have used the self-adjointness of $\two$, {\it i.e.}, $\omega^i_\alpha(\p_j)=\omega^j_\alpha(\p_i)$.
The final equality follows from the definition of $\dd\omega^k_\gamma$ and $\omega^k_\beta\wedge\omega^\beta_\gamma$.

To compute the Lie bracket term in the last equality of Eq. \eqref{eqn zz},
we invoke the torsion-free condition of the affine connection:
\begin{align*}
\sum_\gamma \omega^k_\gamma [\p_i,\p_j] \p_\gamma
&=\sum_\gamma \omega^k_\gamma \big(\na_{\p_i}\p_j - \na_{\p_j}\p_i\big)\p_\gamma
\\&
= \sum_\gamma\sum_l \omega^k_\gamma \big(\omega^l_j(\p_i) \p_l -\omega^l_i(\p_j) \p_l \big)\p_\gamma\\
&= \sum_\gamma\sum_l  \big(\omega^l_j(\p_i)\omega^l_\gamma(\p_k) - \omega^l_i(\p_j)\omega^l_\gamma(\p_k)\big) \p_\gamma
\\&
= \sum_\gamma\sum_l (\omega_l^k\wedge \omega_\gamma^l)(\p_i,\p_j) \p_\gamma,
\end{align*}
again owing to the symmetries of the second fundamental form and $\W\wedge \W$.
Substituting it back to Eq. \eqref{eqn zz} yields that $\dd\omega^k_\gamma = \sum_b \omega^k_b \wedge \omega^b_\gamma$.

\medskip
\noindent
{\bf 4.} Combining Steps $1$--$3$ together, we conclude
\begin{equation*}
\dd\omega^a_c =\sum_b \omega^a_b \wedge \omega_c^b.
\end{equation*}
Moreover, as an equation on $\Omega^2\big(\littleonk\big)$,
Eq. \eqref{eqn_second structural eqn} is independent of the choice of moving frames.
This completes the proof.

\subsection{\, \bf Derivation of the First Structural System \eqref{eqn_first structure eq}}

We now present a derivation of the first structural system \eqref{eqn_first structure eq} and
show that it is equivalent to the torsion-free condition of the affine connection.

We compute the Lie bracket of the basic vector fields $\p_i$ and $\p_j$ in two different ways.
On one hand, we have
\begin{equation*}
[\p_i,\p_j] = \sum_l \epsilon^l \theta^l[\p_i,\p_j] \p_l.
\end{equation*}
On the other hand, the torsion-free condition of $\na$ gives
\begin{align*}
[\p_i,\p_j] &= \na_{\p_i}\p_j - \na_{\p_j}\p_i \\
&= \na_{\p_i}\big(\epsilon^j \sum_k \theta^k(\p_j)\p_k\big) - \na_{\p_j}\big(\epsilon^i \sum_k\theta^k(\p_i)\p_k\big)\\
&= \sum_k \big( \epsilon^k \delta^k_j \na_{\p_i}\p_k - \epsilon^k\delta^k_i \na_{\p_j}\p_k\big)\\
&= \sum_l\epsilon^l\big(\omega^i_l(\p_j) - \omega^j_l(\p_i) \big)\partial_l,
\end{align*}
where the last equality follows from the {\it semi-skew-symmetry} of the connection $1$-form
$\W$ (Lemma \ref{lemma_skew symmetry of W}). Then
\begin{equation}\label{eqn_x}
\theta^l[\p_i,\p_j] = \omega^i_l(\p_j) - \omega^j_l(\p_i).
\end{equation}
Now we observe
\begin{equation*}
\dd\theta^l (\p_i,\p_j) = \p_i(\theta^l(\p_j)) - \p_j(\theta^l(\p_i)) + \theta^l[\p_i,\p_j]=\theta^l[\p_i,\p_j],
\end{equation*}
and
\begin{align*}
\sum_k (\theta^k\wedge \omega^l_k )(\p_i,\p_j)
&= \sum_k\big(\theta^k(\p_i) \omega^l_k(\p_j) -\theta^k(\p_j)\omega^l_k(\p_i) \big)\\
&=  \sum_k \delta^k_i \omega^l_k(\p_j) - \delta^k_j \omega^l_k(\p_i) =\omega^l_i(\p_j) - \omega^l_j(\p_i).
\end{align*}
Utilizing Eq. \eqref{eqn_x}, we obtain
$$
\dd\theta^l = \sum_k \theta^k\wedge \omega^l_k.
$$
As Eq.  \eqref{eqn_first structure eq} is independent of the choice of local coordinates,
This completes the proof.

\subsection{\,\bf Proof of Theorem~\ref{theorem_main theorem, isometric immersions and GCR} in the $C^\infty$ Case}
We now present a proof of the realization theorem in the $C^\infty$ case,
following Tenenblat's arguments in \cite{Ten71} for the Riemannian case.
We emphasize that various modifications are necessary due to the semi-Riemannian geometry.
We divide the arguments into four steps.

\medskip
{\bf 1.}  We start with solving a {\it Pfaff system} for the bundle connection $A$ on $TM\oplus E$.
More precisely, we show that, for any $x_0\in M$,
the following initial value problem for first-order PDEs:
\begin{equation}\label{eqn pfaff system}
\W= \dd A\cdot A^{-1},\qquad A(0)=A(x_0)\in \bigonk,
\end{equation}
has a solution $A\in C^\infty(U;\bigonk)$
in some neighborhood $U$ of $x_0$.

Indeed, without loss of generality, assume that $x_0=0$ in the local coordinate $\{\p_i\}_1^n$.
Also, take $Z=\{Z^a_b\}$ as the canonical frame field on $\glnk\cong\R^{(n+k)^2}$,
with signature inherited from $\widetilde{M}=\R^{n+k}_{\nu+\tau}$.
For example, $Z:=\gzero$ is one suitable choice.
Motivated by \cite{Ten71},  we consider the following map:
\begin{align*}
\lxz: T_xM \times T_Z\bigonk &\longrightarrow T_Z\bigonk, \\
(X,\mm) \quad&\longmapsto \quad \dd_xZ(\mm) + Z\cdot \W(X)|_{x},
\end{align*}
which is abbreviated in the sequel as
\begin{equation}\label{eqn definition of Lambda}
\lxz = \dd Z-\W\cdot Z.
\end{equation}
Using the characterization of tangent spaces of the semi-orthogonal group and
its Lie algebra ({\it cf.} the proof of Lemma \ref{lemma_skew symmetry of W}),
we see that $\lxz$ is well-defined. Indeed,
\begin{align*}
&Z\gzero (\lxz(X,\mm))^\top + \lxz(X,\mm) \gzero  Z^\top\\
&= Z\gzero  ( \dd Z(\mm)^\top - Z^\top \W(X)^\top) +\big(\dd Z(\mm) - \W(X) Z\big)\gzero  Z^\top \\
&= -\big(\gzero \W(X)^\top +  \W(X)\gzero \big)+ \big(Z\gzero \dd Z(\mm)^\top
  + \dd Z(\mm)\gzero  Z^\top\big) \\
&= 0,
\end{align*}
since
$$
\dd Z(\mm)\in T_Z\bigonk, \quad \W(X)\in \littleonk.
$$

Next, we define the following {\it distribution} in the Frobenius sense:
\begin{equation*}
\dxz := \ker (\lxz) \subset \, T_xM\times T_Z\bigonk.
\end{equation*}
Our goal is to show that it is {\em completely integrable}.
Assuming so, we can find the unique maximal integral submanifold
in some neighborhood of $x_0$. Notice that
\begin{equation*}
\mathcal{D}^{(0,Z)} \cap \big(\{0\}\oplus T_Z\bigonk\big) =\{0\},
\end{equation*}
{\it i.e.}, the distribution is transverse to the $TM$ factor at point $x_0 = 0$,
because
\begin{equation*}
\lxz (0,\mm) = \mm.
\end{equation*}
In view of the classical implicit function theorem,
$\dxz$ is locally a graph of a smooth function $A$ from $TU$ to $T\bigonk$,
with $x$ lies in $U$,
an open neighborhood of $x_0$.
This function $A$ solves the Pfaff system \eqref{eqn pfaff system} in view of the definition of $\lxz$.
	
\medskip
\noindent
{\bf 2.} It now remains to prove the complete integrability of distribution $\dxz$.
By the {\em Frobenius theorem}, we show that $\dxz$ is {\em involutive}.
That is, for any $(X_i,\mm_i) \in \dxz$ for $i=1,2$, the commutator stays
in $\dxz$:
$$
\lxz [(X_1,\mm_1), (X_2, \mm_2)] =0.
$$

Indeed, utilizing the following identity for the exterior differential:
\begin{equation*}
\dd\lxz (s_1,s_2) = s_1 (\lxz  s_2) -s_2 (\lxz s_1)  - \lxz [s_1,s_2]
\end{equation*}
for $s_1,s_2\in T_xU \oplus T_Z\bigonk$, we reduce the problem to proving the identity:
\begin{equation}\label{step 3 in main theorem}
\dd\lxz((X_1,\mm_1),(X_2,\mm_2))=0.
\end{equation}
To this end, we compute $\dd\lxz$.
Since
$$
\Lambda=\dd Z-\W\cdot Z,
$$
we have
\begin{align*}
\dd\Lambda &= \dd(\dd Z) -\dd\W\cdot Z + \W\wedge (\dd Z)\\
&=- \W\wedge \W \cdot Z + \W \wedge (\Lambda + \W\cdot Z) =\W\wedge \Lambda,
\end{align*}
where we have used the second structural
system \eqref{eqn_second structural eqn},
together with the definition of $\Lambda$ in Eq. \eqref{eqn definition of Lambda},
for the second equality.
As $(X_i,\mm_i)\in\dxz$ for $i=1,2$, we then have
\begin{align*}
&\dd\lxz ((X_1,\mm_1),(X_2,\mm_2))\\
&= (\W|_x\wedge \lxz)((X_1,\mm_1),(X_2,\mm_2))\\
&= \W(X_1,\mm_1)|_x \lxz(X_2,\mm_2)  -\W(X_2,\mm_2)|_x \,\lxz(X_1,\mm_1)\\
&= 0.
\end{align*}
This completes the proof of Eq. \eqref{step 3 in main theorem},
which implies that the Pfaff system \eqref{eqn pfaff system} is solvable.

\medskip
\noindent
{\bf 3.}
Now, define
\begin{equation*}
\tth=(\theta^1,\ldots, \theta^n, 0,\ldots,0)^\top\in \Omega^1(\R^{n+k})
\end{equation*}	
and, for $x_0\in M$, consider the {\em Poincar\'{e}} system:
\begin{equation}\label{eqn_poincare}
\dd f =  \tth \cdot A, \qquad f(x_0)=f_0,
\end{equation}
where $f_0\in C^\infty(M;\widetilde{M})$ and
$\dd_{x_0}f_0(v)\neq 0$ for $v\in T_{x_0}M\setminus\{0\}$.
Suppose that this system is solvable. Then, as $A$ takes values in $\bigonk$,
$\det(A)=\pm 1$, by using  Eq. \eqref{eqn_semi orthogonal group}.
In particular, $A$ is invertible. It follows from Eq. \eqref{eqn_poincare}
that the linear map $\dd f$ has rank $n$, so that $f$ is an immersion indeed.

Solving for $f$ from Eq. \eqref{eqn_poincare} is equivalent to showing that $ \tth \cdot A$ is
an exact $1$-form.
For  simply-connected $M$, by the {\em Poincar\'{e} lemma}, it suffices to
verify that $\dd( \tth \cdot A)=0$; that is, it is a closed $1$-form. Indeed,
\begin{equation*}
\dd( \tth \cdot A)=\dd\tth \cdot A - \tth\wedge \dd A,
\end{equation*}
and we can compute the second term by $\tth\wedge \dd A = (\tth \wedge \W) \cdot A$,
thanks to the Pfaff system \eqref{eqn pfaff system} solved in Steps 1--2 above.
Thus, the exactness of $\tth\cdot A$ follows directly from
\begin{equation*}
\dd\tth = \tth\wedge \W,
\end{equation*}
which is just the first structural system \eqref{eqn_first structure eq}.
Thus, we have established the solvability of the initial value problem
for the Poincar\'{e} system \eqref{eqn_poincare}.

\medskip
\noindent
{\bf 4.}
With
the immersion $f$ from the Poincar\'{e} system,
we now identify the normal bundle  $TM^\perp:={f^*T\widetilde{M}}/{TM}$ with the given bundle $E$,
and identify the second fundamental form induced by $f$ with the given symmetric tensor field $\two$.
Moreover, we can deduce the uniqueness of the local immersion up to the rigid motions of $\R^{n+k}_{\nu+\tau}$, {\it i.e.},
modulo the actions by the semi-Riemannian congruence group $\R^{n+k}\rtimes \bigonk$.

\medskip
{\bf 4(a)}. First of all, define an orthonormal frame $\{\pb_a\}_1^{n+k}$ on $T\widetilde{M}$ via maps $f$ and $A$
solved by the Pfaff and Poincar\'{e} systems.
We denote by $\big\{\frac{\partial}{\partial Z^a}\big\}_1^{n+k}$ the canonical orthonormal basis
on $\widetilde{M}=\R^{n+k}_{\nu+\tau}$ with respect to $\gzero=
{\e_{n\,\nu}}\oplus
{\e_{k\,\tau}}$.
In this basis, we set
\begin{equation}\label{eqn define the barred frame on M-bar}
\pb_i:=\dd f(\p_i),\qquad
\pb_\alpha :=\sum_b A^\alpha_b\frac{\partial}{\partial Z^b},
\end{equation}
where the definition of $\{\pb_a\}$ is independent of the choice of bases on $\widetilde{M}$: Recall that,
for each $x\in M$, $A(x)$  lies in $\bigonk \subset \mathfrak{gl}(n+k;\R)\cong\text{End} \,T\widetilde{M}$,
the group of linear transformations on $T\widetilde{M}$.
Using a further identification: $\text{End}\, T\widetilde{M} \cong T^*\widetilde{M}\otimes T\widetilde{M}$,
we view $A(x)$ as a linear map from $T\widetilde{M}$ to itself.
From this perspective, $\pb_\alpha$ coincides with $A^\alpha$, {\it i.e.}, the normal component of $A$
as a $T\widetilde{M}$-valued function defined on $M\times T\widetilde{M}$.
Thus, Eq.  \eqref{eqn define the barred frame on M-bar} is equivalent to
\begin{equation}\label{shorthand}
\pb= (\dd f)^\sharp \oplus \nor A,
\end{equation}
where $\sharp: T^*\widetilde{M} \rightarrow T\widetilde{M}$ is the canonical bundle isomorphism
turning a $1$-form into the corresponding vector field.
This gives us an intrinsic definition of frame $\{\widetilde{\p}_a\}$.

Now we verify that $\{\pb_a\}$ is indeed an orthonormal frame.
First, using the Poincar\'{e} system \eqref{eqn_poincare} defining $f$,
together with the characterization of the semi-orthogonal
group ({\it cf.} Eq. \eqref{eqn_semi orthogonal group}),
we have
\begin{eqnarray*}
\gzero(\pb_i,\pb_j) &=& \gzero(\dd f(\p_i), \dd f(\p_j))
  = \gzero(\tth(\p_i) A, \tth(\p_j)A)\\[1mm]
&=& \tth(\p_i) A^\top \gzero A \tth(\p_j)^\top = \tth(\p_i) \gzero \tth(\p_j)^\top\\[1mm]
&=& (\epsilon_{n,\nu})^i_j :=\epsilon^i \delta^i_j.
\end{eqnarray*}
Also, for the normal directions, using the shorthand notations in Eq. \eqref{shorthand}, we have
\begin{align*}
\gzero(\pb_\alpha,\pb_\beta) = (\gzero (\nor A, \nor A))^\alpha_\beta
= \nor (A^\top \gzero A)^\alpha_\beta = (\e_{k,\tau})^\alpha_\beta
= \epsilon^\alpha\delta^\alpha_\beta.
\end{align*}
Finally, it follows from the Poincar\'{e} system \eqref{eqn_poincare} that
\begin{equation*}
\gzero(\p_i,\p_\alpha) = \gzero(\tth\cdot A(\p_i), A^\alpha)
= (\tth(\p_i))^\top (A^\top \gzero A)^\alpha = (\gzero)^\alpha_i \equiv 0,
\end{equation*}
since $\gzero$ is a diagonal matrix. The orthonormality of $\{\pb_a\}$ now follows.

\medskip
{\bf 4(b)}. Next, we identify the normal bundle induced by $f$,
written as $TM^\perp:={f^*T\widetilde{M}}/{TM}$ (Convention \ref{convention on iota}),
with the prescribed vector bundle $E$.
For this purpose, we define the following {\em identification map} on a trivialized chart $U\subset M$:
\begin{align*}
\mathcal{I}: (U	\subset M) \times (F\cong \R^k) &\longrightarrow \widetilde{M},\\
 (x, \xi)&\longmapsto f(x) + \sum_\beta \xi_\beta \widetilde{\p}_\beta.
\end{align*}

Indeed, $\dd\mathcal{I}:TU\oplus TF \rightarrow T\widetilde{M}$ coincides with $\dd f+\nor A$;
equivalently, one can write $\dd\mathcal{I}(\p_a) = \widetilde{\p}_a$ for each $a\in\{1,\ldots,n+k\}$.
In particular, $\dd\mathcal{I}(TU) \subset TM$ and $\dd\mathcal{I}(F)\subset TM^\perp$,
which indicates that the identification map $\mathcal{I}$ preserves the horizontal and vertical subspaces of the vector
bundles $TM \oplus E$ and $T\widetilde{M}$.
Moreover, as $f$ is an immersion (justified in Step $3$ above), we deduce that $\mathcal{I}$ is a diffeomorphism,
by shrinking chart $U$ if necessary. Thus, we have obtained
an identification of $E$ with $TM^\perp$ in the trivialized local charts.

In addition, by the construction of the moving frame $\{\widetilde{\p}_a\}$ on $T\widetilde{M}$ in Eq. \eqref{eqn define the barred frame on M-bar},
we have
\begin{align*}
&\mathcal{I}^*\gzero (\sum_i U^i\p_i +   \sum_\alpha U^\alpha\p_\alpha, \sum_j V^j\p_j+ \sum_\beta V^\beta \p_\beta) \\
&= \gzero(\sum_i U^i \dd f(\p_i), \sum_j V^j \dd f(\p_j))
   + \gzero(\sum_\alpha U^\alpha \widetilde{\p}_\alpha, \sum_\beta V^\beta \widetilde{\p}_\beta)\\
&= \sum_i\sum_j U^i V^j \gzero (\widetilde{\p}_i, \widetilde{\p}_j) + \sum_\alpha\sum_\beta U^\alpha V^\beta \gzero (\widetilde{\p}_\alpha, \widetilde{\p}_\beta)\\
&= g(U|_{TM},  V|_{TM}) + g^E(U|_E + V|_E)
\end{align*}
for any $U, V \in \Gamma(T\widetilde{M})$.
It follows that $\mathcal{I}$ is an isometry between $TU \oplus E$ and $T\widetilde{M}$:
\begin{equation*}
\mathcal{I}^*\gzero = g\oplus  g^E
\end{equation*}
as the block direct sum of matrices. Thus, $f$ is a local isometric immersion.

\medskip
{\bf 4(c)}. Now, we identify the second fundamental form and the normal connection induced
by $f$ with $\two$ and $\na^E$, respectively.
This is done via Cartan's formalism for the isometric immersion $f$.

Let $f: (V\subset M,g)\rightarrow (\R^{n+k}, \gzero)$ be the isometric immersion as above.
We write $\tilde{\theta}=(\tilde{\theta}^1, \ldots,{\tilde{\theta}^{n+k}})\in
\Omega^1(\R^{n+k})\cong C^\infty(\widetilde{M}; T^*\widetilde{M}\otimes T^*\widetilde{M})$ as the co-frame of $\{\widetilde{\p}_a\}_1^{n+k}$.
Recall from \S \ref{subsec:Cartan}
that the GCR system for $f$ are equivalent to the second structural system
with respect to $\widetilde{\p}$ or $\tilde{\theta}$.
In particular, the corresponding connection $1$-form on $\widetilde{M}$ for the Levi-Civita connection is
\begin{equation}\label{schematic rep of bar-W}
\widetilde{\W} = \begin{bmatrix}
\tilde\omega^i_j & \tilde\omega^\alpha_i\\[2mm]
\tilde\omega_\alpha^i & \tilde\omega^\beta_\alpha
\end{bmatrix}
= \begin{bmatrix}
\tilde\theta^j(\ta \widetilde{\na}_\bullet \widetilde\p_i)& \widetilde{S}_{\widetilde{\partial}_\alpha}\widetilde\p_i\\[2mm]
-\widetilde{S}_{\widetilde{\partial}_\alpha}{\widetilde\p_i} & \tilde{\theta}^\beta (\na^\perp_\bullet \widetilde\p_\alpha)
\end{bmatrix};
\end{equation}
see Eq. \eqref{schematic rep for W}.
It satisfies
\begin{align*}
\widetilde{\W}=\{\tilde{\omega}^a_b\}\in &\Omega^1(\littleonk)\\
&=C^\infty (\widetilde{M};T^*\widetilde{M} \otimes \littleonk),
\end{align*}
where $\widetilde{S}$ is the shape operator associated to $f$,
and $\na^\perp$ is the projection of $\widetilde{\na}$ onto the normal bundle $TM^\perp$.
Also, by the torsion-free condition of $\widetilde{\na}$,
the first structural system \eqref{eqn_first structure eq} holds:
\begin{equation}\label{zzz}
\dd\tilde{\theta} = \tilde{\theta} \wedge \widetilde{\W}.
\end{equation}
Therefore, by comparing the coordinate-wise representations of  $\W$ and $\widetilde{\W}$,
{\it i.e.}, Eqs. \eqref{schematic rep for W} and \eqref{schematic rep of bar-W},
in order to identify $(\widetilde{S}, \na^\perp)$ with $(S,\na^E)$,
it suffices to establish
\begin{equation}\label{I star W bar = W}
\mathcal{I}^*\widetilde{\W} = \W.
\end{equation}

Indeed, we pullback Eq. \eqref{zzz} under $\mathcal{I}$.
On one hand,
\begin{equation}\label{revise, 4}
\mathcal{I}^*(\dd\tilde{\theta}) = \dd(\mathcal{I}^* \tilde{\theta})
= \dd(f^* \tilde{\theta}) = \dd\tth = \tth \wedge \W,
\end{equation}
where we have used the commutativity of pullback and exterior differential,
so that $\mathcal{I}$ respects the orthogonal splitting of $T\widetilde{M}$ and $TM \oplus E$,
the duality of $\dd f(\p_i)=\widetilde{\p}_i$,
as well as the first structural system on $TM\oplus E$. On the other hand,
\begin{equation}\label{revise, 5}
\mathcal{I}^*(\tilde{\theta}\wedge \widetilde{\W})
= \mathcal{I}^* \tilde{\theta} \wedge \mathcal{I}^*\widetilde{\W} = \tth \wedge \mathcal{I}^*\widetilde{\W},
\end{equation}
owing to the distributivity of the pullback operation with respect to the wedge product,
so that $\mathcal{I}^*\tilde{\theta}=f^*\tilde{\theta}=\underaccent{\widetilde}{\Theta}$
as above.
Eq. \eqref{I star W bar = W} follows directly from Eqs. \eqref{revise, 4}--\eqref{revise, 5}.

\medskip
{\bf 4(d)}. Finally, we prove the uniqueness of local isometric immersions up to rigid motions
of the semi-Euclidean space. It is a direct consequence of the arguments in Step $3$.
Indeed, if ${f'}: (V,g)\rightarrow (\widetilde{M},\gzero)$ is another isometric immersion on $V$
(a trivialized local chart) with ${f'}({q}')$ given,
then, for any local frame $\{{\p}'_a\}$,
we can take a rigid motion that transforms both $q$ to ${q}'$ and $\{{\widetilde{\p}}_a\}$ to $\{\widetilde{\p}'_a\}$;
that is,  a translation composed with an element of $\bigonk$.
Then the argument follows from the uniqueness of solutions of the Pfaff system (which is based in turn on the uniqueness
of the maximal integral submanifold found by the Frobenius theorem),
as well as the uniqueness of solutions of the Poincar\'{e} system up to an additive constant.

We can now conclude the realization theorem in the $C^\infty$ case from Steps 4(a)--4(d).

\smallskip
\section{\,\bf Proof of Theorem \ref{thm:3.5}}
In this appendix, we prove Theorem \ref{thm:3.5} (the generalized quadratic theorem on LCA groups) for the theory of compensated compactness,
a further extension of the classical quadratic theorem  in  \cite{Murat,tartar}.
First, we point out that the underlying strategy for the general case is similar to that
in \cite{Murat,tartar}, in which separate estimates are derived in the Fourier space $\hg$ for the {\em low-frequency region} ({\it i.e.}, in
a compact set $\Xi$ around $0$) and the {\em high-frequency region} ({\it i.e.}, in the non-compact set $\hg \setminus \Xi$).
Assumptions (i)--(iii) are required only for controlling the high-frequency region.
Notice that the high-frequency region always exists unless $\hg$ is compact,
which is equivalent to the condition that $G$ is discrete,
for which Theorem \ref{thm:3.5} trivially holds.

\medskip
\noindent
{\bf Proof of Theorem {\rm \ref{thm:3.5}}}.
By substituting $u_\varepsilon$ with $u_\varepsilon - u$ as in the proof
of Theorem \ref{thm: generalized quadratic theorem on manifolds},
it suffices to assume that $u\equiv 0$.
We divide the proof into five steps.

\medskip
\noindent
{\bf 1.} Since $u_\varepsilon \in L^2_c(G; \C^J)$ implies $u_\varepsilon \in L^1(G; \C^J)$, by the {\em Riemann--Lebesgue lemma} on LCA groups,
we can find a compact set $\Xi \Subset \hat{G}$ such that $|\hat{u}_\varepsilon(\xi)| \leq \alpha$ for $\xi \in \hg \setminus \Xi$,
for each $\alpha >0$ ({\it cf}. Tao \cite{Tao}).
In particular, $\sup{\{|\hat{u}_\varepsilon(\xi)|\,:\,\xi \in \hg\}} \leq M$.
On the other hand, for any $\phi \in L^2(G; \C^J)$, by the Plancherel formula, we have
\begin{align*}
\Big|\int_{G} u_\varepsilon \phi \,\dd\mu_G\Big| &= \Big|\int_{\hg} \hat{u}_\varepsilon \hat{\phi} \,\dd\mu_{\hg}\Big|,
\end{align*}
which converges to zero as $\varepsilon \rightarrow 0$ by assumption (C1).
Thus, choosing $\hat{\phi} = \overline{{\rm sgn} (\hat{u}_\varepsilon)}\,\chi_{\Xi}$
with ${\rm sgn}(z):=\frac{z}{|z|}$ for $z\neq 0$, we obtain
\begin{equation*}
\int_\Xi |\hat{u}_\varepsilon|^2 \,\dd\mu_{\hg} \leq M \int_\Xi |\hat{u}_\varepsilon| \,\dd\mu_{\hg}
= M \Big|\int_{\hg} \hat{u}_\varepsilon \hat{\phi} \,\dd\mu_{\hg} \Big| \longrightarrow 0.
\end{equation*}
Therefore, for the quadratic polynomial $Q$, we deduce
\begin{equation}\label{low frequency: LCA group}
\int_{\Xi} |Q \circ \hat{u}_\varepsilon| \,\dd\mu_{\hg} \longrightarrow 0.
\end{equation}
For the subsequent development,
notice that there is a freedom of enlarging $\Xi$:
It can be chosen as any large enough (with respect to $\mu_\hg$) compact subset of $\hg$ containing $0$.

\medskip
\noindent
{\bf 2.}
In this step, we establish the following claim:

\smallskip
\noindent
{\em Claim}: {\it Given any $\delta >0$ and any compact subset $\mathcal{K} \Subset \hg$ such that
$0 \notin {\mathcal{K}}$, there exists a constant $C_{\delta,\mathcal{K}} \in (0,\infty)$ so that,
for any $\lambda \in \C^J$ and $\eta \in \mathcal{K}$,
\begin{equation}\label{claim in LCA group}
{\rm Re} \{Q(\lambda)\} \geq -\delta |\lambda|^2 - C_{\delta,\K}|m(\eta)(\lambda)|^2,
\end{equation}
provided that ${\rm Re} (Q) \geq 0$ on $\Lambda_T$.
Meanwhile, under the same conditions for $\delta$ and $\K$,
\begin{equation}\label{claim in LCA group, imaginary}
{\rm Im} \{Q(\lambda)\}\geq -\delta |\lambda|^2 - C_{\delta,\K}|m(\eta)(\lambda)|^2
\end{equation}
when ${\rm Im} (Q) \geq 0$ on $\Lambda_T$.
Notice that such a compact subset $\K$ exists, since $\hg$ has locally compact topology.
}

\smallskip	
Indeed, observe that the claim holds for $\lambda = 0$.
For $\lambda \neq 0$, we prove by contradiction.
If the statement were false, there would exist $\delta_0 > 0$ such that,
for each $n \in \mathbb{N}$, there exist $\lambda_n \in \C^J$ and $\eta_n \in  \mathcal{K}$ so that
\begin{equation}\label{contradiction in LCA group}
{\rm Re} \{Q(\lambda_n)\} < -\delta_0 |\lambda_n|^2 - n |m(\eta_n) (\lambda_n)|^2.
\end{equation}
Notice that this inequality is $2$-homogeneous in $\lambda_n$;
in particular, it is invariant under the scaling: $\lambda_n \mapsto c{\lambda}_n$ for any $c \in \C\setminus \{0\}$.
Thus, without loss of generality, we may require $|\lambda_n| = 1$ for all $n$,
so that $\{\lambda_n\}$ converges to some $\lambda_\infty \in \C^J$ of norm $1$, after passing to a subsequence.

In this case, $|{\rm Re} (Q(\lambda_n))|$ is bounded uniformly in $n$ (say, by $C_0$) so that
\begin{equation*}
n | m(\eta_n) (\lambda_n)|^2 \leq C_0 - \delta_0.
\end{equation*}
This forces $|m(\eta_\infty)(\lambda_\infty)| =0$, where $\eta_\infty \in\K$ is a limit of $\{\eta_n\}$,
after passing to a further subsequence if necessary.
Indeed, the subsequential convergence is guaranteed by the fact that $\hg$ is Hausdorff,
which is a part of the definition of LCA groups.
The assumptions on $\K$ ensure that $\eta_\infty \neq 0$.
Thus, by the definition of the cone in Eq. \eqref{def cone of T, LCA group}, $\lambda_\infty\in\Lambda_\T$.
However, this implies
\begin{equation*}
 {\rm Re} \{Q(\lambda_\infty)\} \leq -\delta_0,
\end{equation*}
which contradicts the assumption that ${\rm Re} (Q) \geq 0$ on $\Lambda_\T$.
Thus, the {\em claim} is proved for ${\rm Re} (Q)$.
The arguments for ${\rm Im}(Q)$ are exactly the same, hence are omitted here.

\medskip
\noindent
{\bf 3.} Now, employing the {\it claim} in Step 3, we prove the following statement:
{\it Whenever ${\rm Re}(Q) \geq 0$ on $\Lambda_\T$,
\begin{equation}\label{liminf inequality for LCA group}
\liminf_{\varepsilon \rightarrow 0} \int_{\hg \setminus \Xi} {\rm Re} (Q \circ \hat{u}_\varepsilon) \,\dd\mu_{\hg} \geq 0.
\end{equation}
Similarly, for ${\rm Im}(Q)\geq 0$ on $\Lambda_\T$,
\begin{equation}\label{liminf inequality for LCA group--imaginary part}
\liminf_{\varepsilon \rightarrow 0} \int_{\hg \setminus \Xi} {\rm Im} (Q \circ \hat{u}_\varepsilon) \,\dd\mu_{\hg} \geq 0.
\end{equation}
}

To prove this statement, we invoke assumption (C2) on the Fourier multiplier.
As $\{[\Phi^*m]\hat{u}_\varepsilon\}$ is pre-compact in $L^2(\hg; \C^J)$
and $\{\hat{u}_\varepsilon\}$ converges
to zero weakly in $L^2$ (by the Plancherel formula),
we have
\begin{equation*}
\int_{\hg \setminus \Xi}\big|m(\Phi(\xi))\hat{u}_\varepsilon(\xi)\big|^2\,\dd\mu_\hg(\xi)\longrightarrow 0.
\end{equation*}
Take $\eta = \Phi(\xi) \in \K$ and $\lambda = \hat{u}_\varepsilon(\xi) \in \C^J$ in Eq. \eqref{claim in LCA group}
in Step $2$. It shows that, for each $\delta > 0$, there exists $0< C_{\delta, \K} <\infty$ such that
\begin{equation*}
{\rm Re}(Q\circ \hat{u}_\varepsilon(\xi)) > -\delta |\hat{u}_\varepsilon(\xi)|^2 - C_{\delta,\K} \big|m(\Phi(\xi)) \hat{u}_\varepsilon(\xi)\big|^2
\qquad \text{ for } \xi \in \hg\setminus\Xi.
\end{equation*}
Then,  integrating over $\hg\setminus\Xi$ and sending $\varepsilon\rightarrow 0$, we have
\begin{equation*}
\liminf_{\varepsilon\rightarrow 0}\int_{\hg \setminus \Xi} {\rm Re}(Q \circ \hat{u}_\varepsilon)\,\dd\mu_{\hg}
\geq -\delta  \sup_{\varepsilon \geq 0}\|\hat{u}_\varepsilon\|^2_{L^2(\hg \setminus \Xi)} \geq -\delta M
\end{equation*}
for a universal constant  $M < \infty$, where we have used the precompactness of $\{\hat{u}_\varepsilon\}$ in $L^2(\hg ; \C^J)$,
which is implied by assumption (C1) and the Plancherel formula.
As $\delta >0$ is arbitrary, Eq. \eqref{liminf inequality for LCA group} is proved.
The proof for the imaginary part, {\it i.e.}, Eq. \eqref{liminf inequality for LCA group--imaginary part},
holds analogously.

\medskip
\noindent
{\bf 4.}
To conclude the theorem, note that, by changing $Q \mapsto -Q$ in Eq. \eqref{liminf inequality for LCA group},
the following inequality holds:
\begin{equation}\label{B.9a}
\limsup_{\varepsilon \rightarrow 0} \int_{\hg \setminus \Xi} {\rm Re} (Q \circ \hat{u}_\varepsilon) \,\dd\mu_{\hg}
\leq 0 \qquad \text{ for }  {\rm Re}(Q) \leq 0  \text{ on } \Lambda_\T.
\end{equation}
By assumption (C3), {\it i.e.}, ${\rm Re}(Q) = 0$ on $\Lambda_\T$,
inequality \eqref{B.9a} together with \eqref{liminf inequality for LCA group}
verifies the assertion outside a compact set $\Xi$,
{\it i.e.}, $\lim_{\varepsilon\rightarrow 0} \int_{\hg \setminus \Xi} {\rm Re} (Q \circ \hat{u}_\varepsilon)\,\dd\mu_{\hg} = 0$.
Moreover, in Step $1$, the same result on $\Xi$ has been established in Eq. \eqref{low frequency: LCA group}.
Thus, in view of the Plancherel formula, we have
\begin{equation*}
\lim_{\varepsilon \rightarrow 0} \int_{G}{\rm Re} (Q \circ u_\varepsilon)\,\dd\mu_G = 0.
\end{equation*}
As in Steps 2--3, the analogous statement for ${\rm Im}(Q \circ u_\varepsilon)$ can be established similarly.
This completes the proof.
\end{appendices}

\begin{acknowledgements}
The authors would like to thank Professors John Ball, Lawrence Craig Evans, Marshall Slemrod,
and Dehua Wang
for helpful discussions. This paper is finalized during Siran Li's stay as a CRM--ISM postdoctoral fellow
at the Centre de Recherches Math\'{e}matiques, Universit\'{e} de Montr\'{e}al and Institut des Sciences Math\'{e}matiques;
Siran Li would like to thank these institutions for their hospitality.
\end{acknowledgements}

 \section*{Conflict of interest}

 The authors declare that they have no conflict of interest.


\end{document}